\documentclass[reqno]{article}
\usepackage{authblk}
\pdfoutput=1
\usepackage{amsmath}
\usepackage{amsfonts,amsthm,amssymb}
\usepackage{mathabx}
\usepackage{bm}
\usepackage{amscd}
\usepackage[latin2]{inputenc}
\usepackage{t1enc}
\usepackage[mathscr]{eucal}
\usepackage{indentfirst}
\usepackage{graphicx}
\usepackage{graphics}
\usepackage{pict2e}
\usepackage{epic}
\numberwithin{equation}{section}
\usepackage[margin=3cm]{geometry}
\usepackage{epstopdf} 
\usepackage[colorlinks,linkcolor=blue]{hyperref}
\usepackage[capitalise,noabbrev]{cleveref}
\usepackage{todonotes}
\usepackage{verbatim}

\crefformat{equation}{(#2#1#3)}
\crefmultiformat{equation}{(#2#1#3)}{ and~(#2#1#3)}{, (#2#1#3)}{ and~(#2#1#3)}
\crefrangeformat{equation}{(#3#1#4) to~(#5#2#6)}

\usepackage[normalem]{ulem}

\allowdisplaybreaks

\theoremstyle{plain}
\newtheorem{Thm}{Theorem}[section]
\newtheorem*{Thm*}{Theorem}
\newtheorem{Lem}[Thm]{Lemma}
\newtheorem{Cor}[Thm]{Corollary}
\newtheorem{Prop}[Thm]{Proposition}

\theoremstyle{definition}

\newtheorem{Rem}[Thm]{Remark}
\newtheorem{?}[Thm]{Problem}

\newcommand{\pii}{\partial_i}
\newcommand{\pij}{\partial_{ij}}

\newcommand{\p}{\partial}
\newcommand{\pt}{\partial_t}
\newcommand{\px}{\partial_1}
\newcommand{\R}{\mathbb{R}}
\newcommand{\T}{\theta}
\newcommand{\e}{\varepsilon}
\newcommand{\E}{\mathcal{E}}

\newcommand{\D}{\mathbf{D}}
\newcommand{\Do}{\mathbf{D}_0}
\newcommand{\Dn}{\mathbf{D}_{\neq}}
\newcommand{\Et}{\tilde{\mathcal{E}}}

\newcommand{\Eb}{\bar{\mathcal{E}}}

\newcommand{\Tt}{\tilde{\theta}}
\newcommand{\Tb}{\bar{\theta}}

\newcommand{\Tm}{\mathring{\theta}}
\newcommand{\Ta}{\acute{\theta}}
\newcommand{\phia}{\acute{\phi}}
\newcommand{\psia}{\acute{\psi}}
\newcommand{\zetaa}{\acute{\zeta}}

\newcommand{\mr}{\mathring}

\newcommand{\uf}{u}

\newcommand{\Torus}{\mathbb{T}}

\newcommand{\dv}{\text{div}}

\newcommand{\rhob}{\bar{\rho}}

\newcommand{\rhot}{\tilde{\rho}}

\newcommand{\rhom}{\mathring{\rho}}
\newcommand{\rhoa}{\acute{\rho}}
\newcommand{\ub}{\bar{u}}

\newcommand{\ut}{\tilde{u}}

\newcommand{\um}{\mathring{u}}
\newcommand{\ua}{\acute{u}}
\newcommand{\mb}{\bar{m}}

\newcommand{\mt}{\tilde{m}}

\newcommand{\thetat}{\tilde{\theta}}
\newcommand{\thetab}{\bar{\theta}}

\newcommand{\thetam}{\mathring{\theta}}
\newcommand{\thetaa}{\acute{\theta}}

\newcommand{\phim}{\mathring{\phi}}
\newcommand{\psim}{\mathring{\psi}}
\newcommand{\zetam}{\mathring{\zeta}}

\newcommand{\ev}{\mathbf{e}}
\newcommand{\red}{\textcolor{red}}
\newcommand{\blue}{\textcolor{blue}}

\newcommand{\green}{\textcolor{green}}

\newcommand{\norm}[1]{\left\lVert#1\right\rVert}

\begin{document}
	
	\begin{titlepage}
		\title{Vanishing viscosity limit to the planar rarefaction wave with vacuum for 3-D full compressible Navier-Stokes equations with temperature-dependent transport coefficients 
		}
		
		\author[]{Meichen Hou $ ^{1}$}
		
		\author[]{{Lingjun Liu $ ^2$}\thanks{Corresponding author.} }
		
		\author[]{Shu Wang $ ^2$}
		
		\author[]{Lingda Xu $ ^{3
		}$}
		
			\affil{\footnotesize $ ^1 $ Center for nonlinear studies, School of Mathematics, Northwest University, Xi' an 710069, P.R.China,\\
			meichenhou@nwu.edu.cn (M. Hou). }
		\affil{\footnotesize $ ^2 $ School of Mathematics, Statistics and Mechanics, 
		Beijing University of Technology, Beijing 100124, P.R.China,\\
		lingjunliu@bjut.edu.cn (L. Liu), wangshu@bjut.edu.cn (S. Wang).}
\affil{\footnotesize $ ^3 $ Department of Applied Mathematics, The Hong Kong Polytechnic University, Hong Kong, China,
\\
lingda.xu@polyu.edu.hk 
 (L. Xu). }
\date{}
\end{titlepage}
\maketitle
\begin{abstract}
In this paper, we construct a family of global-in-time solutions of the 3-D full compressible Navier-Stokes (N-S) equations with temperature-dependent transport coefficients (including viscosity and heat-conductivity), and show that at arbitrary times {and arbitrary strength} this family of solutions converges to planar rarefaction waves connected to the vacuum as the viscosity vanishes in the sense of $L^\infty(\R^3)$.  
We consider the Cauchy problem in $\R^3$ with perturbations of the infinite global norm, particularly, periodic perturbations. To deal with the infinite oscillation, we construct a suitable ansatz carrying this periodic oscillation such that the difference between the solution and the ansatz belongs to some Sobolev space and thus the energy method is feasible. The novelty of this paper is that the viscosity and heat-conductivity are temperature-dependent and degeneracies caused by vacuum. Thus the a priori assumptions and two Gagliardo-Nirenberg type inequalities are essentially used. Next, more careful energy estimates are carried out in this paper, by studying the zero and non-zero modes of the solutions, we obtain not only the convergence rate concerning the viscosity and heat conductivity coefficients but also the exponential time decay rate for the non-zero mode. 

\end{abstract}


\section{Introduction and Main Result}

We study the vanishing viscosity 
 limit of the 3-D full compressible 
  Navier-Stokes (N-S) 
 equations, which read
\begin{align}\label{NS}
\begin{cases}
	\rho_{t}+\dv (\rho  u)=0, \\
	(\rho  u)_{t}+\dv(\rho  u \otimes  u)+\nabla p=\varepsilon\dv \mathcal{T},\ \ \ \ \ \ \ \ \ \ \ \ \ \ \ \ \ \ (t,x)\in\R_{+}\times\R^3, \\
	(\rho E)_{t}+\dv(\rho E  u+p  u)=\varepsilon\dv(\kappa(\theta) \nabla \theta)+\varepsilon\dv( u \mathcal{T}),
\end{cases}
\end{align}
where $t \geq 0$ is the time variable and $ x
=\left( {x_1}
, x'\right)=\left(x_{1}, x_{2}, x_{3}\right) \in \mathbb{R}^{3}$ is the spatial variable. The functions $\rho,\ p,\ \theta,\ e:\ \R_{+}\times\R^3\rightarrow\R_{+}$ are all related to the space $x$ and time $t$, represent the fluid density, pressure,  absolute temperature, and internal energy, respectively. $ u=$ $\left(u_{1}, u_{2}, u_{3}\right)^{t}:\ \R_{+}\times\R^3\rightarrow\R^3$ is the fluid velocity, and $E:=e+\frac{1}{2}| u|^{2}$ denotes the specific total energy.
The viscous stress tensor $\mathcal{T}$ is 
\begin{align*}
\mathcal{T}=2\mu(\theta)\mathbb{D}( u)+\lambda(\theta)\dv u\mathbb{I},
\end{align*}
where $\mathbb{D}( u):=\frac{\nabla  u+(\nabla  u)^{t}}{2}\in\R^{3\times3}$ represents the deformation tensor, $\mathbb{I}$ is the identity matrix. 
 $\varepsilon\kappa(\theta)>0$ denotes the heat-conductivity 
 for the shear and bulk viscosity coefficients of the fluids, with $\varepsilon>0$ being positive constant, and the viscosity coefficients $\e\mu(\theta)$ and $\e\lambda(\theta)$ satisfying
$$
\mu(\theta)>0, \quad  \mu(\theta)+\lambda(\theta) \geq 0,
$$
\[
\mu(\theta)=\mu_1\theta^\alpha,\quad\lambda(\theta)=\lambda_1\theta^{\alpha},\quad\kappa(\theta)=\kappa_1\theta^{\alpha} ,
\]
{where $\alpha$, $\mu_1,\lambda_1
$ and $\kappa_1$ are some positive constants.} 
For the ideal polytropic gas, the pressure $p$ and the internal energy $e$ can be determined by 
\begin{align}\label{1}
p=R\rho\theta=A\rho^\gamma \exp(\frac{\gamma-1}{R}S),\ \ \ \ \ \ e=\frac{R}{\gamma-1}\theta,
\end{align}
where fluid constants $A>0$ and $R>0$ are normalized to
$A=R=\gamma-1$, the adiabatic exponent $\gamma>1$. 
The second law of thermodynamics
\begin{align}\label{2}
de=\theta dS+\frac{p}{\rho^2}d\rho,
\end{align}
where $S=S(t,x)$ denotes the entropy of the gas. 

 {The viscosity terms are vanishing as the parameter $\varepsilon\rightarrow0$,} 
N-S equations \eqref{NS} 
 converge to the 3-D inviscid Euler equations, 
\begin{align}\label{EE}
\left\{\begin{aligned}
	&\rho_{t}+\dv (\rho  u)=0, \\
	&(\rho  u)_{t}+\dv (\rho  u \otimes  u)+\nabla p=0, \\
	&{\big[\rho\big(\theta+\frac{| u|^{2}}{2}\big)\big]_{t}+\dv \left[\rho  u\left(\theta+\frac{| u|^{2}}{2}\right)+p  u\right]}=0. \\
\end{aligned}\right.
\end{align}
A planar centered rarefaction wave 
\begin{align}\label{center}
(\rho^r,  u^r, \theta^r)(t,x)=(\rho^r, u_1^r, 0, 0, \theta^r)(t,x_1)
\end{align}
 is a weak entropy solution to this hyperbolic system \cref{EE}, 
  $(\rho^r, u_1^r,\T^r)$ solves the following 1-D Riemann problem,
\begin{align}\label{EE1}
\begin{cases}
	\pt\rho+\px\left(\rho u_{1}\right)=0, &  \\
	\pt\left(\rho u_{1}\right)+\px\left(\rho u_{1}^{2}+p\right)=0, &{x_1} \in \mathbb{R},\  t>0, \\
	\pt(\rho E)+\px\left(\rho E u_{1}+p u_{1}\right)=0, &
\end{cases}
\end{align}
with the initial data:
\begin{align}\label{EE0}
\left(\rho, u_{1}, \theta\right)\left(0, x\right)=\left(\rho_{0}^{r}, u_{10}^{r}, \theta_{0}^{r}\right)\left(x\right)=\left\{\begin{aligned}
	&\left(\rho_{-}, u_{1-}, \theta_{-}\right)=(0, u_{1-}, 0), & {x_1}<0, \\
	&\left(\rho_{+}, u_{1+}, \theta_{+}\right), \ \ \  \ \ \ &{x_1}>0,
\end{aligned}\right.
\end{align}
where $\px=\p_{x_1}$ and $(\rho_{\pm},u_{1\pm},\theta_{\pm})$ are two constant states with $\rho_{+}>0, \theta_{+}>0$. 
The system \cref{EE1} has three distinct eigenvalues
\begin{align}\label{eig}
\lambda_{i}\left(\rho, u_{1}, S\right)=u_{1}+(-1)^{\frac{i+1}{2}} \sqrt{p_{\rho}
(\rho, S)} \ \ \text{for } \ i=1,3, \quad \lambda_{2}\left(\rho, u_{1}, S\right)=u_{1}.
\end{align}
The $i$-Riemann invariant $R_{i,j} \ (j=1,2)$ is given by 
\begin{align}\label{RI}
R_{i,1}=u_{1}+(-1)^{\frac{i-1}{2}} \int^{\rho} \frac{\sqrt{p_{z}
(z, S)}}{z} d z, \quad R_{i,2}=S,
\end{align}
which are 
 constants 
  along $i$-right eigenvectors. And the $i$-rarefaction wave curve in the phase space $\left(\rho, u_{1}, \theta\right)$ with $\rho>0$ and $\theta>0$ can be defined by
\begin{align}\label{rc}
R_{i}\left(\rho_{+}, u_{1+}, \theta_{+}\right):=\left\{\left(\rho, u_{1}, \theta\right) | \px\lambda_{i }>0, R_{i,j}\left(\rho, u_{1}, S\right)=R_{i,j}\left(\rho_{+}, u_{1+}, S_{+}\right), j=1,2\right\}.
\end{align}

Singularity is an important content in the study of wave phenomena to compressible Euler equations, including discontinuous solutions such as shock wave and contact discontinuity, 
continuous but not smooth solutions such as rarefaction wave, 
 degenerate solutions such as vacuum. In particular, 
 these strange phenomena often occur together (such as fluids near supersonic aircraft), which makes the study more difficult.
 As pointed out by Liu-Smoller \cite{LS}, only the rarefaction wave can be connected to the vacuum states.  In this paper, we study the vanishing viscosity 
  limit to the planar rarefaction wave for the multi-dimensional full compressible 
  N-S equations \eqref{NS} with vacuum. 

 For definiteness, we consider 3-rarefaction wave $ (\rho^r,u_1^r,\theta^r)(\frac{x_1}{t}) $ 
 connecting the vacuum state $(0,u_{1-},0)$ to $(\rho_+, u_{1+}, \theta_+)$, the velocity satisfies $u_{1-}=R_{3,1}\left(\rho_+, u_{1 +}, \theta_+\right)$. 
 $ (\rho^r,u_1^r,\theta^r)(\frac{x_1}{t}) $ is a self-similar solution of the 1-d Euler equations \cref{EE1} 
  and solved explicitly through
\begin{align}\label{nrw2}
\left\{\begin{aligned}
	&w^{r}\left(\frac{x_1}{t}\right)=\lambda_{3}\left(\rho^{r}, u_{1}^{r}, \theta^{r}\right)\left(\frac{x_1}{t}\right),\\
	&R_{3,1}\left(\rho^{r}, u_{1}^{r}, \theta^{r}\right)\left(\frac{x_1}{t}\right)=R_{3,1}\left(\rho_+, u_{1 +}, \theta_+\right)=R_{3,1}\left(0, u_{1 -}, 0\right), \\ 
	&R_{3,2}=S^r=S_+:=-(\gamma-1)\log \rho_++\log\theta_+,
\end{aligned}\right.
\end{align}
where 
\begin{equation}\label{nrw3}
w^{r}
\left(\frac{x_1}{t}\right)=
\begin{cases}
	\rho^r(\frac{x_1}{t})\equiv 0,  &\frac{x_1}{t}<\lambda_{3}(0, u_{1-}, 0)=u_{1-},\\  
	\frac{x_1}{t}, &u_-\leq \frac{x_1}{t} \leq \lambda_{3}(\rho_+, u_{1+}, \theta_+), \\
	\lambda_{3}(\rho_+, u_{1+}, \theta_+), &\frac{x_1}{t}>\lambda_{3}(\rho_+, u_{1+}, \theta_+).
\end{cases}
\end{equation}
We define the momentum $m^r$ and the total internal energy $n^r$ of a 3-rarefaction wave by 
\begin{align}\label{ss2}
m^r(\frac{x_1}{t}):=
\begin{cases}
\rho^r(\frac{x_1}{t})u^r(\frac{x_1}{t}), \quad \ \ \ & \text{if} \ \  \rho^r>0,\\
0, \quad \ \ \ & \text{if}\ \  \rho^r=0,
\end{cases}
\end{align}
and 
\begin{align}\label{ss1}
n^r(\frac{x_1}{t}):=
\begin{cases}
\rho^r(\frac{x_1}{t})\theta^r(\frac{x_1}{t}), \quad \ \ \ & \text{if} \ \  \rho^r>0,\\
0, \quad \ \ \ & \text{if}\ \  \rho^r=0.
\end{cases}
\end{align}
In this paper, we construct a sequence of global-in-time solutions $(\rho^\e,m^\e=\rho^\e u^\e, n^\e=\rho^\e\theta^\e)(t,x
)$ to the N-S equations \eqref{NS}, which converge to the 3-rarefaction wave $(\rho^r,m^r,n^r)(\frac{x_1}{t})$ as $\e$ tends to zero. Since the initial data \eqref{init0} is dependent on the viscosity for the compressible N-S equations, the effects of initial layers will be ignored.

Many efforts have been made on the vanishing viscosity 
 limit of the compressible 
fluid with basic wave patterns without vacuum in $\R$, such as \cite{BianchiniBressan2005,GoodmanXin1992,YuSH1999} for the hyperbolic conservation laws with artificial viscosity, 
\cite{ChenPerepelitsa2010,HoffLiu1989, Xin1993} for the isentropic compressible N-S equations (where the conservation of energy in \eqref{NS} is neglected), 
\cite{HuangWangYang2012,HuangWangWangYang2013,JiangSNiSun2006,XinZeng2010} for the full compressible N-S equations. For multi-dimensional case, by introducing hyperbolic waves, Li-Wang-Wang \cite{LWW2020, LWW2022} studied the vanishing viscosity limit to the planar rarefaction wave for 2-D isentropic N-S equations in $[0,T]\times\R\times\mathbb{T}$ and 3-D N-S-Fourier equations in $[0,T]\times\R\times\mathbb{T}^2$ ($\mathbb{T}$ is unit flat torus, $T<\infty$), respectively, with constant viscosity. 
Meanwhile, the vanishing viscosity limit of 2-D full compressible N-S equations was obtained by Li-Li \cite{LiLi2020JDE}.

 Since the vacuum states may occur in the Riemann solution instantaneously as $t>0$ even if the initial Riemann data are non-vacuum states, the vacuum states are important and make the physical system degenerate, which causes some essential analytical difficulties. 
There are some mathematical results about  
the vanishing viscosity limit to the rarefaction wave with the vacuum. 
Huang-Li-Wang \cite{HLW} considered the zero dissipation 
 limit of 1-D isentropic 
 N-S equations with constant viscosity to the rarefaction wave with one-side vacuum state to the corresponding compressible Euler equations. 
Then Li-Wang-Wang \cite{LWW} extended to 
 the 1-D full compressible N-S equations with temperature-dependent viscosity and heat conduction coefficients.  And the large time 
  behavior toward rarefaction wave with vacuum in $\R$ can refer to \cite{JWX,P}. 
Recently, Chen-Chen-Zhu \cite{CCZ2022} focused on the lower power-law case and finite mass to show the vanishing viscosity limit of the 3-D barotropic compressible N-S equations with degenerate viscosity and far-field vacuum. In addition, by similar construction of hyperbolic wave as \cite{LWW2020,LWW2022}, Bian-Wang-Xie \cite{BWX2023} investigated the vanishing viscosity limit to planar rarefaction wave with vacuum for the compressible isentropic N-S equations in $[0,T]\times\R\times\mathbb{T}^2$, with constant viscosity.

For compressible viscous flow away from a vacuum, the local existence and uniqueness of classical solution can be traced back to \cite{Nash1962, Serrin1959}. If the initial density does not have a positive lower bound, Cho-CHoe-Kim \cite{ChoCHoeKim2004} established the local well-posedness of strong solutions with vacuum in $\R^3$ by imposing initially a compatibility condition. Huang-Li-Xin \cite{HuangLiXin2012} extended this solution to be a global one with small energy for 3-D barotropic flow by the uniform estimate of the upper bound of the density. However, for degenerate viscous flow, the arguments above-mentioned cannot work. There are some significant results on weak solutions 
by employing the B-D entropy, which introduced in Bresch-Desjardins \cite{BreschDesjardins2003,BreschDesjardins2004}, such as \cite{BreschDesjardinsLin2003,LiXinarXiv, MelletVasseur2007, VasseurYu2016}. 
Then the 3-D global well-posedness of classical solutions in some homogeneous Sobolev spaces was established by Xin-Zhu \cite{XinZhu2021} 
 under initial smallness assumptions. In fact, the global well-posedness of weak solutions to many fundamental equations remains open for other cases involving vacuum.
In the present paper, we study the multi-dimensional viscosity vanishing problem in global time, 
and the setting of Riemann initial data has an initial discontinuity and vacuum states, 
which propagate with the development of solution.

It's worth noting that the viscosity vanishing limit to planar rarefaction wave with space-periodic perturbations is also full of challenge and importance in the theories of conservation laws, in which some resonant phenomena
may happen for the non-isentropic Euler equations, refer to \cite{MajdaRosales1984}. Some important progress has been made in 
 the stability of wave patterns with space-periodic perturbations, such as 
\cite{GlimmLax1970,Lax1957,XinYuanYuan2019,XinYuanYuan2021,YuanYuan2019} for 1-D conservation laws. Different from the perturbations in the previous works \cite{LW2018,LWW2018}, Huang-Yuan \cite{HY2020} studied the stability of planar rarefaction waves under multi-dimensional periodic perturbations for the 
 scalar viscous conservation laws and obtained the time-decay rates. Then Huang-Xu-Yuan \cite{HuangXuYuan2022} continues to show the nonlinear stability for planar rarefaction wave of 3-D isentropic N-S equations. 
However, as far as we know, there are no any results on the viscosity vanishing limit to the planar rarefaction wave with 
multi-dimensional periodic perturbations and even in the 1-D case. 


We consider the Cauchy problem in $\R^3$, where the initial perturbations 
  are periodic and have infinite oscillations at the far field and are not integrable along any direction of space. 
  And we obtain global-in-time solutions that converge to the planar rarefaction wave with vacuum at arbitrary times. We construct a suitable ansatz $(\tilde\rho,\tilde m,\tilde\E)$ to carry the same oscillations as those of solution $(\rho,m,\E)$ at the far field in the normal direction of the wave propagation, i.e. the difference $(\rho-\tilde\rho,m-\tilde m,\E-\tilde\E)(t,x)\rightarrow0$ as $|x_1|\rightarrow+\infty$. The viscosities are constants and the states are far away from a vacuum in \cite{HuangXuYuan2022}, but temperature-dependent viscosity and rarefaction wave with one-side vacuum are considered in this paper. 
    To overcome the degeneracies, we cut off the rarefaction wave from the vacuum with $\nu$ dependent on viscosity $\varepsilon$, and then use the viscosity to control the degeneracies. Thus, more accurate energy estimates are needed and two Gagliardo-Nirenberg (G-N) type inequalities are essentially used, which are generalized results of  \cite{HY2020}. Furthermore,  time exponential decay rates for the non-zero mode of solutions are obtained. 

Now we are ready to formulate the main result. Since the centered rarefaction wave \cref{center} is only Lipschitz continuous, we should construct a smooth profile as in \cite{MatsumuraNishihara1986}, moreover, the convergence rate with respect to $\varepsilon$ is also considered
. Inspired by \eqref{nrw2}, let $(\rho^{\bar r},u^{\bar r}_1,\T^{\bar r})(t,x_1)$ be the unique smooth solution solved by 
\begin{align}\label{nrw4}
\left\{\begin{aligned}
	&w^{\bar r}\left(t,x_1\right)=\lambda_{3}\left(\rho^{\bar r}, u_{1}^{\bar r}, \theta^{\bar r}\right)\left(t,x_1\right),\\
	&R_{3,1}\left(\rho^{\bar r}, u_{1}^{\bar r}, \theta^{\bar r}\right)\left(t,x_1\right)=R_{3,1}\left(\rho_+, u_{1 +}, \theta_+\right)=R_{3,1}\left(0, u_{1 -}, 0\right), \\ 
	&R_{3,2}=S^{\bar r}=S_+:=-(\gamma-1)\log \rho_++\log\theta_+,
\end{aligned}\right.
\end{align}
where $w^{\bar r}(t,x_1)$ is the unique smooth solution to the problem, 
\begin{align}\label{bes}
\left\{\begin{aligned}
	&\pt w^{\bar{r}}+w^{\bar{r}} \px w^{\bar{r}}=0, \\
	&w^{\bar{r}}\left(0, {x_1}
	\right)=w^{\bar{r}}_\delta({x_1})=w^{\bar{r}}(\frac{{x_1}}{\delta})=\frac{w_{+}+w_{-}}{2}+\frac{w_{+}-w_{-}}{2} \tanh\frac{{x_1}}{\delta},
\end{aligned}\right.
\end{align}
where $\delta>0$ is a small parameter with respect to $\varepsilon$ and will be given later, the 
function $\tanh x_1=\frac{e^{x_1}-e^{-x_1}}{e^{x_1}+e^{-x_1}}$. Note that this smooth rarefaction wave is time-asymptotically equivalent to the centered rarefaction wave \eqref{center} in the $L^\infty(\R)$ space, see \cite{HuangXuYuan2022,MatsumuraNishihara1986}.
However, in order to overcome the degeneracies caused by the vacuum, we need to construct the cut-off 3-rarefaction wave to obtain a new smooth planar rarefaction wave $(\bar\rho,\bar u,\bar\T)(t,x_1)$, 
 which will be described in section 2.

To study the vanishing viscosity limit to planar rarefaction wave with vacuum
, we prescribe the initial data of \eqref{NS} as 
\begin{align}\label{inih}
\begin{aligned}
(\rho_0,m_0,{\E_0}
){(x)}=(\rhob_0,\mb_0,\Eb_0){(x)}+(V_0,W_0,Z_0){(x)}, \ \ 
(V_0,W_0,Z_0){(x)}\in H^5(\Torus_\e^3), \  \Torus_\e^3=[0,\e^b]^3, %
\end{aligned}
\end{align}
for positive constant $ b\ge1$ to be determined, where 
\begin{align}\label{ini}
(\rho,m,\E)(t=0, x)=(\rho_0,m_0,\E_0)(x),
\end{align}
\begin{align}
m:=\rho u=(\rho u_1,\rho u_2,\rho u_3),\qquad \E:=\rho(\frac{R}{\gamma-1}\theta+\frac{1}{2}|u|^2),
\end{align}
 the perturbation $(V_0,W_0,Z_0)(x)=(V_0,W_{01},W_{02},W_{03},Z_0)(x)\in\R\times\R^3\times\R$ is periodic over the 3-D torus $\Torus_\e^3$, satisfying  
\begin{align}\label{cd-zeroav}
\int_{\Torus_\e^3}(V_0,W_0,Z_0){(x)}dx=0.
\end{align}
Then there exists a constant $\eta_0:=O(\e^{20})>0$ such that for
\begin{align}\label{eqe095} 
\eta:=\left\|(V_{0}, {W}_{0},Z_0)(x)\right\|_{H^{5}\left(\mathbb{T}_{\e}^{3}\right)} \leqslant \eta_{0}.
\end{align}  

\begin{Thm}\label{mt}
Denote the planar rarefaction wave solution with one-side vacuum to the 3-d compressible Euler system \cref{EE} as $\left(\rho^{r},  m^r=\rho^{r}u^{r}, n^r=\rho^{r}\theta^{r}\right)\left(\frac{x_{1}}{t}\right)$
, which is defined in 
\cref{nrw2}-\eqref{ss1}. Then there exists a constant $\varepsilon_{0}>0$ such that for any $\varepsilon \in\left(0, \varepsilon_{0}\right)$, the N-S equations \cref{NS} 
admits a family of global-in-time smooth solutions $\left(\rho^{\varepsilon},  m^{\varepsilon}=\rho^{\varepsilon}u^{\varepsilon}, n^{\varepsilon}=\rho^{\varepsilon}\theta^{\varepsilon}\right)(t, x)$
, which satisfies the following
\begin{align}
	\begin{cases}
		\left(\rho^{\varepsilon}-\rho^{r},  m^{\varepsilon}- m^{r}, n^{\varepsilon}-n^{r}\right) 
\in C
		\left(0,\infty; L^{2}(\R^{3})\right), \\
		\left(\nabla \rho^{\varepsilon}, \nabla  m^{\varepsilon}, \nabla n^{\varepsilon}\right) \in C
		\left(0,\infty ; H^{1}(\R^{3})\right), \\
		\left(\nabla^{3}  u^{\varepsilon}, \nabla^{3} \theta^{\varepsilon}\right) 
\in L^{2}\left(0,\infty; L^{2}(\R^{3})\right),
	\end{cases}
\end{align}
and for any small constant $h>0$, there exists a constant $C_{h}>0$ independent of $\varepsilon,\eta,t$, such that
\begin{align}\label{eqa}
\sup _{h \leq t}\left\|\left(\rho^{\varepsilon},  m^{\varepsilon}, n^{\varepsilon}\right)(t, x)-\left(\rho^{r},  m_1^{r},0,0, n^{r}\right)\left(\frac{x_{1}}{t}\right)\right\|_{L^{\infty}(\R^{3})}
	 \leq C_{h}\varepsilon^{ Za}|\log\varepsilon|
	,
\end{align}
where the positive constants $a$ and $Z$ are 
\begin{align}\label{eqda}
	{a=
	\frac{(\alpha+1)
	}{9\gamma(\alpha+2)
	-3},\ \ \ \ \ 
	 Z=
\frac{1}{4
(\alpha+1)\gamma}.}
\end{align}
\end{Thm} 

\begin{Thm}\label{mt1}
Assume the perturbation $(V_0,W_0,Z_0)$ in \eqref{inih} is periodic on torus $\mathbb T^3_\e$, satisfying \eqref{cd-zeroav} and \eqref{eqe095}.  The N-S equations \cref{NS}, \eqref{inih} admits a family of global-in-time smooth solutions $(\rho^\e,u^\e,\theta^\e)(x,t)$, which is periodic in $x'\in\Torus_\e^2$ and satisfies that 
\begin{align}
	\begin{cases}
	\left(\rho^{\varepsilon}-\rhot,  u^{\varepsilon}- \tilde u, \theta^{\varepsilon}-\thetat\right) 
\in C
		\left(0,+\infty; H^{3}(\Omega_\e)\right), \\
		\nabla(\rho^\e-\rhot)\in L^2\left((0,+\infty;H^2(\Omega_\e)\right),\\
		\nabla\big(u^\e-\tilde u,,\T^\e-\Tt\big)\in L^2\left(0,+\infty; H^3(\Omega_\e)\right),
	\end{cases}
\end{align}
where $\Omega_\e:=\R\times\Torus^2_\e$ and $(\rhot,\tilde u,\Tt)$ is the ansatz given by \eqref{eq-ansatz}-\eqref{eq-ansatzs1}. Furthermore, it holds that 
\begin{align}\label{time0}
\|(\rho^\e,u^\e,\T^\e)(t,x)-(\rho^r,u_1^r,0,0,\T^r)(\frac{x_1}{t})\|_{W^{1.\infty}(\R^3)}\rightarrow 0,\ \ \text{as} \ \ t\rightarrow +\infty.
\end{align}
 The non-zero mode of the solution is 
\begin{align}\label{nonzero}
(\rhoa,\acute u,\Ta)(t,x)=(\rho,u,\T)(t,x)-\int_{{\mathbb T}_\e^2}(\rho,u,\T)(t,x_1,x')dx',\ \ x\in\R^3,\ \ t>0,
\end{align}
 there exists a constant $\eta\le\eta_0:=O(\e^{20})>0$ such that 
\begin{align}\label{nonzero1}
 \|(\rhoa,\acute u,\Ta)(t,\cdot)\|_{L^\infty(\R^3)}\le O(\e^{\frac{1}{3}})
 e^{-\e^{n_2} t},\ \ \forall t>0,
 \end{align}
 where constant $n_2>1$ and $C>0$ is independent of $\varepsilon,\eta
 ,t$.
\end{Thm}  

\begin{Rem}\label{2}
 Compared to the convergence rate in the previous results  \cite{CCZ2022,LWW2020,LWW2022}, the ones in this paper are independent of the time so that the decay is feasible for infinite time. 
\end{Rem}

\begin{Rem}\label{1}
In this paper, although $(\rho,m,\E)$ oscillates in the transverse directions $x_2$ and $x_3$ when the initial perturbations carry periodic oscillations at the far field,  we still use the energy method with the aid of the G-N inequality on $\R\times\mathbb{T}_\e^2$ to achieve the vanishing viscosity limit of planar rarefaction waves in global space $\R^3$
. Moreover, we get an exponential decay rate of the non-zero mode of the solution. 
\end{Rem}


\begin{Rem}\label{4}
It is interesting to find that when $1\le b$ 
and set $a=
\frac{(\alpha+1)
}{9\gamma(\alpha+2)
-3}$, the exponential decay of the non-zero modes is independent of the values of $\gamma>1$ and $\alpha>0$. 
 The exponential decay of the non-zero modes can still be obtained when $1\ge b\ge b(\gamma,\alpha)$.
\end{Rem}

Since the well-structured and dissipative mechanism of the multi-dimensional compressible N-S equations \eqref{NS}, the purpose of this paper is to study the vanishing viscosity limit to planar rarefaction wave with vacuum by using an energy method. Let us outline the proof of main Theorems \ref{mt} and \ref{mt1}.
 Motivated by \cite{HLW,LWW}, the vacuum part of the rarefaction wave is cut off, and the cut-off parameter is very small and tends to zero as 
  the vanishing viscosity 
   converges to vacuum. Similarly to \cite{HuangXuYuan2022,LiuWangXu2023}, we construct a suitable ansatz $(\rhot,\ut,\Tt)$ to overcome the difficulties caused by the multi-dimensional 
   periodic perturbations. By the energy method with the aid of the G-N inequality, 
   we can construct a family of time-global solutions to the Cauchy problem for  
 3-D compressible N-S equations \eqref{NS}. When the viscosity tends to zero, the solutions of this family eventually converge to the corresponding planar rarefaction wave with vacuum in the sense of $L^\infty(\R^3)$. 
  In order to get the exponential decay of the non-zero mode, the solution of equations \eqref{NS} is decomposed into zero 
   and non-zero modes 
    in Fourier space for more accurate energy estimation. 

In this paper, we denote $\|\cdot\|=\|\cdot\|_{L^2
}$ and $\|\cdot\|_{\infty}=\|\cdot\|_{L^\infty
}$. For simplicity, we write $C$ as generic positive constants which are independent of time $t$ or $\tau$ and viscosity $\e$.

\section{ Preliminaries}




\subsection{Rarefaction Wave}
In order to overcome the degeneracies caused by the vacuum, we will cut off the 3-rarefaction wave $ (\rho^r,u_1^r,\theta^r)(\frac{x_1}{t}) $ with vacuum along the wave curve. For any constant $\nu>0$, there is a state $(\rho, u_1, \theta)=(\nu, u_{1\nu}, e^{\bar{S}}\nu^{\gamma-1})$ belonging to the 3-rarefaction wave curve, where $$u_{1\nu}=R_{3,1}\left(\rho_+, u_{1+}, S_+\right)+2\sqrt{\frac{\gamma}{\gamma-1}\nu^{\gamma-1}e^{S_+}},\ \ \ \bar{S}=S_+.$$ Now, the new cut-off 3-rarefaction wave $(\rho^r_{\nu}, u^r_{1\nu}, \theta^r_{\nu})(\xi)$, $(\xi=\frac{x_1}{t})$ connecting the state $(\rho_-, u_{1-}, \theta_-):=(\nu, u_{1\nu}, e^{\bar{S}}\nu^{\gamma-1})$ to the state $(\rho_+, u_{1+}, \theta_+)$ can be expressed explicitly by 
\begin{equation}\label{nrw4}
\lambda_{3}\left(\rho_{\nu}^{r}, u_{1\nu}^{r}, \theta_{\nu}^{r}\right)\left(\xi\right)=
\begin{cases}
	\lambda_3(\nu, u_{1\nu}, e^{\bar{S}}\nu^{\gamma-1}),  &\xi<\lambda_{3}(\nu, u_{1\nu}, e^{\bar{S}}\nu^{\gamma-1}),\\  
	\xi, &\lambda_3(\nu, u_{1\nu}, e^{\bar{S}}\nu^{\gamma-1})\leq \xi \leq \lambda_3(\nu, u_{1\nu}, e^{\bar{S}}\nu^{\gamma-1}), \\
	\lambda_{3}(\rho_+, u_{1+}, \theta_+), &\frac{x_1}{t}>\lambda_{3}(\rho_+, u_{1+}, \theta_+),
\end{cases}
\end{equation}
and 
\begin{equation}\label{nrw5}
R_{3,1}\left(\rho_{\nu}^{r}, u_{1\nu}^{r}, \theta_{\nu}^{r}\right)=R_{3,1}\left(\nu, u_{1\nu}, e^{\bar{S}}\nu^{\gamma-1}\right)=R_{3,1}\left(\rho_+, u_{1 +}, \theta_+\right).
\end{equation}
Then we have the following Lemma {to show that the cut-off 3-rarefaction wave $(\rho_\nu^r,u_{1\nu}^r, n^r_{\nu})$ converges to the original 3-rarefaction wave $(\rho^r, u_1^r, n^r)$ in sup-norm with the convergence rate $\nu$ as $\nu\rightarrow0$.}
\begin{Lem}\label{sr1}
There exists a constant $\nu_0\in(0, 1)$, for $\forall\nu\in(0, \nu_0]$ such that
\begin{equation}\label{nrw6}
	||(\rho_{\nu}^r, m_{1\nu}^r, n_{\nu}^r)(\frac{\cdot}{t})-(\rho^r, m_1^r,n^r)(\frac{\cdot}{t})||_{{L^\infty}}\leq C\nu, \ \ \  \ t>0,
\end{equation}
where the positive constant $C$ is independent of $\nu$, $m^r_{1\nu}=\rho^r_\nu  u^r_{1\nu}, m_1^r=\rho^ru_1^r$ and $n^r_\nu=\rho^r_\nu \theta^r_\nu, n^r=\rho^r\theta^r$ represent the momentum function and the total internal energy respectively.
\end{Lem}

Similar to \cref{nrw2}, the approximate 3-rarefaction wave $(\rho_{\nu}^{\bar{r}}, u_{1\nu}^{\bar{r}}, \theta_{\nu}^{\bar{r}})\left(t, {x_1}\right)$ of the cut-off 3-rarefaction wave $(\rho_{\nu}^r, u_{1\nu}^r, \theta_{\nu}^r)(\frac{x_1}{t})$ to the Euler system \cref{EE} can be obtained 
 by solving
\begin{align}\label{nrws}
\left\{\begin{aligned}
	&w_+=\lambda_{3}(\rho_+,u_{1+},\theta_+), \ \  w_-=\lambda_{3}(\nu, u_{1\nu}, e^{\bar{S}}\nu^{\gamma-1}),\\
	&w^{\bar{r}}\left(1+t, {x_1}\right)=\lambda_{3}\left(\rho_{\nu}^{\bar{r}}, u^{\bar{r}}_{1\nu}, \theta_{\nu}^{\bar{r}}\right)\left(t, {x_1}\right), \\
	&R_{3,1}\left(\rho_{\nu}^{\bar{r}}, u^{\bar{r}}_{1\nu}, \theta_{\nu}^{\bar{r}}\right)\left(t, {x_1}\right)=R_{3,1}\left(\rho_{+}, u_{1 +}, \theta_{+}\right)=R_{3,1}(\nu, u_{1\nu}, e^{\bar{S}}\nu^{\gamma-1}),  \\
	&R_{3,2}\left(\rho_{\nu}^{\bar{r}}, u^{\bar{r}}_{1\nu}, \theta_{\nu}^{\bar{r}}\right)\left(t, {x_1}\right)=S_{\nu}^{\bar{r}}=S_+.
\end{aligned}\right.
\end{align}
Direct calculations give that the smooth planar rarefaction wave 
\begin{equation}\label{prw}
(\rho_{\nu}^{\bar{r}},  u_{\nu}^{\bar{r}}, \theta_{\nu}^{\bar{r}})\left(t, {x_1} \right) = (\rho_{\nu}^{\bar{r}},u_{1\nu}^{\bar{r}},0,0,\theta_{\nu}^{\bar{r}})(t,{x_1})
\end{equation}
{is the solution of} 
Euler system \cref{EE}. By the remarkable result of \cite{MatsumuraNishihara1986}, we verify that the profile $(\rho_{\nu}^{\bar{r}},  u_{\nu}^{\bar{r}}, \theta_{\nu}^{\bar{r}})\left(t, x_1 \right)$ has following properties. For simplicity, we write $(\rho_{\nu}^{\bar{r}}, u_{\nu}^{\bar{r}},\theta_{\nu}^{\bar{r}})$ for short as $(\bar{\rho}, \bar{ u}, \bar{\theta})$ that satisfies 
\begin{align}\label{eq7}
\left\{\begin{aligned}
	&\pt\bar{\rho}+\px(\bar\rho \bar u_1)=0, \\
	&\bar\rho\pt \bar u+\bar\rho \bar u_1 \px\bar  u+\px\bar p\mathbb{I}_1=0, \\
	&\bar\rho\pt{\bar \theta}+\bar\rho\bar u_1\px\bar \theta+\bar p\px\bar u_{1}=0, \\
\end{aligned}\right.
\end{align}
where $\bar p=R\bar\rho\bar\theta=A\bar\rho^\gamma\exp(\frac{\gamma-1}{R}\bar S)$.

\begin{Lem}\cite{MatsumuraNishihara1986} 
\label{sr}
The approximate smooth 3-rarefaction wave $(\bar{\rho}, \bar{u}_1, \bar{\theta})$ constructed in \cref{nrw2} such that

(i) $\px\bar{u}_{1}>0, \px\bar{\rho}=\frac{1}{\sqrt{R \gamma \rho_{+}^{1-\gamma} \theta_{+}}}   \bar{\rho}^{\frac{3-\gamma}{2}} \px\bar{u}_{1}>0$, $\px\bar{\theta}=\frac{\gamma-1}{\sqrt{R \gamma}} \sqrt{\bar\theta} \px\bar{u}_{1}>0,$ and
\begin{align}\label{eq66}
	\begin{aligned}
		\px^2\bar{\rho}=\frac{1}{\sqrt{R \gamma \rho_{+}^{1-\gamma} \theta_{+}}}   \bar{\rho}^{\frac{3-\gamma}{2}}  \px^2\bar{u}_{1 }+\frac{3-\gamma}{2R \gamma \rho_{+}^{1-\gamma} \theta_{+}}   \bar{\rho}^{2-\gamma} (\px\bar u)^2,\ \ \px^2\bar{\theta}=\frac{\gamma-1}{\sqrt{R \gamma}} \sqrt{\bar\theta} \px^2\bar{u}_{1 }+\frac{\gamma-1}{2R\gamma}(\px\bar u)^2,
	\end{aligned}
\end{align}
for all
$
x_1 \in \mathbb{R} \text { and } t \geq 0.
$

(ii) For all $t \geq 0, \delta>0$ and $p \in[1,+\infty]$, the following estimate holds 
\begin{align}
\left\|\px \bar {u}_{1}\right\|_{L^{p}(\mathbb{R})} \leq C(\delta+t)^{-1+1/p}, \quad %
\left\|\px^2\bar {u}_{1}\right\|_{L^{p}(\mathbb{R})} \leq C(\delta+t)^{-1}\delta^{-1+1/p}.
\end{align}

(iii) There exists a constant $\delta_{0} \in(0,1)$ and a uniform constant $C$ such that for all $\delta \in \left(0, \delta_{0}\right]$ and $t>0$,
\begin{align}\label{eqsr}
\left\|\left(\bar\rho, \bar{u}_{1}, \bar\theta\right)(t, \cdot)-\left(\rho_{\nu}^{r}, u_{1\nu}^{r}, \theta_{\nu}^{r}\right)(t, \cdot)\right\|_{L^{\infty}(\mathbb{R})} \leq C \delta t^{-1}{[\log (1+t)+|\log \delta|]} .
\end{align}
\end{Lem} 

\subsection{Ansatz}
 Although the initial perturbations \eqref{inih} are periodic, we cannot expect that the solutions are also periodic because of 
  the non-trivial background solution. Thus, the construction of ansatz to make the perturbations in the right Sobolev space is crucial for our proof. The key of construction is to eliminate the oscillations at infinity. We use the periodic solutions to achieve this aim. Set $(\rhob_\pm,\mb_\pm,\Eb_\pm){(t,x)}$ as the solutions of N-S equations with the following periodic initial data
\begin{align}\label{initial}
(\rhob_\pm,\mb_\pm,\Eb_\pm)(0,x)=(\rho_\pm,m_\pm,\E_\pm)+(V_0,W_0,Z_0){(x)},
\end{align}
{where $\rho_\pm, {\E}_\pm>0, {m}_\pm \in \mathbb{R}^{3}$ are given constants.} 
We can deduce following Lemma.  
\begin{Lem}\label{lem-periodic-solution}
For the N-S equations \eqref{NS} with the 
periodic initial data
 {\eqref{initial}},
and $(V_0,W_0,Z_0){(x)} \in H^{k+2}\left(\mathbb{T}_\e^{3}\right)(k \geqslant 1)$ in \eqref{cd-zeroav} are periodic functions defining on $\mathbb{T}_\e^{3}=[0,\e^b]^{3}$ 
for $1\le b
$. 
Then there exists a constant $\eta_0:=O(\e^{20})>0$ such that for
\begin{align}\label{eq095} 
\eta:=\left\|(V_{0}, {W}_{0},Z_0)(x)\right\|_{H^{k+2}\left(\mathbb{T}_{\e}^{3}\right)} \leqslant \eta_{0},
\end{align}  
 the Cauchy problem \cref{NS,initial} admits a pair of unique global periodic solutions $(\rhob_\pm, \mb_\pm,\Eb_\pm) \in C\left((0,+\infty) ; H^{k+2}\left(\mathbb{T}_\e^{3}\right)\right)$, satisfying
$$
\int_{\mathbb{T}_\e^{3}}(\rhob_\pm-{\rho}_\pm, \mb_\pm-m_\pm,\Eb_\pm-{\E_\pm}
)(t,x) d x=0, \quad \forall t \geqslant 0,
$$
and
\begin{align}\label{eq096}
\|(\rhob_\pm-{\rho}_\pm, \mb_\pm-m_\pm,\Eb_\pm-\E_\pm)\|_{W^{k, \infty}\left(\mathbb{R}^{3}\right)} \lesssim \eta_0 e^{-\e^{n_k} t}, \quad \forall t>0,
\end{align}
where $n_k>0$ is a constant that depends on $k$.
\end{Lem}

\begin{proof}
The proof of Lemma \ref{lem-periodic-solution} is stated in the appendix, which is based on the energy method with the aid of the G-N inequality and Poincar
{\'e}'s inequality on the bounded domain.
\end{proof}

Next, we define some weight functions 
\begin{align}
\eta_1:=\frac{\rhob-\rho_-}{\rho_+-\rho_-},\quad\eta_2:=\frac{\mb-m_-}{m_+-m_-},\quad\eta_3:=\frac{\Eb-\E_-}{\E_+-\E_-},
\end{align}then the proper ansatz is constructed as
\begin{align}\label{eq-ansatz}
	\begin{aligned}
	\rhot&=\bar\rho+(1-\eta_{1}){\rhot}_-+\eta_{1}{\rhot}_+,\quad
	\tilde{m}=\mb+(1-\eta_{2}){\mt}_-+\eta_{2}{\mt}_+,\quad
		\Et=\bar\E+(1-\eta_3){\Et}_-+\eta_3{\Et}_+,
	\end{aligned}
\end{align}
where $\rhot_{\pm}=\bar\rho_{\pm}-\rho_{\pm}, \mt_{\pm}=\mb_{\pm}-m_{\pm}, \Et_{\pm}=\bar\E_{\pm}-\E_{\pm}.$
And correspondingly,
\begin{align}\label{eq-ansatzs1}
	\ut:=\frac{\mt}{\rhot},\qquad\thetat=\frac{\gamma-1}{R}\left(\frac{\Et}{\rhot}-\frac{|\ut|^2}{2}\right).
\end{align} 
Moreover, the ansatz $(\rhot,\ut,\Tt)(t,x)$ satisfies the following asymptotic N-S equations,
\begin{align}\label{ae}
	\begin{cases}
		\rhot_{t}+\dv (\rhot \ut):={e_0}, \\
		\rhot \ut_{t}+\rhot \ut \cdot \nabla\ut+R \nabla(\rhot \thetat)-\varepsilon\dv(2\mu(\thetat
		) \mathbb{D}( \ut)+\lambda(\thetat
		)\dv\tilde{u}\mathbb{I}):=\ev{-e_0\tilde{u}}, \\
	\rhot \thetat_{t}+\rhot\ut\cdot \nabla\thetat+\tilde{p} \dv  \ut
		-\varepsilon\dv(\kappa(\thetat
		) \nabla \thetat)-\frac{\mu( \thetat)\varepsilon}{2}\left|\nabla \ut+(\nabla \ut)^{t}\right|^{2}-\lambda(\thetat)\varepsilon(\dv  \ut)^{2}\\
		\qquad :=e_4-\tilde{u}\cdot\ev{+{e_0}\left(\frac{|\tilde{u}|^2}{2}-\tilde\theta\right)},
	\end{cases}
\end{align}
where $\mu(\thetat)=\mu_1\thetat^\alpha, \lambda(\thetat)=\lambda_1\thetat^{\alpha}, \kappa(\thetat)=\kappa_1\thetat^{\alpha}, e_0=e_0(t,x),  \ev=(e_1,e_2,e_3)^t(t,x)$
. {{In the light of the Lemma \ref{lem-periodic-solution}, it is easy to check that }}
\begin{align}\label{eq094}
\|(\rhot-\bar\rho, \mt-\bar m, \Et-\bar\E)(t, x)\|_{W^{k,\infty}(\R^3)}\le C\eta_0
e^{-\varepsilon^{n_k} t},\ \ \ \forall t\ge0.
\end{align}
The errors satisfy the following estimates. 
 \begin{Lem}\label{er} 
	Under the same assumptions of \cref{mt1}, the error terms defined in \cref{ae} satisfy that
	\begin{equation}\label{est-h}
		\begin{aligned}
			&\norm{e_0}_{W^{2,p}(\Omega_\e)} +\norm{e_4+\left[(2\mu_1+\lambda_1)\e\Tb^\alpha|\px\bar{u}|^2+\px(\e\kappa_1\Tb^{\alpha}\px\Tb)\right]}_{W^{1,p}(\Omega_\e)}  \leq C\eta_{0}
			 e^{-\e^{n_2} t},\ \forall t>0,\\
			& \norm{e_1+\px\left[(2\mu_1+\lambda_1)\Tb^{\alpha}\e\px\bar{u}_{1}\right]}_{W^{1,p}(\Omega_\e)}+\norm{(e_2,e_3)}_{W^{2,p}(\Omega_\e)} \leq C\eta_{0} 
			e^{-\e^{n_2} t}, \ \ \forall t>0,
		\end{aligned}
	\end{equation}
	where 
	 ${p\in[1,+\infty]}, n_2>1
	$ and $\Omega_\e:=\R\times\Torus^2_\e.$
\end{Lem}
The proof is similar to the Lemma 2.4 
 in \cite{HuangXuYuan2022,HuangXinXuYuan2023} and will be omitted for brevity.

\subsection{Decomposition}

This subsection will decompose the solution into principal and transversal parts. Since the solution is periodic in $x'\in\Torus_\e^2$, without loss of generality, we set $\int_{\Torus_\e^2}1dx'={\e^{2b}}
$. Then we can define the following decompositions $\mathbf{D}_0$ and $\mathbf{D}_{\neq}$,
\begin{align}\label{def-decom}
	\D_{0} f:=\mathring{f}:=\int_{\Torus_\e^2}f dx',\qquad \D_{\neq} f:=\acute{f}:=f-\mathring{f},
\end{align} 
for an arbitrary function $f$ which is integrable on $\Torus_\e^2$. With simple analysis, the following propositions of $\D_0$ and $\D_{\neq}$ hold for any function $f$ on $\Torus_\e^2$. 
\begin{Prop}\label{prop-decom}
	For the projections $\Do$ and $\Dn$ defined in \cref{def-decom}, the following holds,
	
	i) $\Do\Dn f=\Dn\Do f=0$.
	
	ii) For any non-linear function $F$, one has
	\begin{align}
		\Do F(U)-F(\Do U)=O(1)F''(\D_0U)\D_0\big((\Dn U)^2\big),
	\end{align}
	and similar results hold for $\tilde{U}$, $\bar{U}$, etc.
	
	iii) $\|f\|^2_{L^2(\Omega_\e)}=\|\Do f\|^2_{L^2(\R)}+\|\Dn f\|^2_{L^2(\Omega_\e)}.$
\end{Prop}
The proof of \cref{prop-decom} is basic and we omit it.

\begin{Lem}\label{lem6}\cite{Yuan2023SIAM,Yuanarxiv}
{For any $p\in[1,+\infty]$, it holds that 
\begin{align}
\begin{aligned}
&\|\Do f\|_{L^p(\R)}\le\|f\|_{L^p(\Omega_\e)},\ \ 
& \|\Dn f\|_{L^p(\Omega_\e)}\le\|f\|_{L^p(\Omega_\e)}+\|\Do f\|_{L^p(\R)}\le\|f\|_{L^p(\Omega_\e)}.
\end{aligned}
\end{align}}
\end{Lem}
{The proof of Lemma \ref{lem6} can be deduced directly by making use of Cauchy's inequality and Poincar
{\'e}'s inequality.}

\subsection{Gagliardo-Nirenberg type inequality}

\begin{Lem}
\label{Lem-GN}
	Let $\Omega_\Lambda:=\R\times\Torus_{\Lambda}^2$, where $\Torus_{\Lambda}:=[0,\Lambda]$ and $\Lambda$ is a positive constant. Assume that $ u(x) $ is in the $ L^q(\Omega_{{\Lambda}}) $ space with $ \nabla^m u \in L^r(\Omega_{{\Lambda}}), $ where $ 1\leq q,r\leq +\infty $ and $ m\geq 1, $ and $ u $ is periodic with respect to $ x_2 $ and $ x_3. $ Then there exists a decomposition $ u(x) = \sum\limits_{k=1}^{3} u^{(k)}(x) $ such that each $ u^{(k)} $ satisfies the $ k $-dimensional G-N inequality,
	\begin{equation}\label{G-N-type-1}
		\norm{\nabla^j u^{(k)}}_{L^p(\Omega_{{\Lambda}})} \leq C \Lambda^{-\frac{n-k}{k}(m\T_k-j)} \norm{\nabla^{m} u}_{L^{r}(\Omega_{{\Lambda}})}^{\theta_k} \norm{u}_{L^{q}(\Omega_{{\Lambda}})}^{1-\theta_k},\ \ 1\le k\le n, n=3,
	\end{equation}
	where the positive constant $C$ is independent of $\Lambda$, $ 0\leq j< m $ is any integer and $ 1\leq p \leq +\infty $ is any number, satisfying 
	\begin{align}\label{equality}
	 \frac{1}{p} = \frac{j}{k} + \Big(\frac{1}{r}-\frac{m}{k}\Big) \theta_k + \frac{1}{q}\big(1-\theta_k\big) \quad \text{with} \quad \frac{j}{m} \leq \theta_k \leq 1. 
	 \end{align}
	Moreover, it holds that 
	\begin{equation}\label{G-N-type-2}
		\norm{\nabla^j u}_{L^p(\Omega_{{\Lambda}})} \leq C \Lambda^{-\frac{n-k}{k}(m\T_k-j)}\sum_{k=1}^3 \norm{\nabla^{m} u}_{L^{r}(\Omega_{{\Lambda}})}^{\theta_k} \norm{u}_{L^{q}(\Omega_{{\Lambda}})}^{1-\theta_k}.
	\end{equation}
	\end{Lem}
\begin{proof}
Similar to 
 the Theorem 1.4 in \cite{HY2020}, we consider the region $\Omega_{\Lambda^2}:=\R\times\mathbb{T}^2_{\Lambda^2}$, where $\mathbb{T}_{\Lambda^2}=[0,\Lambda^2]$, $v$ is periodic in the $x_i$ direction for $i=2,\cdots, n, $ it holds that 
\begin{align}\label{eq0106}
\|\nabla^jv\|_{L^p(\Omega_{\Lambda^2})}\le C\sum_{k=1}^{n}\|\nabla^mv\|_{L^r(\Omega_{\Lambda^2})}^{\theta_k}\|v\|_{L^q(\Omega_{\Lambda^2})}^{1-\theta_k},\ \ v = \sum\limits_{k=1}^{n} v^{(k)},
\end{align}
where positive constant $C$ is independent of $\Lambda$.
Set $u^{(k)}(x)=
\Lambda^{\hat\alpha}v^{(k)}(\frac{y}{\Lambda})
$, $y\in\Omega_{\Lambda^2}$, $x=\frac{y}{\Lambda}\in\Omega
_\Lambda
$ and $\hat\alpha$ is dependent on $q$ and $n$ to be determined such that $\int_{\Omega_\Lambda
}u^q(x)dx=1$. We rewrite $u^{(k)}(x)$ and $v^{(k)}(\frac{y}{\Lambda})$ as $u(x)$ and $v(\frac{y}{\Lambda})$ for convenience. 
Let $$\bar\Lambda(\Omega_\Lambda
)=\inf\left\{\frac{\|\nabla^j u\|_{L^p(\Omega_\Lambda
)}}{\|\nabla^mu\|_{L^r(\Omega_\Lambda
)}^{\theta_k}}
\big|
 \|u\|_{L^q(\Omega_\Lambda)}^q=1
\right\}
,\ \ \bar\Lambda(\Omega_{\Lambda^2})=\inf\left\{\frac{\|\nabla^j v\|_{L^p(\Omega_{\Lambda^2})}}{\|\nabla^mv\|_{L^r(\Omega_{\Lambda^2})}^{\theta_k}}\big|
 \|v\|_{L^q(\Omega_{\Lambda^2})}^q=1
\right\}
.$$ 
 Here we take $\hat\alpha=-\frac{n}{q}$.  Then 
\begin{align}\label{eq0105}
\begin{aligned}
&\int_{\Omega_\Lambda
}|\nabla^ju(x)|^{p}dx=\Lambda^{(-\frac{n}{q}-j)p}\int_{\Omega_\Lambda
}|\nabla^jv(y)|^{p}dx=\Lambda^{(-\frac{n}{q}-j)p+n}\int_{\Omega_{\Lambda^2}}|\nabla^mv(y)|^{r}dy,\\
&\int_{\Omega_\Lambda
}|\nabla^mu(x)|^{r}dx=\Lambda^{(-\frac{n}{q}-m)r}\int_{\Omega_\Lambda
}|\nabla^mv(y)|^{r}dx=\Lambda^{(-\frac{n}{q}-m)r+n}\int_{\Omega_{\Lambda^2}}|\nabla^mv(y)|^{r}dy.
\end{aligned}
\end{align}
So we have the optimal constant
\begin{align}\label{eq0107}
\bar\Lambda(\Omega_{\Lambda^2}):=\frac{\|\nabla^jv\|_{L^p(\Omega_{\Lambda^2})}}{\|\nabla^mv\|_{L^r(\Omega_{\Lambda^2})}^{\theta_k}}=\Lambda^{\frac{n}{p}-\frac{n}{q}(1-\theta_k)-\frac{\theta_k}{r}n-j+m\theta_k}
\frac{\|\nabla^ju\|_{L^p(\Omega_\Lambda
)}}{\|\nabla^mu\|_{L^r(\Omega_\Lambda
)}^{\theta_k}}=\Lambda^{-\frac{n-k}{k}(m\T_k-j)}\bar\Lambda(\Omega_\Lambda
),
\end{align}
 for $1\le k\le n$.
 Combining \eqref{eq0106} with \eqref{eq0107}, we can obtain \eqref{G-N-type-1} and \eqref{G-N-type-2} by the decomposition.
\end{proof}

We list the following special cases for the sake of convenience.
\begin{Cor}
	Under the same assumptions of \cref{Lem-GN}, one has,
	\begin{align}\label{eq46i}
		\begin{aligned}
		\|u\|_{L^4(\Omega_{{\Lambda}})}\leq& C\big(\Lambda^{-\frac{1}{2}}
		\|\nabla u\|^{\frac{1}{4}}_{L^2(\Omega_{{\Lambda}})}\| u\|^{\frac{3}{4}}_{L^2(\Omega_{{\Lambda}})}+\Lambda^{-\frac{1}{4}}
		\|\nabla u\|^{\frac{1}{2}}_{L^2(\Omega_{{\Lambda}})}\| u\|^{\frac{1}{2}}_{L^2(\Omega_{{\Lambda}})}+\|\nabla u\|^{\frac{3}{4}}_{L^2(\Omega_{{\Lambda}})}\| u\|^{\frac{1}{4}}_{L^2(\Omega_{{\Lambda}})}\big),\\
		\|u\|_{L^6(\Omega_{{\Lambda}})}\leq &C\Lambda^{-\frac{2}{3}}
		\|\nabla u\|^{\frac{1}{3}}_{L^2(\Omega_{{\Lambda}})}\| u\|^{\frac{2}{3}}_{L^2(\Omega_{{\Lambda}})}+C\Lambda^{-\frac{1}{3}}
		\|\nabla u\|^{\frac{2}{3}}_{L^2(\Omega_{{\Lambda}})}\| u\|^{\frac{1}{3}}_{L^2(\Omega_{{\Lambda}})}+C\|\nabla u\|_{L^2(\Omega_{{\Lambda}})},\\
		\|u\|_{L^{\infty}(\Omega_{{\Lambda}})}\leq& C\Lambda^{-1}\|\nabla u\|^{\frac{1}{2}}_{L^2(\Omega_{{\Lambda}})}\| u\|^{\frac{1}{2}}_{L^2(\Omega_{{\Lambda}})}+C\Lambda^{-\frac{1}{2}}
		\|\nabla u\|_{L^2(\Omega_{{\Lambda}})}+C\|\nabla^2 u\|^{\frac{3}{4}}_{L^2(\Omega_{{\Lambda}})}\| u\|^{\frac{1}{4}}_{L^2(\Omega_{{\Lambda}})}.
	\end{aligned}
	\end{align}
\end{Cor}

\begin{Lem}[\cite{Nirenberg}, Theorem 1]\label{Lem-GN2}
	Set $\Torus_{\Lambda}:=[0,\Lambda]$ and $\Lambda$ is a positive constant, $\Torus_{\Lambda}^n$ is the bounded domain. Assume that $ u(x) $ is in the $ L^q(\Torus_{\Lambda}^n) $ space with $ \nabla^m u \in L^r(\Torus_{\Lambda}^n), $ where $ 1\leq q,r\leq +\infty $. Then $u(x)$ satisfies the extended G-N inequality,
	\begin{equation}\label{G-N-type-3}
	\norm{\nabla^j u}_{L^p(\Torus_{\Lambda}^n)} \leq C_1 \norm{\nabla^{m} u}_{L^{r}(\Torus_{\Lambda}^n)}^{\theta} \norm{u}_{L^{q}(\Torus_{\Lambda}^n)}^{1-\theta}+C_2\Lambda^{\frac{n}{p}-\frac{n}{q}-j}\norm{u}_{L^{q}(\Torus_{\Lambda}^n)},\ n=3, 
	\end{equation}
	where the positive constant $C_1,C_2$ is independent of $\Lambda$, $ 0\leq j< m $ is any integer and $ 1\leq p \leq +\infty $ is any number, satisfying
	\begin{align}\label{eq109}
	\frac{1}{p} = \frac{j}{n} + \Big(\frac{1}{r}-\frac{m}{n}\Big) \theta + \frac{1}{q}\big(1-\theta\big) \quad \text{with} \quad \frac{j}{m} \leq \theta \leq 1.
	\end{align}
	Unless $1< r< +\infty $ and $m-j-\frac{n}{r}$ is a nonnegative integer, in which \eqref{G-N-type-3} holds only for $\T$ satisfying $\frac{j}{m} \leq \theta < 1$.

	\begin{proof}
		
		 We consider the region $\Torus^n=[0,1]^n$, for $v=v(y), y\in \Torus^n$, it holds that
		\begin{align}\label{eqs0110}
		\norm{\nabla^j v}_{L^p(\Torus^n)} \leq C_1 \norm{\nabla^{m} v}_{L^{r}(\Torus^n)}^{\theta} \norm{v}_{L^{q}(\Torus^n)}^{1-\theta}+C_2\norm{v}_{L^{q}(\Torus^n)},\ n=3,
		\end{align}
		where the positive constants $C_1,C_2$ are independent of $\Lambda$.
		Set $
		v(y)=u(\Lambda y)$, $y\in \Torus^n$, $x=\Lambda y\in\Torus_{\Lambda}^n
		$.
	   Then
		\begin{align}\label{eq0111}
		\begin{aligned}
		&\bigg(\int_{\Torus^n}|\nabla^jv(y)|^{p}dy\bigg)^{\frac{1}{p}}=\Lambda^{(-\frac{n}{p}+j)}\bigg(\int_{\Torus_{\Lambda}^n}|\nabla^ju(x)|^{p}dx\bigg)^{\frac{1}{p}}=\Lambda^{(j-\frac{n}{p})}\norm{\nabla^j u}_{L^p(\Torus_{\Lambda}^n)},\\
		&\bigg(\int_{\Torus^n}|\nabla^mv(y)|^{r}dy\bigg)^{\frac{1}{r}}=\Lambda^{(-\frac{n}{r}+m)}\bigg(\int_{\Torus_{\Lambda}^n}|\nabla^ju(x)|^{r}dx\bigg)^{\frac{1}{r}}=\Lambda^{(m-\frac{n}{r})}\norm{\nabla^m u}_{L^r(\Torus_{\Lambda}^n)},\\
		&\bigg(\int_{\Torus^n}|v(y)|^{q}dy\bigg)^{\frac{1}{q}}=\Lambda^{-\frac{n}{q}}\bigg(\int_{\Torus_{\Lambda}^n}|u(x)|^{q}dx\bigg)^{\frac{1}{q}}=\Lambda^{-\frac{n}{q}}\norm{ u}_{L^q(\Torus_{\Lambda}^n)}.\\
		\end{aligned}
		\end{align}
Inserting \eqref{eq0111} into \eqref{eqs0110},  we have \eqref{G-N-type-3}.

	\end{proof}

\end{Lem}
We list the following special cases for the sake of convenience.
\begin{Cor}
	Under the same assumptions of \ref{Lem-GN2}, one has,
	\begin{align}\label{GN2}
	\begin{aligned}
	\|u\|_{L^4(\Torus_{\Lambda}^3)}\leq& C_1\|\nabla u\|^{\frac{3}{4}}_{L^2(\Torus_{\Lambda}^3)}\| u\|^{\frac{1}{4}}_{L^2(\Torus_{\Lambda}^3)}+C_2\Lambda^{-\frac{3}{4}}\|u\|_{L^2(\Torus_{\Lambda}^3)},\\
	\|u\|_{L^6(\Torus_{\Lambda}^3)}\leq &C_1\|\nabla u\|_{L^2(\Torus_{\Lambda}^3)}+C_2\Lambda^{-1}\|u\|_{L^2(\Torus_{\Lambda}^3)},\\
	\|u\|_{L^{\infty}(\Torus_{\Lambda}^3)}\leq& C_1\|\nabla^2 u\|^{\frac{3}{4}}_{L^2(\Torus_{\Lambda}^3)}\| u\|^{\frac{1}{4}}_{L^2(\Torus_{\Lambda}^3)}+C_2\Lambda^{-\frac{3}{2}}\|u\|_{L^2(\Torus_{\Lambda}^3)}.
	\end{aligned}
	\end{align}
\end{Cor}

\section{Vanishing Viscosity Limit}\label{sec3} 

\subsection{The perturbation system}

We first introduce the scaled variables. Setting 
\begin{align}\label{eq11} 
y=\frac{x}{\varepsilon}, \ \ \ \tau=\frac{t}{\varepsilon},\ \ 
{y\in\Omega:=\R\times\bar\Torus^2_\e,\ \ \tau\in\R^+,}
\end{align}
where $\bar\Torus_\e^2:=[0,\e^{b-1}]^2$ and $ y
=\left( {y_1}
, y'\right)=\left(y_{1}, y_{2}, y_{3}\right) \in \mathbb{R}^{3}$. In order to prove Theorem \ref{mt}, we study the perturbation of {the global smooth solution} $\left(\rho^\e, \uf^\e, \theta^\e\right)$ to the problem \eqref{NS} around {the ansatz
} $(\tilde{\rho}, \tilde{\uf}, \tilde{\theta})\left(t, x\right)$ by
\begin{align}\label{eqii}
	\phi{(\tau,y)}:=\rho^\e(t, x)-{\rhot}\left(t, x\right),\quad
	\psi{(\tau,y)}:=\uf^\e(t, x)-{\ut}(t, x),\quad
	\xi{(\tau,y)}:=\theta^\e(t, x)-\tilde{\theta}\left(t, x\right),
\end{align}
{with a family of smooth initial data
\begin{align}
(\rho^\e, u^\e,\T^\e)(0,x)=(\rhot,\ut,\Tt)(0,x)+(\phi,\psi,\zeta)(0,y),
\end{align}
where 
\begin{align}\label{init0}
\|(\phi_0,\psi_0,\zeta_0)(y)\|^2_{H^{2}
(\Omega)}=O(1)\e^{\frac{1}{3}+\gamma a}.
\end{align}}
For simplicity, we drop the superscript $\e$ of $\left(\rho^\e, \uf^\e, \theta^\e\right)$ as $(\rho,\uf,\theta)$ from now on.
Substituting \eqref{eq11} and \eqref{eqii} into the system \eqref{NS} and using the equations \eqref{ae} for $(\rhot,\ut,\thetat)$, one can obtain 
\begin{align}\label{sys-perturb}
	\begin{cases}
		\phi_{\tau}+u \cdot \nabla \phi+\rho \operatorname{div} \psi+\psi \cdot \nabla \tilde{\rho}+\phi \operatorname{div} \tilde{u}=-e_0, \\
		\rho\psi_{\tau}+\rho u \cdot \nabla \psi+\nabla(p-\tilde{p})+
		\rho\psi \cdot \nabla \tilde{u}-\frac{\phi}{\tilde{\rho}}\nabla\tilde{p}=\dv\left(2\mu(\Tt)\mathbb{D}(\psi)+\lambda(\Tt)\dv\psi \mathbb{I}\right)+\tilde{R},\\
		\rho\zeta_{\tau}+\rho u \cdot \nabla \zeta+p \operatorname{div} \psi+\rho\psi \cdot \nabla \tilde{\theta}+R\rho \zeta \operatorname{div} \tilde{u} =\dv(\kappa(\thetat)\nabla \zeta)+\tilde{R}_4,
	\end{cases}
\end{align}
with the initial data
\begin{align}\label{eqin}
(\phi,\psi,\zeta)(0,y)=(\phi_0,\psi_0,\zeta_0)(y),
\end{align}
where 
\begin{align}\label{R}
\begin{aligned}
	\tilde{R}=\frac{\rho}{\rhot}(e_0\ut-
	\ev)+\dv\big[2(\mu(\T)-\mu(\Tt))\mathbb{D}u+(\lambda(\T)-\lambda(\Tt))\dv  u\mathbb{I}\big]-\frac{\phi}{\rhot}
	\dv\big(2\mu(\Tt)\mathbb{D}\ut+\lambda(\Tt)\dv\ut\mathbb{I}\big),
	\end{aligned}
	\end{align}
\begin{align} \label{R_4}
\begin{aligned}
	\tilde{R}_4&=\frac{\rho}{\rhot}\left[\ev\ut-e_0\left(\frac{|\tilde{u}|^2}{2}-\tilde\theta\right)
	-e_4\right]
	+\dv\left((\kappa(\theta)-\kappa(\thetat))\nabla\T\right)-\frac{\phi}{\rhot}\left[\mu(\thetat)\frac{\left|\nabla \tilde{u}+(\nabla \tilde{u})^{t}\right|^{2}}{2}+\lambda(\thetat)(\operatorname{div} \tilde{u})^{2}\right] \\
	&+\mu(\theta)\frac{\left|\nabla u+(\nabla u)^{t}\right|^{2}}{2}+\lambda(\theta)(\operatorname{div} u)^{2}-\mu(\thetat)\frac{\left|\nabla \tilde{u}+(\nabla \tilde{u})^{t}\right|^{2}}{2}-\lambda(\thetat)(\operatorname{div} \tilde{u})^{2}-\frac{\phi}{\rhot}\dv(\kappa(\Tt)\nabla\Tt),
\end{aligned}
\end{align}
and $\tilde p=R\rhot\Tt$. 

\subsection{ A priori estimates}

For any  $0<\tau_1(\e)\le+\infty,$ define the solution space for \eqref{sys-perturb}-\eqref{R_4} by 
\begin{align}
\begin{aligned}
\mathcal{A}(0,\tau_1(\e))=\big\{&(\phi,\psi,\zeta)\big|(\phi,\psi,\zeta)\in C\big(0,\tau_1(\e); H^2(\Omega)\big),\nabla\phi\in L^2\big(0,\tau_1(\e); H^1(\Omega)\big),\\
&(\nabla\psi,\nabla\zeta)\in L^2\big(0,\tau_1(\e);H^2(\Omega)\big)\big\}.
\end{aligned}
\end{align}

The analysis is performed under the a priori assumptions
\begin{align}\label{eq09_1}
\begin{aligned}
\sup_{\tau\in[0,\tau_1(\varepsilon)]}\|(\phi,\psi)\|_{L^\infty(\Omega)}&\leq C\sup_{\tau\in[0,\tau_1(\varepsilon)]}\Big\{ \e^{-(b-1)}\|\nabla (\phi,\psi)\|^{\frac{1}{2}}\| (\phi,\psi)\|^{\frac{1}{2}}+\e^{-\frac{b-1}{2}}
		\|\nabla (\phi,\psi)\|\\
		&\qquad+\|\nabla^2 (\phi,\psi)\|^{\frac{3}{4}}\| (\phi,\psi)\|^{\frac{1}{4}}\Big\}
		\le \e^a,
		\end{aligned}
		\end{align}
	\begin{align}\label{eq09_2}
	\begin{aligned}
	\sup_{\tau\in[0,\tau_1(\varepsilon)]}\|\zeta\|_{L^\infty(\Omega)}&\leq C\sup_{\tau\in[0,\tau_1(\varepsilon)]}\Big\{ \e^{-(b-1)}\|\nabla \zeta\|^{\frac{1}{2}}\| \zeta\|^{\frac{1}{2}}+\e^{-\frac{b-1}{2}}
		\|\nabla \zeta\|+\|\nabla^2 \zeta\|^{\frac{3}{4}}\|\zeta\|^{\frac{1}{4}}\Big\}\\
		&
		\le \e^{a(\gamma-1)},
		\end{aligned}
		\end{align}
where $a$ is defined by \eqref{eqda}, 
  $[0,\tau_1(\varepsilon)]$ is the time interval in which the solution exists. Set 
\begin{align}\label{eq091}
\nu=\varepsilon^{ Za}|\log\varepsilon|,\ \ \delta=\varepsilon^a, \ \ Z=
\frac{1}{4
(\alpha+1)\gamma},
\end{align}
with $\varepsilon\ll1$.
From \eqref{eq094} and \eqref{eq09_1}-\eqref{eq09_2}, and by the definition of the rarefaction wave profile $\bar\theta=\bar\rho^{\gamma-1}\ge\nu^{\gamma-1}e^{\bar S}$, it holds
\begin{align}\label{eq092}
\frac{\rhot}{2}\le\rho\le\frac{3\rhot}{2}
, \ \ \ 
\frac{\thetat}{2}\le\theta\le\frac{3\thetat}{2}.
\end{align}
 
Since the density has lower bound which mainly depends on $\varepsilon$ if $\eta_0$ is small enough in the initial assumptions \eqref{eq095}, it is standard to proof the local existence of the solution to \eqref{sys-perturb}-\eqref{R_4} in the time interval $[0,\tau_0(\varepsilon)]$, we omit this proof process for brevity. Then in order to complete the proof of Theorem \ref{mt}
, it is requisite to extend the local solution to the global solution in $[0,\infty)$ for some small but fixed viscosity 
and heat conductivity coefficients $0<\varepsilon\ll1$. Thus it suffices to obtain the following uniform a priori estimates.
\begin{Lem}\label{ape1}
{\bf (A 
 priori estimate).} Let $(\phi,\psi,\zeta)\in\mathcal A(0,\tau_1(\e))$ be a solution to the problem \eqref{sys-perturb}-\eqref{eqin} under the a priori assumptions \eqref{eq09_1}-\eqref{eq09_2}
, where $\tau_1(\e)$ is the maximum existence time of the solution. Then there exists a positive constant $\e_1$ such that   if $0<\e\le\e_1$, then 
\begin{align}\label{eqe01}
\begin{aligned}
\sup _{\tau \in[0,\tau_1(\e)]} &\int_{\Omega}\left({{\rhob}}^{\gamma-2} \phi^{2}+{{\rhob}} \psi^{2}+{{\rhob}}^{2-\gamma} \zeta^{2}\right)(\tau, y) d y \\
		&\quad+\int_{0}^{\tau_1(\e)} \int_{\Omega}\left[\px\bar{u}\left({{\rhob}}^{\gamma-2} \phi^{2}+{{\rhob}} \psi^{2}+{{\rhob}}^{2-\gamma} \zeta^{2}\right)+{{\thetab}}^{\alpha} |\nabla\psi|^{2}+{\thetab}^{\alpha-1} |\nabla\zeta|^{2}\right] d yd\tau 
  \leq \e^{\frac{1}{3}},
\end{aligned}
\end{align}
\begin{align}\label{eqe02}
	\begin{aligned}
	\sup_{\tau\in[0,\tau_1(\e)]}\int_{\Omega}
	\frac{\Tb^{2\alpha}}{
	\bar\rho^3}|\nabla\phi|^2
	dy
	+\int_0^{\tau_1(\e)}\int_{\Omega}\frac{{\Tb}^{\alpha+1}}{\bar\rho^2}|\nabla\phi|^2dyd\tau
 \leq \e^{\frac{1}{3}-Z(\gamma-1)a}|\log\e|^{-(\gamma-1)},
	\end{aligned}
	\end{align}
\begin{align}\label{eqe03}
\begin{aligned}
\sup_{\tau\in[0,\tau_1(\e)]}\int_\Omega&\left[\frac{\bar\T}{\bar\rho}|\nabla\phi|^2+
\bar\rho|\nabla\psi|^2
+\frac{\bar\rho}{\bar\T}|\nabla\zeta|^2\right]dy\\
&+\int_0^{\tau_1(\e)}\int_{\Omega}\left[\bar\theta^\alpha
|\Delta\psi|^2+\bar\theta^\alpha
|\nabla\dv\psi|^2+\bar\theta^{\alpha-1}
|\Delta\zeta|^2\right]dyd\tau\le\e^{\frac{1}{3}-Z(\gamma-1)a}|\log\e|^{-(\gamma-1)},
\end{aligned}
\end{align}
\begin{align}\label{eqe04}
\begin{aligned}
&\sup_{\tau\in[0,\tau_1(\e)]}\int_\Omega\left[\frac{\Tb}{\rhob}|\nabla^2\phi|^2+
\rhob|\nabla^2\psi|^2+\frac{\rhob}{\Tb}|\nabla^2\zeta|^2\right]dy\\
&\qquad+\int_{0}^{\tau_1(\e)} \int_{\Omega}\left[\Tb^\alpha|\nabla\Delta\psi|^2+\Tb^\alpha|\nabla^2\dv\psi|^2+\Tb^{\alpha-1}|\nabla\Delta\zeta|^2
\right]dyd\tau
\le\e^{\frac{1}{3}-2Z(\gamma-1)a}|\log\e|^{-2(\gamma-1)},
\end{aligned}
\end{align}
\begin{align}\label{eqje0223}
	\begin{aligned}
	&\sup_{\tau\in[0,t]}\int_{\Omega}
	\frac{\Tb^{2\alpha}}{
	\bar\rho^3}|\nabla^2\phi|^2
	dy
	+
	\int_0^t\int_{\Omega}\frac{{\Tb}^{\alpha+1}}{\bar\rho^2}|\nabla^2\phi|^2dyd\tau\le\e^{\frac{1}{3}-2Z(\gamma-1)a}|\log\e|^{-2(\gamma-1)}.
\end{aligned}
\end{align}
\end{Lem}


\begin{proof} The proof of the lemma consists of the following steps.\\
{\bf{Step1.}} First, we aim to obtain the basic energy estimate. Defining
\begin{align*}
	\hat{\mathbb{E}}:=R{{\thetat}}\hat\Phi(\frac{{\tilde{\rho}}}{\rho})+\frac{|\psi|^2}{2}+ {\thetat}\hat\Phi(\frac{\theta}{{\thetat}}),
\end{align*}
where 
\begin{align}
	\hat\Phi(s):=s-\ln s-1,
\end{align}
then we have the following equation by direct computations
\begin{align}
	\p_\tau(\rho\hat{\mathbb{E}})+\dv \hat Q+\frac{\mu({\Tt})}{2}|\nabla\psi+(\nabla\psi)^t|^2+\lambda({\Tt})|\dv\psi|^2+\frac{\kappa({\Tt}){\Tt}}{\T^2}|\nabla\zeta|^2+\hat S=\hat H,
\end{align}
where
\begin{align}
\begin{aligned}
&\hat Q=\rho u\hat {\mathbb{E}}+(p-\tilde p)\psi-2\mu(\tilde\theta)\mathbb{D}(\psi)\cdot\psi-\lambda(\tilde\theta)\dv\psi\mathbb{I}\cdot\psi-\kappa(\tilde\theta)\nabla\zeta\frac{\zeta}{\theta},\\
&\hat S=-[(\rho\tilde\theta)_\tau+\dv(\rho u\tilde\theta)][R\Phi(\frac{\tilde\rho}{\rho})+\Phi(\frac{\theta}{\tilde\theta})]+\frac{\rho\tilde\theta_\tau}{\theta\tilde\theta}\zeta^2+\frac{\rho u\cdot\nabla\tilde\theta}{\theta\tilde\theta}\zeta^2+\rho\psi\cdot\nabla\tilde u\cdot\psi+\rho\psi\cdot\nabla\tilde\theta\frac{\zeta}{\theta}\\
&\qquad\qquad+\frac{R\rho\zeta^2}{\theta}\dv \tilde u-R\phi\psi\cdot\nabla\tilde\theta:=\hat S_1+\hat S_2,\\
&\hat H=-\frac{R\tilde\theta\phi}{
\rho}e_0+\tilde R\cdot\psi+\tilde R_4\frac{\zeta}{\theta}+\kappa(\tilde\theta)\zeta\frac{\nabla\tilde\theta\cdot\nabla\zeta}{\theta^2}.
\end{aligned}
\end{align}
Integrating the above equation over $[0,\tau_1
]\times\Omega
$ with $\tau_1:=\tau_1(\e)$, we have 
\begin{align}\label{eq0128}
\begin{aligned}
	&\left.\int_\Omega
	 R{\rho{\thetat}}\hat\Phi(\frac{{\tilde{\rho}}}{\rho})+\rho\frac{|\psi|^2}{2}+\rho\thetat\hat\Phi(\frac{\theta}{{\thetat}})d{y}\right|_0^{\tau_1}
	 +\int_0^{\tau_1}
	 \int_\Omega\big({\thetat}^{\alpha}|\nabla \psi|^2+{\thetat}^{\alpha-1}|\nabla\zeta|^2\big)d{y}d\tau+\int_{0}^{\tau_1}
	 {\int_\Omega} \hat S{dy}d\tau\\
	 &\qquad\leq\int_{0}^{\tau_1}
	 {\int_\Omega}  |\hat H|{dy}d\tau.
	 \end{aligned}
\end{align}
Since $\Phi(1)=\Phi'(1)=0$, $\Phi''(1)>0$, one has
\begin{align}
	R{\rho{\thetat}}\Phi(\frac{{\tilde{\rho}}}{\rho})+\rho\frac{|\psi|^2}{2}+\rho\thetat\Phi(\frac{\theta}{{\thetat}})\geq C\big({\rhot}^{\gamma-2}\phi^2+{\rhot}|\psi|^2+{\rhot}^{2-\gamma}\zeta^2\big).
\end{align}
Then making use of the estimate in Lemma \ref{lem-periodic-solution}, we have
\begin{align}\label{pe}
\|(\rhot-\bar\rho,\ut-\bar u,\thetat-\bar\theta)\|_{W^{2,\infty}(\mathbb{R}^3)}\leq O(\nu^{-1}\eta_0) e^{-\varepsilon^{n_2}\tau}.
\end{align}
Thus, 
\begin{align}\label{eq0129}
\begin{aligned}
\int_0^{\tau_1}
\int_\Omega\hat S_1dyd\tau&=\int_0^{\tau_1}
\int_\Omega-[(\rho\bar{\theta})_\tau+\dv(\rho u\bar{\theta})][R\Phi(\frac{\tilde\rho}{\rho})+\Phi(\frac{\theta}{\tilde\theta})]+\frac{\rho\bar\theta_\tau}{\theta\tilde\theta}\zeta^2+\frac{\rho u\cdot\nabla\bar\theta}{\theta\tilde\theta}\zeta^2\\
&\qquad+\rho\psi\cdot\nabla\bar u\cdot\psi+\rho\psi\cdot\nabla\bar\theta\frac{\zeta}{\theta}+\frac{R\rho\zeta^2}{\theta}\dv \bar u-R\phi\psi\cdot\nabla\bar\theta dyd\tau\\
&\ge \frac{1}{160}\int_0^{\tau_1}
\int_\Omega\p_1\bar u_1(\tilde\rho^{\gamma-2}\phi^2+\tilde\rho\psi_1^2+\tilde\rho^{2-\gamma}\zeta^2)dyd\tau\\
&\quad-O(\eta_0)\int_{0}^{\tau_1}
e^{-\e^{n_2} \tau}d\tau\int_{\Omega} \p_1\bar u_1\big({\rhot}^{\gamma-2}\phi^2+{\rhot}|\psi|^2+{\rhot}^{2-\gamma}\zeta^2\big)dy,
\end{aligned}
\end{align}
\begin{align}\label{eq0130}
\begin{aligned}
\int_0^{\tau_1}
\int_\Omega\hat S_2dyd\tau&=\int_0^{\tau_1}
\int_\Omega -[(\rho(\tilde\theta-\bar\theta))_\tau+\dv(\rho u(\tilde\theta-\bar\theta))][R\Phi(\frac{\tilde\rho}{\rho})+\Phi(\frac{\theta}{\tilde\theta})]+\frac{\rho(\tilde\theta-\bar\theta)_\tau}{\theta\tilde\theta}\zeta^2\\
&\qquad +\frac{\rho u\cdot\nabla(\tilde\theta-\bar\theta)}{\theta\tilde\theta}\zeta^2+\rho\psi\cdot\nabla(\tilde u-\bar u)\cdot\psi+\rho\psi\cdot\nabla(\tilde\theta-\bar\theta)\frac{\zeta}{\theta}\\
&\qquad+\frac{R\rho\zeta^2}{\theta}\dv (\tilde u-\bar u)-R\phi\psi\cdot\nabla(\tilde\theta-\bar\theta) dyd\tau\\
	&\le O(\eta_0)\int_0^{\tau_1}
	 e^{-\varepsilon^{n_2}\tau}\int_\Omega(\tilde\rho^{\gamma-2}\phi^2+\tilde\rho|\psi|^2+\tilde\rho^{2-\gamma}\zeta^2)dyd\tau.
\end{aligned}
\end{align} 
Similarly, 
\begin{align}\label{eq131}
\begin{aligned}
&\int_{0}^{\tau_1}
{\int_\Omega}  |\hat H|{dy}d\tau
\leq \frac{1}{160}\int_0^{\tau_1}
\int_\Omega\big[\p_1\bar u_1(\tilde\rho^{\gamma-2}\phi^2+\tilde\rho\psi_1^2+\tilde\rho^{2-\gamma}\zeta^2)+{\thetat}^{\alpha}|\nabla \psi|^2+{\thetat}^{\alpha-1}|\nabla\zeta|^2\big]dyd\tau\\
&\quad+\frac{1}{160}\sup_{\tau\in[0,{\tau_1}
]}\int_\Omega(\tilde\rho^{\gamma-2}\phi^2+\tilde\rho|\psi|^2+\tilde\rho^{2-\gamma}\zeta^2)dy+\nu^{-\gamma+\frac{1}{2}}\int_0^{\tau_1}
\|\tilde{\rho}^{\frac{\gamma-2}{2}}\phi\|\|\bar{u}_{1y_1y_1}\|
\|(\psi,\zeta)\|_{L^{\infty}}d\tau\\
&\quad+\int_0^{\tau_1}
\int_{\Omega}(\bar{\T}^{\alpha}|\p_1^2\bar{u}_1|+\bar{\T}^{\alpha-\frac{1}{2}}|\p_1\bar{u}_1|^2)(\psi,\frac{\zeta}{\T})dyd\tau+\int_0^{\tau_1}
\nu^{-\gamma}\eta_0 e^{-\varepsilon^{n_2}\tau}\|(\phi,\psi,\zeta)\|d\tau,
\end{aligned}
\end{align}
where 
\begin{align}
\begin{aligned}
&\int_0^{\tau_1}
\int_{\Omega}(\bar{\T}^{\alpha}|\p_1^2\bar{u}_1|+\bar{\T}^{\alpha-\frac{1}{2}}|\p_1\bar{u}_1|^2)(\psi,\frac{\zeta}{\T})dyd\tau\\
&\leq \int_0^{\tau_1}
\|\thetat^{\alpha-1}\|_{\infty}\bigg(\|(\psi,\zeta)\|_{\infty}\|\p_1^2\bar{u}_1\|_{L^1}+\nu^{-\frac{1}{2}}\|\p_1\bar{u}_1\|_{L^4}^2\|({\rhot^{\frac{1}{2}}}\psi,{\rhot}^{\frac{2-\gamma}{2}}\zeta)\|\bigg)d\tau\\
&\le \frac{1}{160}\bigg(\sup_{\tau\in[0,{\tau_1}
]}\|({\rhot^{\frac{1}{2}}}\psi,{\rhot}^{\frac{2-\gamma}{2}}\zeta)\|^2+\int_0^{\tau_1}
\|({\thetat}^{\frac{\alpha}{2}}\nabla \psi,{\thetat}^{\frac{\alpha-1}{2}}\nabla\zeta)\|^2d\tau\bigg)\\
&\qquad\qquad+\e^{\beta}\int_0^{\tau_1}
\|({\thetat}^{\frac{\alpha}{2}}\nabla^2\psi,{\thetat}^{\frac{\alpha-1}{2}}\nabla^2\zeta)\|^2d\tau+\e^{\frac{1}{3}},\\
&\nu^{-\gamma+\frac{1}{2}}\int_0^{\tau_1}
\|\tilde{\rho}^{\frac{\gamma-2}{2}}\phi\|\|\bar{u}_{1y_1y_1}\|_{L^2}\|(\psi,\zeta)\|_{L^{\infty}}d\tau
\leq \frac{1}{160}\int_0^t\|({\thetat}^{\frac{\alpha}{2}}\nabla \psi,{\thetat}^{\frac{\alpha-1}{2}}\nabla\zeta)\|^2d\tau\\
&\qquad+\e^{\frac{4}{3}(1-a)}\int_0^{\tau_1}
\|(\thetat^{\frac{\alpha}{2}}\nabla^2\psi,\thetat^{\frac{\alpha}{2}}\nabla^2\zeta)\|^2d\tau+\frac{1}{160}\sup _{\tau \in[0,\tau_1]}\|({\rhot}^{\frac{\gamma-2}{2}}\phi,{\rhot^{\frac{1}{2}}}\psi,{\rhot}^{\frac{2-\gamma}{2}}\zeta)\|^2,
\end{aligned}
\end{align}
where if $\e\ll1$, we use the fact that $$\max\big\{C\nu^{-\frac{1}{2}}\thetat^{-\frac{\alpha+4}{2}}\e^{\frac{1-a}{2}-2(b-1)}
, C\nu^{-\frac{1}{4}}\thetat^{-\frac{3\alpha+8}{4}}\e^{\frac{3}{4}(1-a)-\frac{5}{4}\beta}
\big\}\le\e^{\frac{1}{3}},$$
for $1\le b\leq \frac{13}{12}-\frac{a}{4}-\big\{\frac{4\gamma-3+(\gamma-1)\alpha}{4}\big\}Za,$
where small constant $\beta>0$ 
 will be determined later.

Combining all the estimates above and initial data \eqref{init0}, when $\eta_0$ is suitably small, we have
\begin{align}\label{basic-energy01}
	\begin{aligned}
		\sup _{\tau \in[0,\tau_1(\e)]} &\int_{\Omega}\left({{\rhot}}^{\gamma-2} \phi^{2}+{{\rhot}} \psi^{2}+{{\rhot}}^{2-\gamma} \zeta^{2}\right)(\tau, y) d y \\
		&\quad+\int_{0}^{\tau_1}
		 \int_{\Omega}\left[\px\bar{u}\left({{\rhot}}^{\gamma-2} \phi^{2}+{{\rhot}} \psi^{2}+{{\rhot}}^{2-\gamma} \zeta^{2}\right)+{{\thetat}}^{\alpha} |\nabla\psi|^{2}+{\thetat}^{\alpha-1} |\nabla\zeta|^{2}\right] d yd\tau\\
	&\quad \leq \e^{\frac{1}{3}}
 +\e^{\beta}\int_0^{\tau_1}
 \|(\thetat^{\frac{\alpha}{2}}\nabla^2\psi,\thetat^{\frac{\alpha}{2}}\nabla^2\zeta)\|^2d\tau.
\end{aligned}
\end{align}


{\bf Step 2.} Now, we need to estimate $\|\nabla\phi\|^2$. Multiplying \cref{sys-perturb}$_2$ by $\frac{2\mu({\Tt})+\lambda(\Tt)}{\rho^2}\nabla{\phi}$, one has
	\begin{align}\label{eq-phiy01}
		&\p_\tau\left[\frac{\Lambda({\Tt})}{\rho}\psi\cdot\nabla\phi\right]+\frac{R\Lambda({\Tt}){\Tt}}{\rho^2}|\nabla\phi|^2+\frac{\Lambda({\Tt})}{\rho} u\cdot\nabla \psi\cdot\nabla\phi+\frac{R\Lambda({\Tt})}{\rho^2}\nabla(\rho\zeta)\cdot\nabla\phi+
		\frac{\Lambda({\Tt})}{\rho} \psi \cdot \nabla {\tilde{u}}\cdot\nabla\phi\nonumber\\
		&+\frac{\phi R\Lambda({\Tt})}{\rho^2}\nabla\Tt\cdot\nabla\phi-\frac{\phi\Lambda({\Tt})}{{\rhot}\rho^2}\nabla{\tilde{p}}\cdot\nabla\phi=\frac{{\Lambda^2({\Tt})}}{\rho^2}\nabla\phi\cdot\nabla\dv\psi+\frac{\Lambda(\Tt)}{\rho^2}\tilde{R}\cdot\nabla\phi+\tilde{S},
	\end{align}
where $\Lambda(\Tt):=2\mu(\Tt)+\lambda(\Tt)$ and 
\begin{align*}
	\begin{aligned}
	\tilde{S}&:=\psi\cdot\p_\tau\left[\frac{\Lambda({\Tt})}{\rho}\nabla\phi\right]+\frac{\lambda'({\Tt})\Lambda(\Tt)}{\rho^2}\dv\psi\nabla\Tt\cdot\mathbb{I}\cdot\nabla\phi+\frac{2\Lambda(\Tt)}{\rho^2}\mu'(\Tt)\nabla\Tt\cdot\mathbb{D}\psi\cdot\nabla\phi\\
	&-
	\nabla\left(\frac{{\mu({\Tt})}\Lambda(\Tt)}{\rho^2}\right)\cdot\left(\nabla\psi\cdot\nabla\phi-\nabla\phi\cdot\nabla\psi\right)+\dv\left[\frac{\mu({\Tt})\Lambda(\Tt)}{\rho^2}\cdot\left(\nabla\psi\cdot\nabla\phi-\nabla\phi\cdot\nabla\psi\right)\right].
\end{aligned}
\end{align*}
To deal with the crossing term, we apply $\nabla$ to \cref{sys-perturb}$_1$ and multiply it by $\frac{\Lambda^2(\Tt)}{\rho^3}\nabla\phi$, 
\begin{align}\label{eq-phiy02}
	\begin{aligned}
	&\p_\tau\left[\frac{\Lambda^2(\Tt)}{2\rho^3}\left|\nabla\phi\right|^2\right]-\left[\p_\tau\left(\frac{\Lambda^2(\Tt)}{\rho^3}\right)+\dv\left(\frac{\Lambda^2(\Tt)u}{\rho^3}\right)\right]\frac{|\nabla\phi|^2}{2}+\frac{\Lambda^2({\Tt})}{\rho^2}\nabla\phi\cdot\nabla\dv\psi\\
	&+\frac{\Lambda^2({\Tt})}{\rho^3}\nabla\phi\cdot\nabla\rho\dv\psi+\frac{\Lambda^2({\Tt})\nabla\phi\cdot\nabla u\cdot\nabla\phi}{\rho^3}+\frac{\Lambda^2({\Tt})}{\rho^3}\nabla\phi\cdot\nabla(\psi\cdot\nabla{\rhot}+\phi\dv {\tilde{u}}+e_0)\\
	&\quad=\dv(-\frac{\Lambda^2(\Tt)|\nabla\phi|^2 u}{2\rho^3}).
\end{aligned}
\end{align}
Combining \cref{eq-phiy01,eq-phiy02}, integrating the above equation over $[0,\tau_1
]\times\Omega
$, we have
\begin{align}
	\begin{aligned}
	&\int_{\Omega}\frac{\Lambda^2(\Tt)}{2\rho^3}|\nabla\phi|^2+\left.\frac{\Lambda(\Tt)\psi\cdot\nabla\phi}{\rho}dy\right|_0^{\tau_1}+\int_0^{\tau_1}
	\int_{\Omega}\frac{R\Lambda({\Tt}){\Tt}}{\rho^2}|\nabla\phi|^2dyd\tau\leq&C\int_0^{\tau_1}
	\int_{\Omega}|\tilde{H}|dyd\tau,
\end{aligned}
\end{align}
where
\begin{align*}
	&\int_0^{\tau_1}
	\int_{\Omega}|\tilde{H}|dyd\tau\leq C\int_0^{\tau_1}
	\int_\Omega\left|\frac{\Lambda({\Tt})}{\rho} u+\frac{\lambda'({\Tt})\Lambda(\Tt)}{\rho^2}\nabla{\Tt}+\mu({\Tt})\nabla\left(\frac{\Lambda(\Tt)}{\rho^2}\right)\right||\nabla\psi||\nabla\phi|\\
	&+\frac{\Lambda^2({\Tt})}{\rho^3}\left(|\nabla\rho|+|\nabla {\tilde{u}}|+\frac{\rho}{\Lambda(\Tt)}\mu'(\Tt)|\nabla{\Tt}|\right)|\nabla\phi||\nabla\psi|+\frac{\Lambda^2({\Tt})}{\rho^3}|\nabla\phi\cdot\nabla e_0|+\frac{\Lambda(\Tt)}{\rho^2}|\tilde{R}\cdot\nabla\phi|\nonumber\\
	&+\frac{\Lambda({\Tt})}{\rho^2}|\nabla(\rho\zeta)\cdot\nabla\phi|+\left|\phi \nabla{\Tt}+\frac{\phi}{{\rhot}}\nabla{\tilde{p}}+
	\rho \psi\cdot \nabla {\tilde{u}}\right|\frac{ \Lambda({\Tt})}{\rho^2}|\nabla\phi|+\left|\psi\cdot\p_\tau\left[\frac{\Lambda({\Tt})}{\rho}\nabla\phi\right]\right|\\
	&+\left|\p_\tau\left(\frac{\Lambda^2(\Tt)}{\rho^3}\right)+\dv\left(\frac{\Lambda^2(\Tt)u}{\rho^3}\right)\right|\frac{|\nabla\phi|^2}{2}+\frac{\Lambda^2({\Tt})}{\rho^3}|\nabla\phi|\big|(\psi\cdot\nabla)\nabla{\rhot}+\phi\nabla\dv{\tilde{u}}\big|dyd\tau\\
	&:=\sum_{i=1}^{9}J_i.
\end{align*}
Here for $b\le1+a[1-(\gamma(2\alpha+3)
+3) Z],$ we estimate $J_i,i=1,2,\cdots 9:$
\begin{align}\label{eqj1j2}
\begin{aligned}
	&J_1+J_2\leq C\nu^{-(\gamma-1)}\int_0^{\tau_1}
	\|(\thetat^{\frac{\alpha}{2}}\nabla\psi,\thetat^{\frac{\alpha}{2}}\nabla\zeta)\|^2d\tau+\frac{1}{160}\int_0^{\tau_1}
	\int_\Omega\frac{\Tb^{\alpha+1}}{\rhob^2}|\nabla\phi|^2dyd\tau\\
&\qquad\qquad+\e^{a+1-b}\nu^{-3-(\alpha+1)(\gamma-1)}\int_0^{\tau_1}
\|\sqrt{\frac{{\thetat}^{\alpha+1}}{\rhob^2}}\nabla\phi\|_{H^1}^2d\tau,\\
&J_3+J_4\le \e^{\frac{1}{3}}+\nu^{-(\gamma+1)}\e^{1-2a}+\frac{1}{160}\int_0^{\tau_1}
\int_{\Omega}\frac{\Tb^{\alpha+1}}{\rhob^2}|\nabla\phi|^2+{\Tb^{\alpha}}|\nabla\psi|^2+{\Tb^{\alpha-1}}|\nabla\zeta|^2\\
&\qquad\qquad+\bar\theta^\alpha|\nabla^2\psi|^2dyd\tau+O(\nu^{-(\gamma+2)}\e^{1-2a})\sup_{\tau\in[0,\tau_1
]}\int_\Omega(\bar\rho^{\gamma-2}\phi^2+{\bar{\rho}^{2-\gamma}} \zeta^{2})(\tau)dy\\
&\qquad\qquad+\e^{a+1-b}\nu^{-((1+\alpha)(\gamma-1)+2)}\int_0^{\tau_1}
\|({\thetat}^{\frac{\alpha-1}{2}}\nabla\zeta,{\thetat}^{\frac{\alpha}{2}}\nabla\psi)\|_{H^1}^2d\tau,\\
&J_5+J_6+J_9
\le\e^{\frac{1}{3}}+ \frac{1}{160}\sup_{\tau\in[0,\tau_1
]}\int_\Omega(\bar\rho^{\gamma-2}\phi^2+{\bar{\rho}} \psi^{2}+{\bar{\rho}^{2-\gamma}} \zeta^{2})dy\\
&\qquad+\frac{1}{160}\int_{0}^{\tau_1}
\int_\Omega\frac{\Tb^{\alpha+1}}{\rhob^2}|\nabla\phi|^2dyd\tau+O(\nu^{-\gamma}\e^{1-a})\sup_{\tau\in[0,\tau_1
]}\int_\Omega(\bar\rho^{\gamma-2}\phi^2+\rhob\psi^2+{\bar{\rho}^{2-\gamma}} \zeta^{2})(\tau)dy,\\
&J_7
\leq C\int_0^{\tau_1}\int_\Omega\left\{ \left[\left|\p_\tau\left(\frac{\Lambda({\Tt})}{\rho}\right)\right|+\left|\nabla\left(\frac{\Lambda({\Tt})}{\rho}\right){u}\right|\right]|\psi||\nabla\phi|+\left|\frac{\Lambda({\Tt})}{\rho}\dv{\ut}\right||\nabla\psi||\phi|\right.\\
&\qquad\left.+\left[\left|\nabla\left(\frac{\Lambda({\Tt})}{\rho}\right){\rho}\right|+\left|\frac{\Lambda({\Tt})}{\rho}\nabla\rhot\right|\right]|\psi||\nabla\psi|\right\}+\Big\{
\left|\frac{\Lambda({\Tt})}{\rho}u\right||\nabla\psi||\nabla\phi|+\left|{\Lambda({\Tt})}\right||\nabla\psi|^2
\\
&\qquad+
\left|\nabla\left(\frac{\Lambda({\Tt})}{\rho}\right)\right|\left(\dv u|\psi||\phi|+\nabla{\rhot}|\psi|^2\right)\Big\}
+\left\{\left|\nabla\left(\frac{\Lambda({\Tt})}{\rho}\right)
\right||\psi|e_0+\left|\frac{\Lambda({\Tt})}{\rho}\right||\nabla\psi|e_0\right\} dyd\tau\\
&\ \le \e^{\frac{1}{3}}+\nu^{-(\gamma-1)}\int_0^{\tau_1}
\|(\thetat^{\frac{\alpha}{2}}\nabla\psi,\thetat^{\frac{\alpha-1}{2}}\nabla\zeta)\|^2d\tau+\frac{1}{160}\int_{0}^{\tau_1}
\int_\Omega\frac{\Tb^{\alpha+1}}{\rhob^2}|\nabla\phi|^2dyd\tau\\
&\qquad+O(\eta_0+\e^{1-a}\nu^{-\gamma})\sup_{\tau\in[0,\tau_1
]}\int_\Omega(\bar\rho^{\gamma-2}\phi^2+\rhob\psi^2+{\bar{\rho}^{2-\gamma}} \zeta^{2})(\tau)dy.
\end{aligned}
\end{align}
Similarly, there exists
\begin{align}\label{eqj0122}
\begin{aligned}
J_8\le \e^{a+1-b}\nu^{-3-(\alpha+1)(\gamma-1)}\int_0^{\tau_1}
\|\sqrt{\frac{{\thetat}^{\alpha+1}}{\rhob^2}}\nabla\phi\|_{H^1}^2d\tau+
 \frac{1}{160}
\int_0^{\tau_1}
\int_{\Omega}\frac{\Tb^{\alpha+1}}{\rhob^2}|\nabla\phi|^2
dyd\tau.
\end{aligned}
\end{align} 
Therefore, by \eqref{basic-energy01}, we choose $\beta=2(\gamma-1)Za$. It yields
\begin{align}\label{eqj0123}
	\begin{aligned}
	\sup_{\tau\in[0,\tau_1
	]}\int_{\Omega}&\left(\frac{\Tb^{2\alpha}}{
	\bar\rho^3}|\nabla\phi|^2+
	\frac{\Tb^\alpha\psi\nabla\phi}{\bar\rho}\right)(\tau)dy
	+
	\int_0^{\tau_1}
	\int_{\Omega}\frac{{\Tb}^{\alpha+1}}{\bar\rho^2}|\nabla\phi|^2dyd\tau\\
	&\leq \e^{\frac{1}{3}}\nu^{-(\gamma-1)}+
	\int_{\Omega}
	\bar\theta^\alpha|\nabla^2\psi|^2+\Tb^{\alpha-1}|\nabla^2\zeta|^2d\tau+\e^{(\gamma-1)Za}
	\int_0^{\tau_1}
	\|\sqrt{\frac{{\thetat}^{\alpha+1}}{\rhob^2}}\nabla^2\phi\|^2d\tau.	
\end{aligned}
\end{align}


{\bf Step 3. }
Now we estimate $\nabla\psi$ and $\nabla\zeta$. Applying $\p_i:=\p_{y_i}$, $i=1,2,3$, to \cref{sys-perturb}, we have
\begin{align}\label{sys-perturb1}
	\begin{cases}
		\p_i\phi_{\tau}+u \cdot \nabla\p_i\phi+\rho \operatorname{div}\p_i \psi=\hat{Q}_0, \\
		\rho\p_i\psi_{\tau}+\rho u \cdot \nabla\p_i \psi+\nabla\p_i(p-\tilde{p})=\dv\p_i\left(2\mu(\Tt)\mathbb{D}(\psi)+\lambda(\Tt)\dv\psi \mathbb{I}\right)+\hat{Q}_1,\\
		\rho\p_i\zeta_{\tau}+\rho u \cdot \nabla\p_i \zeta+p \operatorname{div}\p_i \psi =\dv\p_i(\kappa(\thetat)\nabla \zeta)+\hat{Q}_2,
	\end{cases}
	\end{align} 
where
\begin{align*}
	\hat{Q}_0:=&-\p_i e_0-\p_i u\cdot\nabla\phi-\p_i\rho\dv\psi-\pii\left(\psi\cdot\nabla\rhot+\phi\dv\ut\right),\\
	\hat{Q}_1:=&\pii\tilde{R}-\pii\rho\psi_{\tau}-\pii(\rho u)\cdot\nabla\psi-\pii
	\left(\rho\psi \cdot \nabla \tilde{u}-\frac{\phi}{\tilde{\rho}}\nabla\tilde{p}\right),\\
	\hat{Q}_2:=&\pii\tilde{R}_4-\pii\rho\zeta_{\tau}-\pii(\rho u)\cdot\nabla\zeta-\pii p\dv\psi-\pii(\rho\psi\cdot\nabla\Tt+R\rho\dv\ut\zeta).
\end{align*}
Multiplying \cref{sys-perturb1}$_1$ by $\frac{R{\Tt}}{\rho}\pii\phi$, \cref{sys-perturb1}$_2$ by $\pii\psi$, \cref{sys-perturb1}$_3$ by $\frac{1}{\T}\pii\zeta$, and adding the results up, one has
		\begin{align*}
			&\p_\tau\left[\frac{R\Tt}{2\rho}|\nabla\phi|^2+\frac{\rho|\nabla\psi|^2}{2}+\frac{\rho}{2\T}|\nabla\zeta|^2\right]+\mu(\Tt)|\Delta\psi|^2+(\mu(\Tt)+\lambda(\Tt))|\nabla\dv\psi|^2+\frac{\kappa(\Tt)}{\T}|\Delta\zeta|^2\\
			&-\left[\p_\tau\left(\frac{R\Tt}{2\rho}\right)+\dv\left(\frac{R\Tt u}{2\rho}\right)\right]|\nabla\phi|^2-\left[\p_\tau\left(\frac{\rho}{\T}\right)+\dv\left(\frac{\rho u}{\T}\right)\right]\frac{|\nabla\zeta|^2}{2}+\nabla\kappa(\Tt)\cdot\nabla\zeta\dv(\frac{\nabla\zeta}{\T})\\
			&+R(\nabla{\Tt} {\phi}+ \nabla \rho{\zeta}) \cdot \nabla \dv{\psi}-\kappa({\Tt})\frac{\nabla\zeta\cdot\nabla\theta\Delta\zeta}{\T^2}+\left[2\nabla\mu(\Tt)\cdot\mathbb{D}\psi+\nabla\lambda(\Tt)\cdot\dv\psi\mathbb{I}\right]\cdot\nabla^2\psi\\
			&-\frac{R{\Tt}}{\rho}\nabla\phi\cdot\hat{Q}_0-\nabla\psi\cdot\hat{Q}_1-\frac{1}{\T}\nabla\zeta\cdot\hat{Q}_2+\dv(\cdots)=0.
		\end{align*}
	
	Integrating the above equation over $\Omega\times[0,\tau_1]$, one has
	\begin{align*}
		\begin{aligned}
			\mathcal{I}_1 +\int_0^{\tau_1}\mathcal{D}_1d\tau\leq \mathcal{H}_1+\mathcal{Q}_1,
		\end{aligned}
	\end{align*}
where

\begin{align*}
\mathcal{I}_1=\int_{\Omega}\frac{R\Tt}{2\rho}|\nabla\phi|^2+\frac{\rho|\nabla\psi|^2}{2}+\frac{\rho}{2\T}|\nabla\zeta|^2dy,\ \ \ \mathcal{D}_1=\int_{\Omega}{\thetat}^{\alpha}|\nabla^2\psi|^2+{\thetat}^{\alpha-1}|\nabla^2\zeta|^2dy,
\end{align*}
\begin{align*}
\mathcal{H}_1:=&\int_{0}^{\tau_1}\int_{\Omega}\left[\p_\tau\left(\frac{R\Tt}{2\rho}\right)+\dv\left(\frac{R\Tt u}{2\rho}\right)\right]|\nabla\phi|^2+\left[\p_\tau\left(\frac{\rho}{\T}\right)+\dv\left(\frac{\rho u}{\T}\right)\right]\frac{|\nabla\zeta|^2}{2}dyd\tau\\
&-\int_{0}^{\tau_1}\int_{\Omega}\left[\nabla\kappa(\Tt)\cdot\nabla\zeta\dv(\frac{\nabla\zeta}{\T})+R(\nabla{\Tt} {\phi}+ \nabla \rho{\zeta}) \cdot \nabla \dv{\psi}-\kappa({\Tt})\frac{\nabla\zeta\cdot\nabla\theta\Delta\zeta}{\T^2}\right]\\
&+\left[2\nabla\mu(\Tt)\cdot\mathbb{D}\psi+\nabla\lambda(\Tt)\cdot\dv\psi\mathbb{I}\right]\cdot\nabla^2\psi dyd\tau:=\sum_{i=1}^{3}\mathcal{H}_{1i},\\
\mathcal{Q}_1:=&\int_{0}^{\tau_1}\int_{\Omega}\frac{R{\Tt}}{\rho}\nabla\phi\cdot\hat{Q}_0+\nabla\psi\cdot\hat{Q}_1+\frac{1}{\T}\nabla\zeta\cdot\hat{Q}_2dyd\tau=:\mathcal{Q}_{11}+\mathcal{Q}_{12}+\mathcal{Q}_{13}.
\end{align*}
By the priori assumptions \eqref{eq09_1}-\eqref{eq09_2} and 
G-N inequality \eqref{eq46i}$_1$. Choosing $\eta_0$ suitably small, these terms can be estimated as follows

\begin{align}\label{H11}
\begin{aligned}
\mathcal{H}_{11}
&\leq 
\int_{0}^{\tau_1}\int_{\Omega}(\frac{\thetat_{\tau}+u\cdot\nabla\thetat}{\rho}+\frac{\Tt}{\rho}\dv u)|\nabla\phi|^2+\frac{1}{\T^2}\Big[(\gamma-1)\rho\T\dv u\\
&\qquad\qquad-\dv(\kappa(\T)\nabla\T)-u\cdot\dv \mathcal{T}\Big]|\nabla\zeta|^2dyd\tau\\
 &\le\e^{1-a}\nu^{-(\alpha+1)(\gamma-1)}\int_0^{\tau_1} \|\sqrt{\frac{\Tb^{\alpha+1}}{\rhob^2}}\nabla\phi\|^2d\tau
+\e^{1-a}\nu^{-\alpha(\gamma-1)}\int_0^{\tau_1}\|{\thetat}^{\frac{\alpha-1}{2}}\nabla\zeta\|^2d\tau\\
&\qquad\qquad+K_{11}+K_{12}+K_{13},
\end{aligned}
\end{align}
where 
\begin{align}
\begin{aligned}
K_{11}&:=\int_0^{\tau_1}\|{\bar\theta}^{-\frac{\alpha}{2}}\|_{\infty}\|\sqrt{\frac{\Tb^{\alpha+1}}{\rhob^2}}\nabla\phi\|\|\nabla\phi\|_{L^4}\|\nabla\psi\|_{L^4}d\tau\\
&\le \e^{a+1-b}{\nu}^{-(\frac{1}{2}+\alpha)(\gamma-1)}
\int_0^{\tau_1}\|\sqrt{\frac{{\thetat}^{\alpha+1}}{\rhob^2}}\nabla\phi\|_{H^1}^2d\tau,\\
K_{12}&:=\int_0^{\tau_1}\|{\thetat}^{-\frac{(1+\alpha)}{2}}\|_\infty\|{\thetat}^{\frac{\alpha-1}{2}}\nabla\zeta\|\|(\nabla\psi,\nabla\zeta)\|_{L^4}^2+\|{\thetat}^{-3}\|_\infty\|\nabla\zeta\|_{L^4}^4d\tau\\
&\leq \frac{1}{160}\int_0^{\tau_1}\|({\thetat}^{\frac{\alpha-1}{2}}\nabla^2\zeta,{\thetat}^{\frac{\alpha}{2}}\nabla^2\psi)\|^2d\tau+\e^{\frac{4}{3}(a+1-b)}\nu^{-\frac{(2+4\alpha)(\gamma-1)}{3}}\int_0^{\tau_1}\|({\thetat}^{\frac{\alpha-1}{2}}\nabla\zeta,{\thetat}^{\frac{\alpha}{2}}\nabla\psi)\|^2d\tau,\\
K_{13}&:=\int_0^{\tau_1}\|{\thetat}^{-2}\|_\infty\|{\thetat}^{\frac{\alpha}{2}}\nabla^2\psi\|\|\nabla\zeta\|_{L^4}^2d\tau\\
&\leq \frac{1}{160}\int_0^{\tau_1} \|{\thetat}^{\frac{\alpha}{2}}\nabla^2\psi\|^2 d\tau+\e^{2(a+1-b)}\nu^{-(4+\alpha)(\gamma-1)}\int_0^{\tau_1}\|{\thetat}^{\frac{\alpha-1}{2}}\nabla\zeta\|^2d\tau.
\end{aligned}
\end{align}
Using \eqref{pe} again, one has 
\begin{align}
\begin{aligned}
\mathcal{H}_{12
	}\leq& \int_{0}^{\tau_1}\int_{\Omega}(\nabla{\Tt} {\phi}+ \nabla \rho{\zeta}) \cdot \nabla \dv{\psi}+\frac{\Tt^{\alpha-1}}{\T}|\nabla\T||\nabla\zeta||\nabla^2\zeta|+\frac{\Tt^{\alpha-1}}{\T^2}|\nabla\Tt||\nabla\T||\nabla\zeta|^2dyd\tau\\	\leq&\e^{1-a}\nu^{-\alpha(\gamma-1)}\int_0^{\tau_1}\int_{\Omega}\p_1\bar u_1(\bar\rho^{\gamma-2}\phi^2+{\bar{\rho}^{2-\gamma}} \zeta^{2})dyd\tau+O(\eta_0)\|(\tilde{\rho}^{\frac{\gamma-2}{2}}\phi,\tilde{\rho}^{\frac{2-\gamma}{2}}\zeta)\|^2\\
&+\e^{2a}\nu^{-(2\alpha+1)(\gamma-1)}\int_0^{\tau_1}\|\sqrt{{\frac{\Tb^{\alpha+1}}{\rhob^2}}}\nabla\phi\|^2d\tau+\frac{1}{160}\int_0^{\tau_1}\|\sqrt{{\thetat}^{\alpha}}\nabla^2\psi\|^2+\|\sqrt{\thetat^{\alpha-1}}\nabla^2\zeta\|^2d\tau\\
&+\e^{\frac{4}{3}(2-(a+b))}\nu^{-(2+\frac{2\alpha}{3})(\gamma-1)}\int_0^{\tau_1}\|\sqrt{\thetat^{\alpha-1}}\nabla\zeta\|^2d\tau,
\end{aligned}
\end{align}
\begin{align}
\begin{aligned}
	\mathcal{H}_{13
	}&\leq \int_{0}^{\tau_1}\int_{\Omega}|\nabla\Tt|\Tt^{\alpha-1}|\nabla\psi||\nabla^2\psi|dyd\tau\\
	&\leq \frac{1}{160}\int_0^{\tau_1}\|{\thetat}^{\frac{\alpha}{2}}\nabla^2\psi\|^2d\tau+\e^{2(1-a)}\nu^{-(\gamma-1)}\int_0^{\tau_1}\|{\thetat}^{\frac{\alpha}{2}}\nabla\psi\|^2d\tau.
\end{aligned}
\end{align}
For $\mathcal{Q}_1$, when  $\|(\frac{\p_1\ub}{\rhob\thetab},\frac{\p_1\ub}{\thetab^2},\frac{\p_1\ub}{\rhob^2})\|_{\infty}\leq 1$, 
we have:
\begin{align}
\begin{aligned}
\mathcal{Q}_{11}&\leq\int_{0}^{\tau_1}\int_{\Omega}\left[|\nabla e_0|+|\nabla u||\nabla\phi|+|\p_i\rho||\nabla\psi|+\left|\pii\left(\psi\cdot\nabla\rhot+\phi\dv\ut\right)\right|\right]\frac{\Tt}{\rho}|\nabla\phi|dyd\tau\\
&\le \e^{\frac{1}{3}}
+\e^{a+1-b}{\nu}^{-(\frac{1}{2}+\alpha)(\gamma-1)}
\bigg(\sup_{\tau\in[0,{\tau_1}]}\|(\tilde{\rho}^{\frac{\gamma-2}{2}}\phi,\tilde{\rho}^{\frac{1}{2}}\psi)\|^2+\int_0^{\tau_1}\|\sqrt{\frac{\Tb^{\alpha+1}}{\rhob^2}}\nabla\phi\|_{H^1}^2d\tau\bigg)\\
&\quad+\e^{1-a}\nu^{-(\alpha+1)(\gamma-1)}\int_0^{\tau_1}\|(\sqrt{{\frac{\Tb^{\alpha+1}}{\rhob^2}}}\nabla\phi,\thetat^{\frac{\alpha}{2}}\nabla\psi)\|^2d\tau,
\end{aligned}
\end{align}
\begin{align}
\begin{aligned}
\mathcal{Q}_{13}&\leq \int_{0}^{\tau_1}\int_{\Omega}\tilde{R}_4\pii(\frac{|\nabla\zeta|}{\T})+\Big[|\pii\rho\zeta_\tau|+|\pii(\rho u)\cdot\nabla\zeta|+|\pii p\dv\psi|\\
&\qquad+|\pii(\rho\psi\cdot\nabla\Tt+R\rho\dv\ut\zeta)|\Big]\frac{|\nabla\zeta|}{\T}dyd\tau\\
	& \le \e^{\frac{1}{3}}+\e^{2(1-a)}\nu^{-((2\alpha+1)(\gamma-1)+2)}\sup_{\tau\in[0,{\tau_1}]}\|(\tilde{\rho}^{\frac{\gamma-2}{2}}\phi,\tilde{\rho}^{\frac{1}{2}}\psi,\tilde{\rho}^{\frac{2-\gamma}{2}}\zeta)\|^2+\frac{1}{160}\sup_{\tau\in[0,{\tau_1}]}\|\sqrt{\frac{\rhot}{\thetat}}\nabla\zeta\|^2\\
 &\qquad+\e^{\frac{4}{3}(a+1-b)}\nu^{-\frac{(6+4\alpha)(\gamma-1)}{3}}\int_0^{\tau_1}\|({\thetat}^{\frac{\alpha-1}{2}}\nabla\zeta,{\thetat}^{\frac{\alpha}{2}}\nabla\psi,\sqrt{\frac{\Tb^{\alpha+1}}{\rhob^2}}\nabla\phi)\|^2d\tau\\
 &\qquad+\frac{1}{160}\int_0^{\tau_1}\|({\thetat}^{\frac{\alpha-1}{2}}\nabla^2\zeta,{\thetat}^{\frac{\alpha}{2}}\nabla^2\psi)\|^2d\tau+\nu^{-(\gamma-1)}\e^{1-2a}.
\end{aligned}
\end{align}

The estimates of $\mathcal{Q}_{12}$  are similar as $\mathcal{Q}_{13}$. 
Thus by \eqref{basic-energy01}, 
we obtain
\begin{align}\label{first1}
\begin{aligned}
\int_\Omega&\left[\frac{\bar\T}{\bar\rho}|\nabla\phi|^2+
\bar\rho|\nabla\psi|^2
+\frac{\bar\rho}{\bar\T}|\nabla\zeta|^2\right]dy+\int_0^{\tau_1}\int_{\Omega}
\bar\theta^\alpha
|\nabla^2\psi|^2+\bar\theta^{\alpha-1}
|\nabla^2\zeta|^2dyd\tau\\
&\le  \e^{(\gamma-1)Za}
\int_0^{\tau_1}\|\sqrt{\frac{{\thetat}^{\alpha+1}}{\rhob^2}}\nabla\phi\|_{H^1}^2d\tau+\frac{1}{16}\int_0^{\tau_1}\|(\thetat^{\frac{\alpha}{2}}\nabla\psi,{\thetat}^{\frac{\alpha-1}{2}}\nabla\zeta)\|^2d\tau\\
 &\qquad+\frac{1}{16}\sup_{\tau\in[0,{\tau_1}]}\|(\tilde{\rho}^{\frac{\gamma-2}{2}}\phi,\tilde{\rho}^{\frac{1}{2}}\psi,\tilde{\rho}^{\frac{2-\gamma}{2}}\zeta)\|^2+\nu^{-(\gamma-1)}\e^{1-2a}
 .\\
\end{aligned}
\end{align}


{\bf Step 4. (The second-order estimates.)} 
Applying $\p_j:=\p_{y_j}$, $j=1,2,3$, to \cref{sys-perturb1}, we have
\begin{align}\label{sys-perturb2}
	\begin{cases}
		\p_{ij}\phi_{\tau}+u \cdot \nabla\p_{ij}\phi+\rho \operatorname{div}\p_{ij} \psi=\check{Q}_0, \\
		\rho\pij\psi_{\tau}+\rho u \cdot \nabla\p_{ij} \psi+\nabla\p_{ij}(p-\tilde{p})=\dv\p_{ij}\left(2\mu(\Tt)\mathbb{D}(\psi)+\lambda(\Tt)\dv\psi \mathbb{I}\right)+\check{Q}_1,\\
		\rho\p_{ij}\zeta_{\tau}+\rho u \cdot \nabla\p_{ij} \zeta+p \operatorname{div}\p_{ij} \psi=\dv\p_{ij}(\kappa(\thetat)\nabla \zeta)+\check{Q}_2, 
	\end{cases}
	\end{align} 
where
\begin{align*}
	\check{Q}_0:=&\p_j \hat{Q}_0-\p_j u\cdot\nabla\pii\phi-\p_j\rho\dv\pii\psi,\\
	\check{Q}_1:=&\p_j\hat{Q}_1-\p_j\rho\pii\psi_\tau-\p_j(\rho u)\cdot\nabla\pii\psi,\\
	\check{Q}_2:=&\p_j\hat{Q}_2-\p_j\rho\pii\zeta_\tau-\p_j(\rho u)\cdot\nabla\pii\zeta-\p_i p\dv\pii\psi.
\end{align*}
Multiplying \cref{sys-perturb2}$_1$ by $\frac{R{\Tt}}{\rho}\pij\phi$, \cref{sys-perturb2}$_2$ by $\pij\psi$, \cref{sys-perturb2}$_3$ by $\frac{1}{\T}\pij\zeta$, and adding the results up, one has
\begin{align*}
	&\p_\tau\left[\frac{R\Tt}{2\rho}|\nabla^2\phi|^2+\frac{\rho|\nabla^2\psi|^2}{2}+\frac{\rho}{2\T}|\nabla^2\zeta|^2\right]+\mu(\Tt)|\nabla\Delta\psi|^2+(\mu(\Tt)+\lambda(\Tt))|\nabla^2\dv\psi|^2+\frac{\kappa(\Tt)}{\T}|\nabla\Delta\zeta|^2\\
	&-\left[\p_\tau\left(\frac{R\Tt}{2\rho}\right)+\dv\left(\frac{R\Tt u}{2\rho}\right)\right]|\nabla^2\phi|^2-\left[\p_\tau\left(\frac{\rho}{\T}\right)+\dv\left(\frac{\rho u}{\T}\right)\right]\frac{|\nabla^2\zeta|^2}{2}+\dv(\cdots)-\frac{\kappa(\Tt)}{\T^2}\nabla\Delta\zeta\cdot\nabla\T\nabla^2\zeta\\
	&-R\Big[\nabla\cdot(\nabla\Tt\phi+\nabla\phi\zeta)+(\nabla\Tt\cdot\nabla\phi+\nabla\rho\cdot\nabla\zeta)\Big]\nabla^2\dv\psi+\nabla\left[2\nabla\mu(\Tt)\cdot\mathbb{D}\psi+\nabla\lambda(\Tt)\cdot\dv\psi\mathbb{I}\right]\cdot\nabla\Delta\psi\\
	&+\sum_{i=1}^{3}\left[2\pii\mu(\Tt)\cdot\dv\mathbb{D}\psi+\pii\lambda(\Tt)\cdot\dv(\dv\psi\mathbb{I})\right]\cdot\pii\Delta\psi+\left[\nabla\left(\nabla\kappa(\Tt)\cdot\nabla\zeta\right)+\Delta\zeta\nabla\kappa(\Tt)\right]\cdot\dv\left(\frac{\nabla^2\zeta}{\T}\right)\\
	&-\frac{R{\Tt}}{\rho}\nabla^2\phi\cdot\check{Q}_0-\nabla^2\psi\cdot\check{Q}_1-\frac{1}{\T}\nabla^2\zeta\cdot\check{Q}_2=0.
\end{align*}
	Integrating the above equation over $\Omega\times[0,{\tau_1}]$, one has
\begin{align*}
	\begin{aligned}
		\mathcal{I}_2 +\int_0^{\tau_1}\mathcal{D}_2d\tau\leq \mathcal{H}_2+\mathcal{Q}_2,
	\end{aligned}
\end{align*}
where

\begin{align*}
&\mathcal{I}_2=\int_{\Omega}\frac{R\Tt}{2\rho}|\nabla^2\phi|^2+\frac{\rho|\nabla^2\psi|^2}{2}+\frac{\rho}{2\T}|\nabla^2\zeta|^2dy,\ \ \ \ \mathcal{D}_2=\int_{\Omega}\Tt^{\alpha}|\nabla^3\psi|^2+\Tt^{\alpha-1}|\nabla\Delta\zeta|^2dy,
\end{align*}

\begin{align*}
	\mathcal{H}_2:=&\int_{0}^{\tau_1}\int_{\Omega}\left[\p_\tau\left(\frac{R\Tt}{2\rho}\right)+\dv\left(\frac{R\Tt u}{2\rho}\right)\right]|\nabla^2\phi|^2+\left[\p_\tau\left(\frac{\rho}{\T}\right)+\dv\left(\frac{\rho u}{\T}\right)\right]\frac{|\nabla^2\zeta|^2}{2}dyd\tau\\
	&-\int_{0}^{\tau_1}\int_{\Omega}\nabla\left[2\nabla\mu(\Tt)\cdot\mathbb{D}\psi+\nabla\lambda(\Tt)\cdot\dv\psi\mathbb{I}\right]\cdot\nabla\Delta\psi dyd\tau\\
	&+\int_{0}^{\tau_1}\int_{\Omega}R\Big[\nabla\cdot(\nabla\Tt\phi+\nabla\phi\zeta)+(\nabla\Tt\cdot\nabla\phi+\nabla\rho\cdot\nabla\zeta)\Big]\nabla^2\dv\psi dyd\tau\\
	&-\int_{0}^{\tau_1}\int_{\Omega}\sum_{i=1}^{3}\left[2\pii\mu(\Tt)\cdot\dv\mathbb{D}\psi+\pii\lambda(\Tt)\cdot\dv(\dv\psi\mathbb{I})\right]\cdot\pii\Delta\psi dyd\tau\\
	&-\int_{0}^{\tau_1}\int_{\Omega}\left[\nabla\left(\nabla\kappa(\Tt)\cdot\nabla\zeta\right)+\Delta\zeta\nabla\kappa(\Tt)\right]\cdot\dv\left(\frac{\nabla^2\zeta}{\T}\right)-\frac{\kappa(\Tt)}{\T^2}\nabla\Delta\zeta\cdot\nabla\T\nabla^2\zeta dyd\tau:=\sum_{i=1}^5\mathcal{H}_{2i},\\
	\mathcal{Q}_2:=&\int_{0}^{\tau_1}\int_{\Omega}\frac{R{\Tt}}{\rho}\nabla^2\phi\cdot\check{Q}_0+\nabla^2\psi\cdot\check{Q}_1+\frac{1}{\T}\nabla^2\zeta\cdot\check{Q}_2dyd\tau=:\mathcal{Q}_{21}+\mathcal{Q}_{22}+\mathcal{Q}_{23}.
\end{align*}
Similar to before, we could also get 
\begin{align}\label{H21}
\begin{aligned}
\mathcal{H}_{21}
 &\le \e^{1-a}\nu^{-(\alpha+1)(\gamma-1)}\int_0^{\tau_1}\|\sqrt{\frac{\Tb^{\alpha+1}}{\rhob^2}}\nabla^2\phi\|^2d\tau
 +\frac{1}{160}\int_0^{\tau_1}\|(\thetat^{\frac{\alpha}{2}}\nabla^3\psi,\thetat^{\frac{\alpha}{2}}\nabla^3\zeta)\|^2d\tau\\
 &\quad+\e^{(a+1-b)}\nu^{-(2+\alpha)(\gamma-1)}\int_0^{\tau_1}\|(\thetat^{\frac{\alpha}{2}}\nabla^2\psi,\thetat^{\frac{\alpha}{2}}\nabla^2\zeta,\thetat^{\frac{\alpha+1}{2}}\nabla^2\phi)\|^2d\tau,
\end{aligned}
\end{align}
where some terms in $\mathcal{H}_{21}$ 
 can be estimated by follows:
\begin{align}
\begin{aligned}
\int_0^{\tau_1}&\|{\thetat}^{-1}\|_{\infty}\|\nabla\psi\|_{\infty}\|(\nabla^2\psi,\nabla^2\phi)\|^2+\|({\thetat}^{-3},(\rhot\thetat)^{-2})\|_{\infty}\|\nabla\psi\|_{\infty}^2\|(\nabla^2\psi,\nabla^2\phi)\|^2d\tau\\
&+\int_0^{\tau_1}\|{\thetat}^{-2}\|_{\infty}\|\nabla^2\zeta\|\|(\nabla^2\psi,\nabla^2\zeta)\|_{L^4}^2d\tau
\le \frac{1}{160}\int_0^{\tau_1}\|(\thetat^{\frac{\alpha}{2}}\nabla^3\psi,\thetat^{\frac{\alpha}{2}}\nabla^3\zeta)\|^2d\tau\\
&\qquad\qquad\qquad\qquad\qquad\qquad+\e^{(a+1-b)}\nu^{-(2+\alpha)(\gamma-1)}\int_0^{\tau_1}\|(\thetat^{\frac{\alpha}{2}}\nabla^2\psi,\thetat^{\frac{\alpha}{2}}\nabla^2\zeta,\thetat^{\frac{\alpha+1}{2}}\nabla^2\phi)\|^2d\tau.
\end{aligned}
\end{align}
And we get
\begin{align}
\begin{aligned}
\mathcal{H}_{22}\leq&\int_{0}^{\tau_1}\int_{\Omega}\nabla\left[2\nabla\mu(\Tt)\cdot\mathbb{D}\psi+\nabla\lambda(\Tt)\cdot\dv\psi\mathbb{I}\right]\cdot\nabla\Delta\psi dyd\tau\\
\leq &\e^{1-a}\nu^{-(\gamma-1)}\int_0^{\tau_1}\|({\thetat}^{\frac{\alpha}{2}}\nabla\psi,{\thetat}^{\frac{\alpha}{2}}\nabla^2\psi,{\thetat}^{\frac{\alpha}{2}}\nabla^3\psi)\|^2d\tau,\\
\mathcal{H}_{23}\leq&\int_{0}^{\tau_1}\int_{\Omega}R\Big[\nabla\cdot(\nabla\Tt\phi+\nabla\phi\zeta)+(\nabla\Tt\cdot\nabla\phi+\nabla\rho\cdot\nabla\zeta)\Big]\nabla^2\dv\psi dyd\tau\\
\leq &\e^{2(a+1-b)}\nu^{-(2\alpha+1)(\gamma-1)}\bigg(\sup_{\tau\in[0,{\tau_1}]}\|\rhot^{\frac{\gamma-2}{2}}\phi\|^2+\int_0^{\tau_1}\|(\sqrt{\frac{\thetat^{\alpha+1}}{\rhot^2}}\nabla\phi,\thetat^{\frac{\alpha-1}{2}}\nabla\zeta)\|_{H^1}^2d\tau\bigg)\\
&+\frac{1}{160}\int_0^{\tau_1}\|{\thetat}^{\frac{\alpha}{2}}\nabla^3\psi\|^2d\tau,\\
\mathcal{H}_{25}\leq& \frac{1}{160}\int_0^{\tau_1}\|\thetat^{\frac{\alpha-1}{2}}\nabla^3\zeta\|^2d\tau+\e^{(a+1-b)}\nu^{-(1+\alpha)(\gamma-1)}\int_0^{\tau_1}\|(\thetat^{\frac{\alpha-1}{2}}\nabla\zeta,\thetat^{\frac{\alpha-1}{2}}\nabla^2\zeta)\|^2d\tau.
\end{aligned}
\end{align}
The estimates of $\mathcal{H}_{24}$  are similar as $\mathcal{H}_{22}$.
Also, we have 
\begin{align}
\begin{aligned}
\mathcal{Q}_{21}
&\le \e^{\frac{1}{3}}+\int_0^{\tau_1}\|\rhot^{-1}\|_{\infty}\|\nabla\psi\|_{\infty}\|\nabla^2\phi\|^2d\tau+{\e^{1-a}\nu^{-((\alpha+2)(\gamma-1)+3)}}\sup_{\tau\in[0,\tau_1]}\|(\tilde{\rho}^{\frac{\gamma-2}{2}}\phi,\tilde{\rho}^{\frac{1}{2}}\psi)\|^2\\
&+\e^{a+1-b}\nu^{-1-(\alpha+2)(\gamma-1)}\bigg(\int_0^{\tau_1}\|(\thetat^{\frac{\alpha}{2}}\nabla\psi,\thetat^{\frac{\alpha}{2}}\nabla^2\psi,\sqrt{{\frac{\Tb^{\alpha+1}}{\rhob^2}}}\nabla\phi,\sqrt{{\frac{\Tb^{\alpha+1}}{\rhob^2}}}\nabla^2\phi)\|^2d\tau\bigg)\\
&+\frac{1}{160}\int_0^{\tau_1}\|\thetat^{\frac{\alpha}{2}}\nabla^3\psi\|^2d\tau,
\end{aligned}
\end{align}
\begin{align}\label{Q22}
\begin{aligned}
\mathcal{Q}_{22}&\leq \int_{0}^{\tau_1}\int_{\Omega}\nabla^2\psi\bigg\{\pii^2\tilde{R}+\pii^2\rho\psi_{\tau}+\p_i\rho\nabla\psi_{\tau}+\pii^2(\rho u)\cdot\nabla\psi+\pii(\rho u)\cdot\nabla^2\psi\\
&\qquad\ +\pii^2\Big(\rho\psi \cdot \nabla \tilde{u}+\frac{\phi}{\tilde{\rho}}\nabla\tilde{p}\Big)\bigg\}dyd\tau\\
&\ \le \e^{\frac{1}{3}}+\frac{1}{160}\big(\sup_{\tau\in[0,\tau_1]}\|\rhot^{\frac{1}{2}}\nabla^2\psi\|^2+\int_0^{\tau_1}\|(\sqrt{\thetat^{\alpha}}\nabla^3\psi,\thetat^{\frac{\alpha}{2}}\nabla^3\zeta)\|^2d\tau\big)
\\
&\qquad+
\e^{a+1-b}\nu^{-1-(\alpha+2)(\gamma-1)}\bigg(\sup_{\tau\in[0,\tau_1]}\|(\sqrt{\frac{\thetab}{\rhob}}\phi,\sqrt{\rhob}\psi,\sqrt{\frac{\rhob}{\thetab}}\zeta)\|^2\\
&\qquad\ +\int_0^{\tau_1}\|(\thetat^{\frac{\alpha}{2}}\nabla\psi,\thetat^{\frac{\alpha-1}{2}}\nabla\zeta,\sqrt{\frac{\thetat^{\alpha+1}}{\rhot^2}}\nabla\phi)\|_{H^1}^2d\tau\bigg),
\end{aligned}
\end{align}
where some terms in $\mathcal{Q}_{22}$ 
 can be estimated by follows:
\begin{align}
\begin{aligned}
\int_0^{\tau_1}&\nu^{-\gamma}\|\nabla^2\phi\|\|\nabla^2\psi\|_{L^4}^2+\nu^{-\gamma}\|\nabla^2\phi\|\|\nabla\phi\|_{L^4}^2+\|\thetat^{-1}\|_{\infty}\|\nabla\phi\|\|\nabla^2\psi\|_{L^4}^2\\
&+\|\rhot^{-2}\|_{\infty}\|\nabla^2\psi\|\|\nabla\phi\|_{L^6}^3+\|(\rhot^{-2}\thetat^{-3})\|_{\infty}\|\nabla\phi\|_{L^6}^2(\|(\nabla\psi,\nabla\zeta)\|_{L^6}^4)d\tau\\
\le &\frac{1}{160}\int_0^{\tau_1}\|\thetat^{\frac{\alpha}{2}}\nabla^3\psi\|^2d\tau+\e^{a+1-b}\nu^{-1-(\alpha+2)(\gamma-1)}\int_0^{\tau_1}\|(\thetat^{\frac{\alpha}{2}}\nabla^2\psi,\thetat^{\frac{\alpha+1}{2}}\nabla\phi,\thetat^{\frac{\alpha+1}{2}}\nabla^2\phi)\|^2d\tau,\\
\int_0^{\tau_1}&\|\thetat^{-5}\|_{\infty}\|\nabla\psi\|_{L^6}^6+\|\thetat^{-2}\|_{\infty}\|(\nabla\psi,\nabla\zeta)\|_{L^4}^4+\e^{4(1-a)}\nu^{-4\gamma}\|\nabla\phi\|_{L^4}^4\\
&\qquad+\|(\rhot^{2}\thetat)^{-1}\|_{\infty}\|\nabla\phi\|_{L^4}^2\|\nabla^2\psi\|_{L^4}^2d\tau\\
\le &\frac{1}{160}\int_0^{\tau_1}\|\thetat^{\frac{\alpha}{2}}\nabla^3\psi\|^2d\tau+\e^{2(a+1-b)}\nu^{-(2+\alpha)(\gamma-1)}\int_0^{\tau_1}\|(\thetat^{\frac{\alpha}{2}}\nabla\zeta,\thetat^{\frac{\alpha}{2}}\nabla\psi,\thetat^{\frac{\alpha+1}{2}}\nabla\phi)\|_{H^1}^2d\tau.
\end{aligned}
\end{align}
The estimates of $\mathcal{Q}_{23}$  are similar as $\mathcal{Q}_{22}.$ Thus, by \eqref{basic-energy01}, \eqref{H21}-\eqref{Q22}, we obtain
\begin{align}\label{second1}
\begin{aligned}
&\int_\Omega\left[\frac{\bar\T}{\bar\rho}|\nabla^2\phi|^2+
\bar\rho|\nabla^2\psi|^2
+\frac{\bar\rho}{\bar\T}|\nabla^2\zeta|^2\right]dy+\int_0^{\tau_1}\int_{\Omega}
\bar\theta^\alpha
|\nabla^3\psi|^2+\bar\theta^{\alpha-1}
|\nabla^3\zeta|^2dyd\tau\\
\le 
&\e^{\frac{1}{3}}+\e^{(1-a)}\nu^{-3\gamma}+\frac{1}{16}\int_0^{\tau_1}\|(\thetat^{\frac{\alpha+1}{2}}\nabla\phi,\thetat^{\frac{\alpha}{2}}\nabla\psi)\|_{H^1}^2\\
&+\frac{1}{16}\bigg(\sup_{\tau\in[0,{\tau_1}]}\|(\sqrt{\frac{\thetab}{\rhob}}\phi,\sqrt{\rhob}\psi,\sqrt{\frac{\rhob}{\thetab}}\zeta)\|^2+\int_0^{\tau_1}\|(\thetat^{\frac{\alpha}{2}}\nabla\psi,\thetat^{\frac{\alpha-1}{2}}\nabla\zeta,\sqrt{\frac{\thetat^{\alpha+1}}{\rhot^2}}\nabla\phi)\|_{H^1}^2d\tau\bigg).
\end{aligned}
\end{align}

Similar to Step 2, we could also get the higher order estimates for $\|\nabla^2\phi\|.$  We omit the details but only show the final results:
\begin{align}\label{eqj0223}
	\begin{aligned}
	\sup_{\tau\in[0,{\tau_1}]}&\int_{\Omega}\left(\frac{\Tb^{2\alpha}}{
	\bar\rho^3}|\nabla^2\phi|^2+
	\frac{\Tb^\alpha\nabla\psi\nabla^2\phi}{\bar\rho}\right)(\tau)dy
	+
	\int_0^{\tau_1}\int_{\Omega}\frac{{\Tb}^{\alpha+1}}{\bar\rho^2}|\nabla^2\phi|^2dyd\tau\\
	&\leq \e^{\frac{1}{3}}
	+\frac{1}{160}\int_0^{\tau_1}\int_{\Omega}
	\bar\theta^\alpha|\nabla^3\psi|^2+\Tb^{\alpha-1}|\nabla^3\zeta|^2dyd\tau\\
&\qquad+\nu^{-(\gamma-1)}\int_0^{\tau_1}\|(\thetat^{\frac{\alpha}{2}}\nabla\psi,\thetat^{\frac{\alpha-1}{2}}\nabla\zeta)\|_{H^1}^2+\|\sqrt{\frac{\thetat^{\alpha+1}}{\rhot^2}}\nabla\phi\|^2d\tau.	
\end{aligned}
\end{align}

Combining \eqref{basic-energy01},\eqref{eqj0123},\eqref{first1}, and \eqref{second1}-\eqref{eqj0223}, we finally get 
\begin{align}\label{step4}
	\begin{aligned}
		&\sup _{\tau \in[0,{\tau_1})} \int_{\Omega}\left({{\rhob}}^{\gamma-2} \phi^{2}+{{\rhob}} \psi^{2}+{{\rhob}}^{2-\gamma} \zeta^{2}+\left[\frac{\bar\T}{\bar\rho}|\nabla\phi|^2+
\bar\rho|\nabla\psi|^2
+\frac{\bar\rho}{\bar\T}|\nabla\zeta|^2\right]\right)(\tau, y) d y \\
&+\sup _{\tau \in[0,{\tau_1})}\int_\Omega\left[\frac{\bar\T}{\bar\rho}|\nabla^2\phi|^2+
\bar\rho|\nabla^2\psi|^2
+\frac{\bar\rho}{\bar\T}|\nabla^2\zeta|^2\right]dy+\int_{0}^{{\tau_1}} \int_{\Omega}\px\bar{u}\left({{\rhob}}^{\gamma-2} \phi^{2}+{{\rhob}} \psi^{2}+{{\rhob}}^{2-\gamma} \zeta^{2}\right)dyd\tau\\
 & +\int_{0}^{{\tau_1}} \int_{\Omega}{{\thetab}}^{\alpha} |\nabla\psi|^{2}+{\thetab}^{\alpha-1} |\nabla\zeta|^{2}+\bar\theta^\alpha
|\nabla^2\psi|^2+\bar\theta^{\alpha-1}
|\nabla^2\zeta|^2+\frac{{\Tb}^{\alpha+1}}{\bar\rho^2}|\nabla\phi|^2+\frac{{\Tb}^{\alpha+1}}{\bar\rho^2}|\nabla^2\phi|^2dyd\tau\\
&+\int_0^{\tau_1}\int_{\Omega}
\bar\theta^\alpha
|\nabla^3\psi|^2+\bar\theta^{\alpha-1}
|\nabla^3\zeta|^2dyd\tau\le 
\e^{\frac{1}{3}}\nu^{-2(\gamma-1)}.
\end{aligned}
\end{align}

The higher-order derivative estimate of $(\phi,\psi,\zeta)$ in \eqref{sys-perturb} 
 is similar to that steps 2 and 3, we omit it for brevity.
Therefore, \eqref{eqe01}-\eqref{eqe04} can be derived from \eqref{basic-energy01}, \eqref{eqj0123}, \eqref{first1} and \eqref{step4}. Then by the initial perturbation \eqref{eqe095} and the Sobolev inequality in 3D, we have
\begin{align}\label{eqe0021}
\begin{aligned}
&\sup_{\tau\in[0,\tau_1(\e)]}\|\phi(x,t)\|_{L^\infty(\Omega)}\\
&\le C\sup_{\tau\in[0,\tau_1(\e))}\Big(\e^{-(b-1)}\|\phi\|^{\frac{1}{2}}\|\nabla\phi\|^{\frac{1}{2}}+\e^{-\frac{(b-1)}{2}}\|\nabla\phi\|+\|\phi\|^{\frac{1}{4}}\|\nabla^2\phi\|^{\frac{3}{4}}\Big)\\
&\le C\nu^{-\frac{\gamma}{2}}\e^{-(b-1)}\sup_{\tau\in[0,\tau_1(\e)]}\big(\|\sqrt{\bar\rho^{\gamma-2}}\phi\|\|\sqrt{\frac{\bar\theta}{\bar\rho}}\nabla\phi\|\big)^{\frac{1}{2}}+C\nu^{-\frac{\gamma}{2}}\e^{-\frac{(b-1)}{2}}\sup_{\tau\in[0,\tau_1(\e)]}\big(
\|\sqrt{\frac{\bar\theta}{\bar\rho}}\nabla\phi\|\big)\\
&\qquad+C\nu^{-\frac{\gamma}{2}}\sup_{\tau\in[0,\tau_1(\e)]}\big(\|\sqrt{\bar\rho^{\gamma-2}}\phi\|^{\frac{1}{4}}\|\sqrt{\frac{\bar\theta}{\bar\rho}}\nabla^2\phi\|^{\frac{3}{4}}\big)
\\
&\le C\big(\e^{\frac{1}{6}+1-b-\frac{3\gamma-1}{4}Za}+\e^{\frac{1}{6}+\frac{1-b}{2}-\frac{2\gamma-1}{2}Za}+\e^{\frac{1}{6}-\frac{5\gamma-3}{4}Za}\big)\le \e^{a},
\end{aligned}
\end{align}
\begin{align}\label{eqe0020} 
\begin{aligned}
&\sup_{\tau\in[0,\tau_1(\e)]}\|\psi(x,t)\|_{L^\infty(\Omega)}\\
&\le C\sup_{\tau\in[0,\tau_1(\e))}\Big(\e^{-(b-1)}\|\psi\|^{\frac{1}{2}}\|\nabla\psi\|^{\frac{1}{2}}+\e^{-\frac{(b-1)}{2}}\|\nabla\psi\|+\|\psi\|^{\frac{1}{4}}\|\nabla^2\psi\|^{\frac{3}{4}}\Big)\\
&\le C\nu^{-\frac{1}{2}}\e^{-(b-1)}\sup_{\tau\in[0,\tau_1(\e)]}\big(\|\sqrt{\bar\rho}\psi\|\|\sqrt{\bar\rho}\nabla\psi\|\big)^{\frac{1}{2}}+C\nu^{-\frac{1}{2}}\e^{-\frac{(b-1)}{2}}\sup_{\tau\in[0,\tau_1(\e)]}\big(\|\sqrt{\bar\rho}\nabla\psi\|\big)\\
&\qquad+C\nu^{-\frac{1}{2}}\sup_{\tau\in[0,\tau_1(\e)]}\big(\|\sqrt{\bar\rho}\psi\|^{\frac{1}{4}}\|\sqrt{\bar\rho}\nabla^2\psi\|^{\frac{3}{4}}\big)
\\
&\le C\big(\e^{\frac{1}{6}+1-b-\frac{\gamma+1}{4}Za}+\e^{\frac{1}{6}+\frac{1-b}{2}-\frac{\gamma}{2}Za}+\e^{\frac{1}{6}-\frac{3\gamma-1}{4}Za}\big)\le \e^{a},
\end{aligned}
\end{align}
\begin{align}\label{eqe0022}
\begin{aligned}
&\sup_{\tau\in[0,\tau_1(\e)]}\|\zeta(x,t)\|_{L^\infty(\Omega)}\\
&\le C\sup_{\tau\in[0,\tau_1(\e))}\Big(\e^{-(b-1)}\|\zeta\|^{\frac{1}{2}}\|\nabla\zeta\|^{\frac{1}{2}}+\e^{-\frac{(b-1)}{2}}\|\nabla\zeta\|+\|\zeta\|^{\frac{1}{4}}\|\nabla^2\zeta\|^{\frac{3}{4}}\Big)\\
&\le C\nu^{-\frac{3}{4}}\e^{-(b-1)}\sup_{\tau\in[0,\tau_1(\e)]}\big(\|\sqrt{\bar\rho^{2-\gamma}}\zeta\|\|\sqrt{\frac{\bar\rho}{\bar\theta}}\nabla\zeta\|\big)^{\frac{1}{2}}+C\nu^{-\frac{1}{2}}\e^{-\frac{(b-1)}{2}}\sup_{\tau\in[0,\tau_1(\e)]}\big(\|\sqrt{\frac{\bar\rho}{\bar\theta}}\nabla\zeta\|\big)\\
&\qquad+C\sup_{\tau\in[0,\tau_1(\e)]}\big(\nu^{-\frac{5}{8}}\|\sqrt{\bar\rho^{2-\gamma}}\zeta\|^{\frac{1}{4}}\|\sqrt{\frac{\bar\rho}{\bar\theta}}\nabla^2\zeta\|^{\frac{3}{4}}\big)
\\
&\le C\big(\e^{\frac{1}{6}+1-b-\frac{(\gamma+2)}{4}Za}+\e^{\frac{1}{6}+\frac{1-b}{2}-\frac{\gamma}{2}Za}+\e^{\frac{1}{6}-(\frac{6\gamma-1}{8})Za}\big)\le \e^{(\gamma-1)a},
\end{aligned}
\end{align}
where $\e\ll1$
. Thus the a priori assumptions \eqref{eq09_1}-\eqref{eq09_2} are verified and the proof of the Lemma \ref{ape1} is finished.
\end{proof}

Since the a priori estimates \eqref{eqe01}-\eqref{eqe04} are better than the a priori assumptions \eqref{eq09_1}-\eqref{eq09_2} 
in the time interval $[0,\tau_1(\e)]$  with $\tau_1(\e)$ being the maximum existence time, we can 
 extend the local solution to the global one in $[0,+\infty)$ for small but fixed $\e$ by a standard continuation argument. 

\subsection{Proof of 
 Theorem \ref{mt}.}
 Once the global existence and uniform estimates 
 are proved, it remains to prove \eqref{eqa} with $a$ and $Z$ given in \eqref{eqda}.
From \eqref{nrw6} in Lemma \ref{sr1}, \eqref{eqsr} in Lemma \ref{sr}
, for any given positive constant $h$, there exists a constant $C_h>0$ which is independent of $\e$ {such that }
\begin{align}\label{eqes001}
\begin{aligned}
\sup_{t\ge h}\|\rho(x,t)-\rho^r(\frac{x_1}{t})\|_{L^\infty(\Omega_\e)}
&\le 
\sup_{t
\in[0,+\infty]
}\|\phi(x,t)\|_\infty
+\sup_{t\ge h}\|\rhob(x_1,t)-\rho_{\nu}^{r}(\frac{x_1}{t})\|_\infty
\\
&\qquad+\sup_{t\ge h}
\|\rho_{\nu}^{r}(\frac{x_1}{t})-\rho^r(\frac{x_1}{t})\|_\infty
\\
&\le C_h\big( \e^a+ \e^{a}|\log\e|+\nu\big)\le C_h\e^{Za}|\log\e|,
\end{aligned}
\end{align}
\begin{align}\label{eqes002}
\begin{aligned}
\sup_{t\ge h}\|m(x,t)-m^r(\frac{x_1}{t})\|_{L^\infty(\Omega_\e)}
&\le 
C\sup_{t
\in[0,+\infty]
}\Big(\|\phi(x,t)\|_\infty+\|\psi(x,t)\|_\infty\Big)
\\
&\qquad+\sup_{t\ge h}\|\bar m(x_1,t)-m_{\nu}^{r}(\frac{x_1}{t})\|_\infty
+\sup_{t\ge h}\|m_{\nu}^{r}(\frac{x_1}{t})-m^r(\frac{x_1}{t})\|_\infty
\\
&\le C_h\big( \e^a+ \e^{a}|\log\e|+\nu\big)\le C_h\e^{Za}|\log\e|,
\end{aligned}
\end{align}
\begin{align}\label{eqes003}
\begin{aligned}
\sup_{t\ge h}\|n(x,t)-n^r(\frac{x_1}{t})\|_{L^\infty(\Omega_\e)}
&\le 
C\sup_{t
\in[0,+\infty]
}\Big(\|\phi(x,t)\|_\infty+\|\zeta(x,t)\|_\infty\Big)
\\
&\qquad+\sup_{t\ge h}\|\bar \rho\bar\T(x_1,t)-n_{\nu}^{r}(\frac{x_1}{t})\|_\infty
+\sup_{t\ge h}\|n_{\nu}^{r}(\frac{x_1}{t})-n^r(\frac{x_1}{t})\|_\infty
\\
&\le C_h\big( \e^a+ \e^{a}|\log\e|+\nu\big)\le C_h\e^{Za}|\log\e|,
\end{aligned}
\end{align}
where we use the following fact  
\begin{align}
\sup_{\tau\in[0,+\infty]}\|\zeta(y,\tau)\|_\infty\le \e^a.
\end{align}
Then the proof of Theorem \ref{mt} is completed.

\section{{Exponential decay rate}
}
Now we prove the exponential decay rate 
 for the non-zero mode of solution $(\rho,u,\T)$  
 to the problem \eqref{NS}. 
For any $T>0$, we take the solution space for \eqref{sys-perturb}-\eqref{R_4} as follows,
\begin{align}
\begin{aligned}
\mathcal{B}(0,T)=\big\{&(\phi,\psi,\zeta)\big|(\phi,\psi,\zeta)\  \text{is periodic in}\ y'=(y_2,y_3)\in\mathbb{T}^2, (\phi,\psi,\zeta)\in C\big(0,T; H^3(\Omega)\big),\\
&
\nabla\phi \in L^2\big(0,T;H^2(\Omega)\big),(\nabla\psi,\nabla\zeta)\in L^2\big(0,T;H^3(\Omega)\big)
\big\},
\end{aligned}
\end{align}
and the initial data
\begin{align}\label{eqins1}
(\phi_0,\psi_0,\zeta_0)(y)\in H^3(\Omega).
\end{align}

We show the following result for the exponential decay rate of the non-zero mode $(\phia,\psia,\zetaa)$.
\begin{Thm}\label{expde}
Under the assumptions of Theorem \ref{mt1}, the Cauchy problem \eqref{sys-perturb}-\eqref{R_4} admits a unique solution $(\phi,\psi,\zeta)\in\mathcal B(0,+\infty)$, satisfying 
\begin{align}\label{eqes5}
\|(\phi,\psi,\zeta)\|_{W^{1,\infty}(\Omega)}\rightarrow 0, \ \text{as}\ \tau\rightarrow+\infty.
\end{align}
The non-zero mode of $(\phi,\psi,\zeta)$ is 
\begin{align}\label{eqe6}
(\phia,\psia,\zetaa)(\tau,y):=(\phi,\psi,\zeta)(\tau,y)-\int_{\mathbb T^2}(\phi,\psi,\zeta)(\tau,y_1,y')dy',
\end{align}
which decays exponentially fast in time, i.e.,  
\begin{align}\label{eqe7}
\|(\phia,\psia,\zetaa)(\tau,\cdot)\|_{L^\infty(\Omega)}\le O(\e^{\frac{1}{3}})
e^{-\e^{n_2}\tau},\ \  \forall \tau\ge
0,
\end{align}
where the positive constant $C>0$ is independent of $\e,\eta
$ and $\tau$.
\end{Thm}

 The proof of Theorem \ref{expde} is based on $L^2$-energy method through  the decomposition idea in which zero and non-zero modes estimate should be  analyzed 
 separately,  
 reference to 
  \cite{Yuan2023SIAM}.
But since the zero mode of any function is independent of the variables $y'=(y_2,y_3)$, we only estimate the non-zero mode as follows. 

Considering the zero mode of $(\phi,\psi,\zeta)$, we apply $\D_0$ to \cref{sys-perturb} to have
\begin{align}\label{sys-0mode}
	\begin{cases}
		\mathring{\phi}_{\tau}+\um_1 \px \phim+\rhom \px \psim_1+\psim_1 \px \mathring{\tilde{\rho}}+\phim \px \mr{\tilde{u}}_1=\hat{E}_0, \\
		\rhom\psim_{\tau}+\rhom \um_1\px \psim+\px(\mr{p}-\mr{\tilde{p}})\mathbb{I}_1+
		\rhom\psim \cdot \nabla \mr{\tilde{u}}-\frac{\phim}{\mr{\tilde{\rho}}}\nabla\mr{\tilde{p}}=\px\big[\mu(\mr{\Tt})\px\psim+(\lambda(\mr{\Tt})+\mu(\mr{\Tt}))\px\psim_1\mathbb{I}_1\big]+\hat{E},\\
		\rhom\zetam_{\tau}+\rhom \um_1 \px \zetam+\mr{p}\px \psim_1+\rhom\psim_1 \px \mr{\tilde{\theta}}+(\gamma-1)\rhom\zetam \px \mr{\tilde{u}}_1 ={\p_1\big(\kappa(\mr{\thetat})\p_1\mr{\zeta}\big)}
		+\hat{E}_4,
	\end{cases}
\end{align}
where
\begin{align*}
	\hat{E}_0:=&\mathring{\phi}_{\tau}+\um \cdot \nabla \phim+\rhom \operatorname{div} \psim+\psim \cdot \nabla \mathring{\tilde{\rho}}+\phim \operatorname{div} \mr{\tilde{u}}\\
	&\qquad-\D_0\left(\phi_{\tau}+u \cdot \nabla \phi+\rho \operatorname{div} \psi+\psi \cdot \nabla \tilde{\rho}+\phi \operatorname{div} \tilde{u}\right)-\D_{0}e_0,\\
	\hat{E}:=&\rhom\psim_{\tau}+\rhom \um \cdot \nabla \psim+\px(\mr{p}-\mr{\tilde{p}})\mathbb{I}_1+
	\rhom\psim \cdot \nabla \mr{\tilde{u}}-\frac{\phim}{\mr{\tilde{\rho}}}\nabla\mr{\tilde{p}}-\operatorname{div}\left(2\mu(\mr{\Tt})\mathbb{D}(\psim)+\lambda(\mr{\Tt})\operatorname{div}\psim \mathbb{I}\right)
	+\D_0\tilde{R}\\
	&-\D_0\left[\rho\psi_{\tau}+\rho u \cdot \nabla \psi+\nabla(p-\tilde{p})+
	\rho\psi \cdot \nabla \tilde{u}-\frac{\phi}{\tilde{\rho}}\nabla\tilde{p}-\dv\left(2\mu(\Tt)\mathbb{D}(\psi)+\lambda(\Tt)\dv\psi \mathbb{I}\right)
	\right],\\
	\hat{E}_4:=&\rhom\zetam_{\tau}+\rhom \um \cdot \nabla \zetam+(\gamma-1) \rhom\thetam \operatorname{div} \psim+\rhom\psim \cdot \nabla \mr{\tilde{\theta}}+(\gamma-1)\rhom \zetam \operatorname{div} \mr{\tilde{u}} -\operatorname{div}\big(\kappa(\mr{\thetat})\nabla \mr{\zeta}\big)\\
	&-\D_0\left[\rho\zeta_{\tau}+\rho u \cdot \nabla \zeta+(\gamma-1) \rho\theta \operatorname{div} \psi+\rho\psi \cdot \nabla \tilde{\theta}+(\gamma-1)\rho \zeta \operatorname{div} \tilde{u} -{\operatorname{div}\big(\kappa(\thetat)\nabla \zeta\big)}
	\right]+\D_0\tilde{R}_4,
\end{align*}
{with 
 $\mr p:=p(\rhom,\thetam)=R\rhom\thetam$, $\mr{\tilde p}=R\mr\rhot\mr\Tt$ and $\tilde R$, $\tilde R_4$ given by $\eqref{R}-\eqref{R_4}$.}

{W}e rewrite \eqref{NS} as the following forms by the zero mode \eqref{sys-0mode},  
\begin{align}\label{eq0108}
\begin{cases}
&\rhom_\tau+\p_1(\rhom\um_{1})=-\p_1\D_0(\rhoa\ua_{1}),\\
&\rhom\um_\tau+\rhom\um_1\p_1\um+\p_1\mr p=\p_1\left[\mu(\thetam)\p_1\um+\big(\mu(\thetam)+\lambda(\thetam)\big)\p_1\um_1\mathbb{I}_1\right]-\grave{Q}_1,\\
&\rhom\thetam_\tau+\rhom\um_1\p_1\thetam+\mr p\p_1\um_1=\p_1\big(\kappa(\thetam)\p_1\thetam\big)+\big(\mu(\thetam)+\lambda(\thetam)\big)(\p_1\um_1)^2+\mu(\thetam)|\p_1\um|^2-\grave{Q_2},
\end{cases}
\end{align}
where 
\begin{align}\label{eq0109}
\begin{aligned}
\grave{Q}_1=&\D_0\big(\rhoa\ua_\tau+\rhoa\ua_1\p_1\ua+\rhoa\ua_1\p_1\um+\rhom\ua_1\p_1\ua+\rhoa\um_1\p_1\ua+\p_1\acute{p}\mathbb{I}_1\big)\\
&-O(1)\alpha(\alpha-1)\D_0\Big\{\mu_1\big[(\alpha-2)\thetam^{\alpha-3}\p_1\um\p_1\thetam+\thetam^{\alpha-2}\p_{11}\um\big]\thetaa^2\\
&+(\mu_1+\lambda_1)\left((\alpha-2)\thetam^{\alpha-3}\p_1\um_1\p_1\thetam+\thetam^{\alpha-2}\p_{11}\um_1\right)\thetaa^2\mathbb{I}\Big\},
\end{aligned}
\end{align}
\begin{align}\label{eq0110}
\begin{aligned}
&\grave{Q}_2=\D_0\big(\rhoa\thetaa_\tau+\rhoa\ua_1\p_1\thetaa+\rhoa\ua_1\p_1\thetam+\rhom\ua_1\p_1\thetaa+\rhoa\um_1\p_1\thetaa+R\rhom\thetaa\dv\ua+R\rhoa\thetam\dv\ua+R\rhoa\thetaa\dv\um\\
&+\acute{p}\dv\ua\big)-O(1)\D_0\Big\{\alpha(\alpha-1)\Big[(\alpha-2)\kappa_1\thetam^{\alpha-3}(\p_1\thetam)^2+\kappa_1\thetam^{\alpha-2}\p_{11}\thetam+(\mu_1+\lambda_1)\thetam^{\alpha-2}(\p_1\um_1)^2\\
&+\mu_1\thetam^{\alpha-2}|\p_1\um|^2\Big]\thetaa^2+2\alpha\kappa_1\thetam^{\alpha-1}(\p_1\thetaa)^2+2(\mu_1+\lambda_1)\thetam^\alpha(\p_1\ua_1)^2+2\mu_1\thetam^\alpha|\p_1\ua|^2\Big\},
\end{aligned}
\end{align}
and $\acute{p}:=p(\rhoa,\thetaa)=R\rhoa\thetaa.$
By \cref{sys-perturb,sys-0mode}, we obtain the non-zero mode as follows:
\begin{align}\label{eq-neq-per}
	\begin{cases}
		\phia_{\tau}+\um\cdot\nabla\phia+\rhom\dv \psia+\psia\cdot\nabla \mathring{\tilde{\rho}}+\phia\dv \mr{\tilde{u}}=\tilde{Q}_0, \\
		\rhom\psia_{\tau}+\rhom \um\cdot\nabla \psia+R\nabla(\rhom\zetaa+\mr{\Tt}\phia)+
		\rhom\psia \cdot \nabla \mr{\tilde{u}}-\frac{\phia}{\mr{\rhot}}\nabla\mr{\tilde{p}}=\dv\left[2\mu(\mr{\Tt})\mathbb{D}\psia+\lambda(\mr{\Tt})\dv\psia\mathbb{I}\right]+\tilde{Q}_{1},\\
		\rhom\zetaa_{\tau}+\rhom \um\cdot\nabla\zetaa+\mr{p}\dv \psia+\rhom\psia\cdot\nabla \mr{\tilde{\theta}}+(\gamma-1)\rhom\zetaa \dv \mr{\tilde{u}} =\dv\left[\kappa(\mr{\thetat})\nabla {\zetaa}\right]+\tilde{Q}_{2},
	\end{cases}
\end{align}
where 
\begin{align*}
	\tilde{Q}_0=&\left[\um\cdot \nabla \phia+\rhom \dv \psia+\psia \cdot \nabla \mathring{\tilde{\rho}}+\phia\dv \mr{\tilde{u}}\right]+\left[\um_1 \px \phim+\rhom \px \psim_1+\psim_1 \px \mathring{\tilde{\rho}}+\phim \px \mr{\tilde{u}}_1\right]\\
	&-\left[u \cdot \nabla \phi+\rho \operatorname{div} \psi+\psi \cdot \nabla \tilde{\rho}+\phi \operatorname{div} \tilde{u}\right]{-}e_0-\hat{E}_0,\\
	\tilde{Q}_1
	=&-\left[\rho\psi_{\tau}+\rho u \cdot \nabla \psi+\nabla(p-\tilde{p})+
	\rho\psi \cdot \nabla \tilde{u}-\frac{\phi}{\tilde{\rho}}\nabla\tilde{p}-\dv\left(2\mu({\Tt})\mathbb{D}\psi+\lambda(\Tt)\dv\psi\mathbb{I}\right)\right]+\tilde{R}-\hat{E}\\
	&+\left[\rhom\psia_{\tau}+\rhom\um\cdot \nabla \psia+R\nabla(\rhom\zetaa+\mr{\Tt}\phia)+
	\rhom\psia \cdot \nabla \mr{\tilde{u}}-\frac{\phia}{\mr{\rhot}}\nabla\mr{\tilde{p}}-\dv\left(2\mu(\mr{\Tt})\mathbb{D}\psia+\lambda(\mr{\Tt})\dv\psia\mathbb{I}\right)\right]\\
	&+\left[\rhom\psim_{\tau}+\rhom \um_1\px \psim+\px(\mr{p}-\mr{\tilde{p}})\mathbb{I}_1+
	\rhom\psim \cdot \nabla \mr{\tilde{u}}-\frac{\phim}{\mr{\rhot}}\nabla\mr{\tilde{p}}-\px\left(\mu(\mr{\Tt})\px\psim+({\lambda(\mr{\Tt})}
	+{\mu(\mr{\Tt})}
	)\px\psim_1\mathbb{I}_1\right)\right],\\
	\tilde{Q}_2
	=&-\left[\rho\zeta_{\tau}+\rho u \cdot \nabla \zeta+p \operatorname{div} \psi+\rho\psi \cdot \nabla \tilde{\theta}+(\gamma-1)\rho \zeta \operatorname{div} \tilde{u} -\dv\left(\kappa(\thetat)\nabla \zeta\right)\right]\\
	&+\left[\rhom\zetaa_{\tau}+\rhom \um\cdot \nabla \zetaa+\mr{p}\dv \psia+\rhom\psia\cdot \nabla \mr{\tilde{\theta}}+(\gamma-1)\rhom\zetaa \dv \mr{\tilde{u}}-\dv\left(\kappa(\mr{\thetat})\nabla {\zetaa}\right)\right]\\
	&+\left[\rhom\zetam_{\tau}+\rhom \um_1 \px \zetam+\mr{p}\px \psim_1+\rhom\psim_1 \px \mr{\tilde{\theta}}+(\gamma-1)\rhom\zetam \px \mr{\tilde{u}}_1 -\px\left(\kappa(\mr{\thetat})\px \mr{\zeta}\right)\right]+\tilde{R}_4-\hat{E}_4.
\end{align*}

From Lemma \ref{lem6}, we know that the a priori assumptions for zero mode {are} restricted by the original perturbation as follows:
	\begin{align}\label{eq09_3}
		\begin{aligned}
		\sup_{\tau\in[0,T]}\|(\phim,\psim)\|_{W^{1,\infty}(\Omega)}&\leq C\sup_{\tau\in[0,T]}\Big\{ \e^{-(b-1)}\|\nabla (\phim,\psim)\|_{H^1(\Omega)}^{\frac{1}{2}}\| (\phim,\psim)\|_{H^1(\Omega)}^{\frac{1}{2}}+\e^{-\frac{b-1}{2}}
		\|\nabla (\phim,\psim)\|_{H^1(\Omega)}\\
		&\qquad+\|\nabla^2 (\phim,\psim)\|_{H^1(\Omega)}^{\frac{3}{4}}\| (\phim,\psim)\|_{H^1(\Omega)}^{\frac{1}{4}}\Big\}
		\le \e^a,
		\end{aligned}
		\end{align}
		\begin{align}\label{eq09_4}
		\begin{aligned}
		\sup_{\tau\in[0,T]}\|\zetam\|_{W^{1,\infty}(\Omega)}&\leq C\sup_{\tau\in[0,T]}\Big\{ \e^{-(b-1)}\|\nabla \zetam\|_{H^1(\Omega)}^{\frac{1}{2}}\| \zetam\|_{H^1(\Omega)}^{\frac{1}{2}}+\e^{-\frac{b-1}{2}}
		\|\nabla \zetam\|_{H^1(\Omega)}\\
		&\qquad+\|\nabla^2 \zetam\|_{H^1(\Omega)}^{\frac{3}{4}}\|\zetam\|_{H^1(\Omega)}^{\frac{1}{4}}\Big\}
		\le \e^{a(\gamma-1)}.
		\end{aligned}
		\end{align}
		The analysis is performed under the a priori assumptions for the  non-zero mode:
		\begin{align}\label{eq09_5}
		\begin{aligned}
		\sup_{\tau\in[0,T]}\|(\phia,\psia,\zetaa)\|_{W^{1,\infty}(\Omega)}&\leq C\sup_{\tau\in[0,T]}\Big\{ \e^{-(b-1)}\|\nabla (\phia,\psia,\zetaa)\|_{H^1(\Omega)}^{\frac{1}{2}}\| (\phia,\psia,\zetaa)\|_{H^1(\Omega)}^{\frac{1}{2}}\\
		&\qquad+\e^{-\frac{b-1}{2}}
		\|\nabla (\phia,\psia,\zetaa)\|_{H^1(\Omega)}+\|\nabla^2 (\phia,\psia,\zetaa)\|_{H^1(\Omega)}^{\frac{3}{4}}\| (\phia,\psia,\zetaa)\|_{H^1(\Omega)}^{\frac{1}{4}}\Big\}\\
		&
		{\le \chi e^{-\e^{n_2}\tau}},
		\end{aligned}
		\end{align}
where $a$ is defined by \eqref{eqda}, $\chi:=O(1)\e^{\frac{1}{3}-\frac{\gamma a}{2}}$, 
  $[0,T]$ is the time interval in which the solution exists. 

Note that we need the third-order derivative regularity of solution $(\phi,\psi,\zeta)$ to deduce the exponential decay rate of $(\phia,\psia,\zetaa)$. In section \ref{sec3}, 
 we obtain the a priori estimates and local existence for the solution $(\phi,\psi,\zeta)$ to the problem \eqref{sys-perturb}-\eqref{R_4}
, and the third-order derivative estimate of solutions is similar to 
 \eqref{eqe01}-\eqref{eqe04} 
by similar proof steps. 
We claim that the Cauchy problem \eqref{sys-perturb}, \eqref{eqins1} admits a global-in-time solution $(\phi,\psi,\zeta)\in\mathcal B(0,+\infty)$ by a standard continuation argument. It remains to prove the time-asymptotic behaviors \eqref{eqes5} and \eqref{eqe7} to complete the proof of Theorem \ref{expde}. 
For $T=+\infty$, it yields that
\begin{align}
\begin{aligned}
&\sup_{\tau\ge 0}\|(\phi,\psi,\zeta)\|_{W^{1,\infty}(\Omega)}\le \sup_{\tau\ge 0}\|(\phi,\psi,\zeta)\|_{H^3(\Omega)}\le\e^a,\\
&\int_0^{+\infty}\|\nabla\phi\|^2_{H^2(\Omega)}d\tau\le\e^a, \ \ \ \int_0^{+\infty}\|(\nabla\psi,\nabla\zeta)\|^2_{H^3(\Omega)}d\tau\le\e^a.
\end{aligned}
\end{align}
Since $\|\nabla(\phi,\psi,\zeta)\|_{L^2(\Omega)}^2(\tau)\in W^{1,1}((0,+\infty))$ and $\|\nabla^2(\phi,\psi,\zeta)\|_{L^2(\Omega)}^2(\tau)\in W^{1,1}((0,+\infty))$, one has
\begin{align*}
& \|\nabla(\phi,\psi,\zeta)\|^2_{L^2(\Omega)}(\tau)\rightarrow 0,\ \ \text{as}\ \ \tau\rightarrow+\infty,\\
& \|\nabla^2(\phi,\psi,\zeta)\|^2_{L^2(\Omega)}(\tau)\rightarrow 0,\ \ \text{as}\ \ \tau\rightarrow+\infty,
\end{align*}
respectively.
Thus by G-N 
 inequality, \eqref{eqes5} holds. 

Then we show the exponential decay rate of $(\phia,\psia,\zetaa)$ to complete the proof of Theorem \ref{expde}. This is the reason why we need third-order regularity of the solution $(\phi,\psi,\zeta)$.
\begin{Lem}\label{non-zero} 
Under the assumption of Theorem \ref{mt1} and a priori assumption \eqref{eq09_3}-\eqref{eq09_5}
, there exists 
a positive constant $
\e_0>0$ such that $0<\e\le\e_0$, constants $n_2>1$ and $\eta\le\eta_0:=o(\e^{20})$, one can obtain 
\begin{equation}\label{eqe001}
		\begin{aligned}
		&\frac{d}{d\tau}\int_{\Omega}\left(\frac{\bar{\T}}{\bar\rho}\phia^2+\bar\rho|\acute{\psi}|^2+\frac{\bar\rho}{\bar\T}\zetaa^2\right)dx+ \int_{\Omega}(\bar\T^\alpha|\nabla\acute{\psi}|^2+\bar\T^{\alpha-1}|\nabla\acute{\zeta}|^2)dx
			\leq 
			O(\eta_0
			)e^{-\e^{n_2}\tau},
		\end{aligned}
	\end{equation}
	\begin{align}\label{eqe0123}
	\begin{aligned}
	\frac{d}{d\tau}\int_{\Omega}\left(\frac{\Tb^{2\alpha}}{\rhob^3}|\nabla\phia|^2+\left|
	\frac{\bar\T^\alpha\psia\nabla\phia}{\rhob}\right|\right)
	dy
	+
	\int_{\Omega}\frac{\bar{\T}^{\alpha+1}}{\rhob^2}|\nabla\phia|^2dy
	\leq O(\eta_0
	)
e^{-\e^{n_2}\tau},
\end{aligned}
\end{align}
\begin{align}\label{eqe0124}
	\begin{aligned}
	\frac{d}{d\tau}\int_\Omega\left(
	\rhob|\nabla  \acute{\psi} |^{2}+ \frac{\rhob}{\Tb} {|\nabla \acute{\zeta}|^{2}}\right)dy+\int_\Omega\left(
	\Tb^\alpha |
	\nabla^2   \acute{\psi} |^{2}+\Tb^{\alpha-1}|\Delta \acute{\zeta}|^{2}\right)dy\le 
	O(\eta_0
	)e^{-\e^{n_2}\tau},
	\end{aligned}
	\end{align}
\begin{align}\label{ss4phi}
\begin{aligned} 
\frac{d}{d\tau}\int_\Omega\frac{\bar\T^{2\alpha}}{\bar\rho^3}|\nabla^2\phia|^2+\frac{\bar\T^\alpha\nabla\psia\cdot\nabla^2\phia}{\bar\rho}
dy+\int_\Omega\frac{\bar\T^{\alpha+1}}{\bar\rho^2}|\nabla^2\phia|^2dy\le 
O(\eta_0
)e^{-\e^{n_2}\tau},
\end{aligned}
\end{align}
\begin{align}\label{eqes0124}
	\begin{aligned}
	\frac{d}{d\tau}\int_\Omega\left(
	\rhob|\nabla^2  \acute{\psi} |^{2}+ \frac{\rhob}{\Tb} {|\nabla^2 \acute{\zeta}|^{2}}\right)dy+\int_\Omega\left(\Tb^\alpha |\nabla^3  \acute{\psi} |^{2}
	+\Tb^{\alpha-1}|\nabla^3 \acute{\zeta}|^{2}\right)dy\le 
	O(\eta_0
	)e^{-\e^{n_2}\tau}.
	\end{aligned}
	\end{align}	
	\end{Lem}

\begin{proof}
{\bf Step 1.} Multiplying \cref{eq-neq-per}$_1$ by $\frac{R{\mr{\T}}}{\rhom}\acute{\phi}$,  $(\ref{eq-neq-per})_2$ by ${\acute{\psi}}$,  $(\ref{eq-neq-per})_3$ by $\frac{\acute{\zeta}}{\Tm}$,  we get 
\begin{equation}\label{*1}
	\begin{aligned}
		\p_\tau(R&\frac{\mr{\T}}{2\rhom}\phia^2+\rhom\frac{|{\psia}|^2}{2}+\frac{\rhom}{2\Tm}\zetaa^2)+\frac{\mu(\mr{\Tt}) }{2}|\nabla\acute{\psi}+(\nabla\acute{\psi})^t|^2+{\lambda(\mr{\Tt})}(div\acute{\psi})^2+\frac{\kappa(\mr{\Tt})|\nabla\acute{\zeta}|^2}{\Tm}\\
		&:=H'+\dv(\cdots),
	\end{aligned}
\end{equation}
where
\begin{equation}
	\begin{aligned}	\label{eq000}
		H'=&\left[\tilde{Q}_{0} \frac{R{\Tm}}{\rhom}\acute{\phi}+\tilde{Q}_{1}\cdot\acute{\psi}+\tilde{Q}_{2}\frac{{\zetaa}}{\Tm}\right]\\
		+&\left[\frac{R}{2}\big(\p_\tau(\frac{\thetam}{\rhom})+\px(\frac{ \um_{1} {\Tm}}{\rhom})\big)\phia^2+\px{\D_0}(\rhoa\ua_{1})\frac{|\psia|^2}{2}+\big(\p_\tau(\frac{\rhom}{\Tm})+\px(\frac{\rhom\um_{1}}{\Tm})\big)\frac{\zetaa^2}{2}+\frac{\kappa(\mr{\Tt})}{\Tm^2}\zetaa\px\zetaa\px\Tm \right]\\
		-&\left(\psia\cdot\nabla \mathring{\tilde{\rho}}+\phia\dv \mr{\tilde{u}}\right)\frac{R\Tm}{\rhom}\phia-
		\left(\rhom\psia \cdot \nabla \mr{\tilde{u}}-\frac{\phia}{\mr{\rhot}}\nabla\mr{\tilde{p}}\right)\psia
		-\left(\rhom\psia\cdot\nabla \mr{\tilde{\theta}}+R\rhom\zetaa \dv \mr{\tilde{u}} \right)\frac{\zetaa}{\Tm}\\  
		:=&H'_1+H'_2+H'_3,
	\end{aligned}
\end{equation}
\begin{equation}
	\begin{aligned}	\label{eq010}
&\dv(\cdots)=\\
&\quad\dv\left[2\mu(\mr{\Tt})\mathbb{D}\psia\cdot\psia+\lambda(\mr{\Tt})\dv\psia\mathbb{I}\cdot\psia+\kappa(\mr{\thetat})\nabla {\zetaa}\frac{\zetaa}{\Tm}-R\frac{\mr{\T}\um}{2\rhom}\phia^2
-\rhom\um\frac{|{\psia}|^2}{2}-\frac{\rhom\um}{2\Tm}\zetaa^2-R\psia(\rhom\zetaa+\mr{\Tt}\phia)\right].
\end{aligned}
\end{equation}
We use the following truth,
	\begin{equation}
		\begin{aligned}
			\int_{\mathbb{T}^2}\acute{\phi} dx'=0,\ \int_{\mathbb{T}^2}\acute{\psi} dx'=0,\ \int_{\mathbb{T}^2}\acute{\zeta} dx'=0,
		\end{aligned}
	\end{equation}
	then the Poincar
	{\'e}'s inequality is valuable for $(\acute{\phi},\acute{\psi},\acute{\zeta})$, that is 
\begin{align}\label{eq0112}
\begin{aligned}
\int_{\Omega
}|(\phia,\psia,\zetaa)|^pdy\le \e^{(b-1)p}\int_{\Omega
}|\p_1(\phia,\psia,\zetaa)|^pdy,\ \ p\ge1, \ \  b\ge1.
\end{aligned}
\end{align}
Using the G-N 
 inequality and Poincar
{\'e}'s inequality, 
we obtain 
\begin{align}\label{eq0113}
\begin{aligned}
\|(\phia,\psia,\zetaa)\|_{L^4(\Omega
)}&\le C\e^{-\frac{1}{2}(b-1)}\|\nabla(\phia,\psia,\zetaa)\|^{\frac{1}{4}}\|(\phia,\psia,\zetaa)\|^{\frac{3}{4}}+C\e^{-\frac{1}{4}(b-1)}\|\nabla(\phia,\psia,\zetaa)\|^{\frac{1}{2}}\|(\phia,\psia,\zetaa)\|^{\frac{1}{2}}\\
&\qquad+C\|\nabla(\phia,\psia,\zetaa)\|^{\frac{3}{4}}\|(\phia,\psia,\zetaa)\|^{\frac{1}{4}}\le C\e^{\frac{1}{4}(b-1)}\|\nabla(\phia,\psia,\zetaa)\|.
\end{aligned}
\end{align}
	Then for $b\ge1$,	
	\begin{align}\label{q0}
	\begin{aligned}
	\int_\Omega\tilde Q_0\frac{\Tm}{\rhom}\phia dy&\le C\|\frac{\Tm}{\rhom}\|_\infty\Big[\|(\acute\ut,\acute\rhot,\nabla\acute\rhot,\dv\acute\ut)\|_\infty\|(\nabla\phi,\dv\psi,\psi,\phi)\|\|\phia\|+\|\psia\|_\infty\|\nabla\phia\|\|\phia\|\\
	&\qquad+\|\dv\psia\|_\infty\|\phia\|^2+\|\nabla\phim\|\|\psia\|_{L^4}\|\psia\|_{L^4}+\|\dv\psim\|\|\phia\|_{L^4}^2\Big]\\
	&\le O(\eta_0)e^{-\e^{n_2}\tau}+C\e^{\frac{1}{2}(b-1)+a}\nu^{-\alpha(\gamma-1)-1}\left\|(\sqrt{\frac{\Tb^{\alpha+1}}{\rhob^2}}\nabla\phia,\sqrt{\Tb^{\alpha-1}}\nabla\zetaa)\right\|^2.
	\end{aligned}
	\end{align}
	Estimations on $\tilde{Q}_i$, $i=1,2$, are similar to $\tilde{Q}_0$, we only present the estimate of the following term 
	for conciseness, 
\begin{align}\label{eq0120}
\begin{aligned}
&\int_\Omega(\tilde R-\D_0{\tilde R})\cdot\psia dy
\le C\|(\frac{\mr e_0\mr\ut}{\mr\rhot}, \frac{\rhom\mr e_0\mr\ut}{\mr\rhot},\frac{\rhom\mr\ut}{\mr\rhot},\frac{\rhom\mr e_0}{\mr\rhot})\|_{L^\infty}\|(\rhoa,\acute\rhot,\acute e_0,\acute\ut)\|\|\psia\|\\
&+\|(e_1-\px[
\bar\theta^\alpha\px\bar u_1],e_2,e_3, \acute\rhot^2,\px[
\bar\theta^\alpha\px\bar u_1],\frac{1}{\mr\rhot})\|_{L^\infty}\|(\acute\rhot,\phia,\frac{1}{\mr\rhot^2})\|\|\psia\|
+\|\mr\thetat^{\alpha-1}\nabla \um\|_{L^\infty}\|\zetaa\|\|\nabla\psia\|\\
&+\|(\mr\thetat^{\alpha-1}\nabla\mr\ut,\mr\thetat^\alpha)\|_{L^\infty}\|(\acute\thetat,\nabla\acute\ut)\|\|(\psia,\nabla\psia)\|+\e^{2(b-1)}\|\acute\thetat\|_{L^\infty}\|\mr\thetat^{\alpha-2}\nabla\um\|_{L^\infty}\|\zetaa\|\|\nabla\psia\|
\\&+\e^{2(b-1)}\|(\acute\rhot,\nabla\acute\rhot,\acute\thetat)\|_{L^\infty}^2\|(\mr\thetat^{\alpha-2}(\nabla(\frac{\phim}{\mr\rhot^3}),\frac{\phim}{\mr\rhot}),\mr\thetat^{\alpha-3}\zetam,\frac{\phim}{\mr\rhot^3}
)\|_{L^\infty}\|(\nabla\mr\ut,\nabla\psim)\|\|(\psia,\nabla\psia)\|\\
&+\|\frac{\nabla\mr\ut}{\mr\rhot}\|_{L^\infty}\|(\phia,\nabla\phia,\acute\rhot,\nabla\acute\rhot)\|\|(\psia,\nabla\psia)\|+\frac{1}{160}\int_\Omega\Big(\frac{\mu(\mr{\Tt}) }{2}|\nabla\acute{\psi}+(\nabla\acute{\psi})^t|^2+\lambda(\mr{\Tt})(div\acute{\psi})^2\Big)dy\\
&\le
 O(\eta_0)e^{-\e^{n_2}\tau}+
 C\e^{(b-1)-a
(2\gamma+1)
 Z+2a}\left\|(\sqrt{\frac{\Tb^{\alpha+1}}{\rhob^2}}\nabla\phia,\sqrt{\Tb^{\alpha-1}}\nabla\zetaa)\right\|^2+\frac{1}{160}\int_\Omega\mr{\Tt}^\alpha\nabla\psia^2dy,
\end{aligned}
\end{align}
where $\mr e_0=\mr\rhot_\tau+\dv(\mr\rhot\mr\ut),\acute e_0=\acute\rhot_\tau+\dv(\mr\rhot\acute\ut+\acute\rhot\mr\ut+\acute\rhot\acute\ut).$
	 The other terms in $H'_1$ can be estimated similarly to \eqref{q0}- \eqref{eq0120}, thus
	\begin{align}
		\begin{aligned}\label{H1}
			\int_{\Omega}|H'_1|dx\le  O(\eta_0)&e^{-\e^{n_2}\tau}+
 C\e^{2a-a
(2\gamma+1)
 Z}\left\|\sqrt{\frac{\Tb^{\alpha+1}}{\rhob^2}}\nabla\phia\right\|^2\\
 &+\frac{1}{160}\int_\Omega\Big(\Tb^\alpha|\nabla\psia|^2+\Tb^{\alpha-1}|\nabla\zetaa|^2\Big)dy.
		\end{aligned}
	\end{align}
One has 
\begin{equation}
	\begin{aligned}
			\int_{\Omega
			}|H'_2|dy
			&\leq C
			\left|\int_{\Omega
			}\frac{\Tm}{\rhom}\px\um\acute{\phi}^2+(\kappa({\Tm})\px\Tm)\px(\frac{\phia^2}{\rhom^2})+
			{\frac{\thetam^\alpha}{\rhom^{2}}(\px\um)^2
			\acute{\phi}^2}+\eta_0^2e^{-\e \tau}\psia^2dy\right|\\
			&\qquad+\left|\int_{\Omega
			}\frac{\rhom}{\Tm}\px\um\zetaa^2+(\kappa(\mr{\T})\px\Tm)\px(\frac{\zetaa^2}{\Tm^2})+\thetam^{-2}\big(\thetam^\alpha(\px\um)^2+\rhom\um_1\px\thetam
			\big)\zetaa^2dy\right|\\
	&\le C\nu^{-\frac{3\alpha}{2}(\gamma-1)}\e^{\frac{1}{2}b-\frac{1}{3}}
	\left\|(\sqrt{\frac{\Tb^{\alpha+1}}{\rhob^2}}\nabla\phia,\sqrt{\Tb^{\alpha-1}}\nabla\zetaa)\right\|^2,
 	\end{aligned}
\end{equation}
for small enough $\eta_0$ and $a$ given by 
\eqref{eqda}, 
 where some terms in $H'_2$ can be estimated as follows:
\begin{align*}
&
{
\int_{\Omega}\px\um(\rhob^{\gamma-2}
\phia^2,\rhob^{-\gamma+2}
\zetaa^2)dy
}\\
&\le C
\|(\bar\rho^{\gamma-2},\bar\rho^{-\gamma+2})\|_{L^\infty}\|\px\ub\|_{L^\infty}\|(\phia,\zetaa)\|^2+\|(\bar\rho^{\gamma-2},\bar\rho^{-\gamma+2})\|_{L^\infty}\|\px\psim\|\|(\phia,\zetaa)\|_{L^4}^2\\
&\leq
\nu^{-\frac{3\alpha}{2}(\gamma-1)}\e^{\frac{1}{2}(b-1)+a}\left\|(\sqrt{\frac{\Tb^{\alpha+1}}{\rhob^2}}\nabla\phia,\sqrt{\bar\theta^{\alpha-1}}\nabla\zetaa)\right\|^2,\\
&
{
\int_{\Omega}(\px\psim)^2(
{\bar\theta^\alpha}\rhob^{-2}\phia^2,
{\bar\theta^{\alpha-2}}
\zetaa^2)dy
}\le C
\|(\bar\theta^\alpha\rhob^{-2},\bar\theta^{\alpha-2})\|_{L^\infty}\|\px\psim\|_{L^\infty}\|\px\psim\|\|(\phia,\zetaa)\|_{L^4}^2\\
&\qquad\qquad\qquad\qquad\qquad\quad \ \leq C
\nu^{-(\frac{\alpha}{2}+1)(\gamma-1)}\e^{\frac{1}{2}(b-1)+2a}\left\|(\sqrt{\frac{\Tb^{\alpha+1}}{\rhob^2}}\nabla\phia,\sqrt{\bar\theta^{\alpha-1}}\nabla\zetaa)\right\|^2,\\
&\int_{\Omega}(\kappa({\Tm})\px\Tm)\px(\frac{\phia^2}{\rhom^2})dy
\leq C\left|\int_{\Omega}\Tb^{\alpha}(\px\zetam+\px\bar\theta)\big(\rhob^{-2}\phia\px\phia+\rhob^{-3}(\px\phim+\px\bar\rho)\phia^2\big)dy
\right|\\
	&\qquad\qquad\qquad\qquad\qquad \leq C\left(\e^{b-1+a(\gamma-1)}\nu^{-\gamma+1}
	+\e^{2b-2+a}\nu^{-\gamma}
	\right)\left\|\sqrt{\frac{\Tb^{\alpha+1}}{\rhob^2}}\nabla\phia\right\|^2.
\end{align*}
And $H'_3$ can be estimated similarly, we have
\begin{align*}
	\int_{\Omega}|H'_3|dy
	&=\int_{\Omega}\left|\left(\psia\cdot\nabla \mathring{\tilde{\rho}}+\phia\dv \mr{\tilde{u}}\right)\frac{R\Tm}{\rhom}\phia+
	\left(\rhom\psia \cdot \nabla \mr{\tilde{u}}-\frac{\phia}{\mr{\rhot}}\nabla\mr{\tilde{p}}\right)\psia+\left(\rhom\psia\cdot\nabla \mr{\tilde{\theta}}+R\rhom\zetaa \dv \mr{\tilde{u}} \right)\frac{\zetaa}{\Tm}\right|dy
	\\
	&\leq C\int_{\Omega}\left|\psia_{
	1}\phia\left(\frac{\Tm}{\rhom}\px\mr{\rhot}+\frac{\px\mr{\tilde{p}}}{\mr{\rhot}}\right)\right|+\left|\psia_{
	1}\zetaa\frac{\rhom\px\mr{\Tt}}{\Tm}\right|+\left|\frac{\Tm\px\mr{\tilde{u}}_{
	1}}{\rhom}\phia^2\right|+|\rhom\psia_1\px\mr{\tilde{u}}\cdot\psia|+\left|\frac{\rhom\px\mr{\ut}_{
	1}}{\Tm}\zetaa^2\right|dy
	\\
	 &\leq C\e^{2b-1{-}a}\nu^{
	 {-(\alpha+\frac{3}{2})(\gamma-1)}}
	  \int_{\Omega}\Tb^{\alpha}|\nabla\psia|^2+\Tb^{\alpha-1}|\nabla\zetaa|^2+\frac{\Tb^{\alpha+1}}{\rhob^2}|\nabla\phia|^2dy.
\end{align*}

	Finally, we have
	\begin{equation}
		\begin{aligned}
		\int_{\Omega}\p_\tau(R&\frac{\bar{\T}}{\bar\rho}\phia^2+\bar\rho\frac{|\acute{\psi}|^2}{2}+\frac{\bar\rho}{\bar\T}\zetaa^2)dy+ \int_{\Omega}(\bar\T^\alpha|\nabla\acute{\psi}|^2+\bar\T^{\alpha-1}|\nabla\acute{\zeta}|^2)dy\\
	&		\leq o(\e^{a(\gamma-1)Z})
			\left\|\sqrt{\frac{\Tb^{\alpha+1}}{\rhob^2}}\nabla\acute{\phi}\right\|^2+O(\eta_0)
			e^{-\e^{n_2}\tau}.
		\end{aligned}
	\end{equation}
	
	{\bf Step 2.} We still need to estimate $\|\nabla\acute{\phi}\|^2$.
	Taking 
	$\cref{eq-neq-per}_2\times\frac{\Lambda(\mr{\Tt})
	}{\rhom^2}\nabla\acute{\phi}+\nabla\left\{\cref{eq-neq-per}_1\right\}\times\frac{\Lambda^2(\mr\Tt)}{\rhom^3}\nabla\phia$, one has
	\begin{align}\label{eq-phiy101}
	\begin{aligned}
		&\p_\tau\left[\frac{\Lambda(\mr{\Tt})}{\rhom}\psia\cdot\nabla\phia+\frac{\Lambda^2(\mr\Tt)}{
		{2}\rhom^3}\left|\nabla\phia\right|^2\right]+\frac{R\Lambda(\mr{\Tt})\mr{\Tt}}{\rhom^2}|\nabla\phia|^2\\
		&=-\Big(\frac{\Lambda(\mr{\Tt})}{\rhom} \um\cdot\nabla \psia\cdot\nabla\phia+\frac{R\Lambda(\mr{\Tt})}{\rhom^2}\nabla(\rhom\zetaa)\cdot\nabla\phia+
		\frac{\Lambda(\mr{\Tt})}{\rhom} \psia \cdot \nabla \mr{\tilde{u}}\cdot\nabla\phia+\frac{\phia R\Lambda(\mr{\Tt})}{\rhom^2}\nabla\mr\Tt\cdot\nabla\phia\\
		&\qquad-\frac{\phia\Lambda(\mr{\Tt})}{\mr{\rhot}\rhom^2}\nabla\mr{\tilde{p}}\cdot\nabla\phia\Big)
		+\frac{\Lambda(\mr\Tt)}{\rhom^2}\tilde{Q}_1\cdot\nabla\phia+\psia\cdot\p_\tau\left[\frac{\Lambda(\mr{\Tt})}{\rhom}\nabla\phia\right]+\frac{\lambda'(\mr{\Tt})\Lambda(\mr\Tt)}{\rhom^2}\dv\psia\nabla\mr\Tt\cdot\mathbb{I}\cdot\nabla\phia\\
		&\qquad+\frac{2\Lambda(\mr\Tt)}{\rhom^2}\mu'(\mr\Tt)\nabla\mr\Tt\cdot\mathbb{D}\psia\cdot\nabla\phia
	{-\nabla\left(\frac{\mu(\mr{\Tt})\Lambda(\mr\Tt)}{\rhom^2}\right)\cdot\left(\nabla\psia\cdot\nabla\phia-\nabla\phia\cdot\nabla\psia\right)}
		\\
	&\qquad+\left[\p_\tau\left(\frac{\Lambda^2(\mr\Tt)}{\rhom^3}\right)+\dv\left(\frac{\Lambda^2(\mr\Tt)\um}{\rhom^3}\right)\right]\frac{|\nabla\phia|^2}{2}-\Big[
	\frac{\Lambda^2(\mr{\Tt})}{\rhom^3}\nabla\phia\cdot\nabla\rhom\dv\psia+
	{\frac{\Lambda^2(\mr{\Tt})\nabla\phia\cdot\nabla\um\cdot\nabla\phia}{\rhom^3}}\\
		&\qquad+\frac{\Lambda^2(\mr{\Tt})}{\rhom^3}\nabla\phia\cdot\nabla(\psia_1\px\mr{\rhot}+\phia\dv \mr{\tilde{u}}-\tilde{Q}_0)\Big]+\dv(\cdots),
	\end{aligned}
	\end{align}
where $\Lambda(\mr\Tt):=2\mu(\mr\Tt)+\lambda(\mr\Tt)$.
Integrating the above equation 
	 on $\Omega$, we have
\begin{align}\label{eqs29}
	\begin{aligned}
	&\frac{d}{d\tau}\int_{\Omega}\frac{\Lambda^2(\mr\Tt)}{
	{2}\rhom^3}|\nabla\phia|^2+
	\frac{\Lambda(\mr\Tt)\psia\cdot\nabla\phia}{\rhom
	}dy
	+
	\int_{\Omega}\frac{R\Lambda(\mr{\Tt})\mr{\Tt}}{\rhom^2}|\nabla\phia|^2dy
	\leq C
	\int_{\Omega}|J'|dy,
\end{aligned}
\end{align}
where
\begin{align*}
	&
	\int_{\Omega}|J'|dyd\tau\leq C
	\int_\Omega\left|\frac{\Lambda(\mr{\Tt})}{\rhom} \um+\frac{\lambda'(\mr{\Tt})\Lambda(\mr\Tt)}{\rhom^2}\px\mr{\Tt}+\mu(\mr{\Tt})\nabla\left(\frac{\Lambda(\mr\Tt)}{\rhom^2}\right)\right||\nabla\psia||\nabla\phia|\\
	&+\frac{\Lambda^2(\mr{\Tt})}{\rhom^3}\left|\px\rhom+\px \mr{\tilde{u}}+\frac{\rhom}{\Lambda(\Tt)}\mu'(\mr\Tt)\px\mr{\Tt}\right||\nabla\phia||\nabla\psia|+\left|\frac{\Lambda^2(\mr{\Tt})}{\rhom^3}\nabla\phia\cdot\nabla\tilde{Q}_0\right|+\left|\frac{\Lambda(\mr\Tt)}{\rhom^2}\tilde{Q}_1\cdot\nabla\phia\right|\nonumber\\
	&+\left|\frac{\Lambda(\mr{\Tt})}{\rhom^2}\nabla(\rhom\zetaa)\cdot\nabla\phia\right|+\left|\phia \px\mr{\Tt}+\frac{\phia}{\mr{\rhot}}\px\mr{\tilde{p}}+
	\rhom |\psia|\ \px \mr{\tilde{u}}\right|\frac{ \Lambda(\mr{\Tt})}{\rhom^2}|\nabla\phia|+\psia\cdot\pt\left[\frac{\Lambda(\mr{\Tt})}{\rhom}\nabla\phia\right]\\
	&+\left|\pt\left(\frac{\Lambda^2(\mr\Tt)}{\rhom^3}\right)+\dv\left(\frac{\Lambda^2(\mr\Tt)\um}{\rhom^3}\right)\right|\frac{|\nabla\phia|^2}{2}+\frac{\Lambda^2(\mr{\Tt})}{\rhom^3}|\nabla\phia|\left|(\psia_1\px^2\mr{\rhot}+\phia\px^2 \mr{\tilde{u}})\right|dy\\
	&:=\sum_{i=1}^{9}J'_i.
\end{align*}
Next, we will give some estimations of $J'_i$, $i=1,...,9$, 
\begin{align}
\begin{aligned}
	&J'_1+J'_2\leq C\nu^{-(\gamma-1)}
	\int_\Omega \Tb^{\alpha}|\nabla\psia|^2dy
	+\frac{1}{160}
	\int_\Omega\frac{\Tb^{\alpha+1}}{\rhob^2}|\nabla\phia|^2dy,\\
&J'_3+J'_4\le O(\eta_0)e^{-\e^{n_2}\tau}+\frac{1}{160}\int_{\Omega}\frac{\Tb^{\alpha+1}}{\rhob^2}|\nabla\phia|^2+{\Tb^{\alpha}}|\nabla\psia|^2
+\bar\theta^\alpha|\Delta\psia|^2dy,\\
&J'_7\le C\nu^{-1}\eta_0
^2
e^{-\e^{n_2}\tau}
+C\e^a\nu^{-{\frac{1}{2}}(\gamma-1)}
\int_{\Omega}\frac{\Tb^{\alpha+1}}{\rhob^2}|\nabla\phia|^2dyd\tau+C
\int_{\Omega}{\Tb^{\alpha}}|\nabla\psia|^2dyd\tau.
\end{aligned}
\end{align}
In addition, other terms are similar as \eqref{eqj1j2} 
 and \eqref{eqj0123}.
Therefore, it yields
\begin{align}\label{eq0123}
	\begin{aligned}
	\frac{d}{d\tau}&\int_{\Omega}\frac{\Lambda^2(\mr\Tt)}{
	{2}\rhom^3}|\nabla\phia|^2+
	\frac{\Lambda(\mr\Tt)\psia\nabla\phia}{\rhom^{\blue2}}dy
	+
	\int_{\Omega}\frac{R\Lambda(\mr{\Tt})\mr{\Tt}}{\rhom^2}|\nabla\phia|^2dy
	\\
	&\leq O(\eta_0
	)
e^{-\e^{n_2}\tau}+\frac{1}{160}\int_\Omega\bar\T^\alpha|\nabla^2\psia|^2dy.
\end{aligned}
\end{align}

	{\bf Step 3. }
	 Taking $\nabla\eqref{eq-neq-per}_1\times
	\frac{ R\mr{\Tt}}{{\rhom}}\nabla\acute{\phi}+\nabla
	\eqref{eq-neq-per} _{2} \times
	\nabla  \acute{\psi} +\nabla\eqref{eq-neq-per}_3\times
	\frac{1}{\Tm} \nabla\acute{\zeta}$, we have
	\begin{align}\label{equ0124}
	\begin{aligned}
	&\p_\tau\left[\frac{R\mr{\Tt}}{
	2\rhom}\left|\nabla\phia\right|^2+\rhom \frac{|\nabla  \acute{\psi} |^{2}}{2}+ \frac{\rhom}{\Tm} \frac{|\nabla \acute{\zeta}|^{2}}{2}\right]+\mu(\mr\Tt) |\Delta  \acute{\psi} |^{2}+(\mu(\mr\Tt)+\lambda(\mr\Tt)) |\nabla \dv   \acute{\psi} |^{2}+\frac{\kappa(\mr{\Tt})  }{\Tm}|\Delta \acute{\zeta}|^{2}\\
	&=\Big[\p_\tau\left(\frac{R\mr{\Tt}}{\rhom}\right)\frac{|\nabla\phia|^2}{2}+\p_{\tau}\rhom\frac{|\nabla \acute{\psi} |^2}{2}+\p_{\tau}\left(\frac{\rhom}{\Tm}\right)\frac{|\nabla \acute{\zeta}|^{2}}{2}\Big]+\Big[\nabla
	\rhom\cdot\nabla
	 \acute{\psi} \cdot  \acute{\psi} _{\tau}+
	   \frac{\nabla{\rhom}}{\Tm} \cdot \nabla \acute{\zeta} \acute{\zeta}_{\tau}\Big]
	\\
	&+\Big[\dv\left(\frac{R\mr{\Tt}}{\rhom}\um\right)\frac{|\nabla\phia|^2}{2}-\frac{R\mr{\Tt}}{\rhom}\nabla\phia\cdot\nabla\um\cdot\nabla\phia-\frac{R\mr{\Tt}}{\rhom}\nabla\phia\cdot\nabla\rhom\dv\psia
	+\dv(\rhom\um)\frac{|\nabla \acute{\psi} |^2}{2}\\
	&-\nabla
	(\rhom \um\cdot \nabla ) \acute{\psi} \cdot \nabla  \acute{\psi}
	{-}R(\nabla \mr{\Tt} \acute{\phi}+ \nabla \rhom \acute{\zeta}) \cdot \nabla \dv   \acute{\psi}
	{-}\left(2\nabla\mu(\mr{\Tt})\cdot\mathbb{D}\psia+\nabla\lambda(\mr\Tt)\dv\psia\cdot\mathbb{I}\right)\cdot\Delta
	\psia\\
	&+\dv\left(\frac{\rhom\um}{\Tm}\right)\frac{|\nabla \acute{\zeta}|^{2}}{2} - \frac{\rhom}{\Tm} \nabla\acute{\zeta} \cdot\nabla  \um\cdot \nabla \acute{\zeta}
	{-\frac{\nabla\rhom\cdot\um}{\thetam}|\nabla\zetaa|^2-}\frac{\Delta\zetaa}{\Tm}\nabla\kappa(\mr{\Tt})\cdot\nabla\zetaa
	{+\nabla\kappa(\mr{\Tt})\cdot\nabla\thetam\frac{|\nabla\zetaa|^2}{\thetam^2}}\\
	&+\kappa(\mr{\Tt})\frac{\nabla\zetaa\cdot\nabla\thetam\Delta\zetaa}{\Tm^2} 
		-\frac{\dv   \acute{\psi} }{\Tm} \nabla \mr{p} \cdot \nabla \acute{\zeta}
	-\frac{R\mr{\Tt}}{\rhom}\nabla\phia\cdot\nabla(\psia_1\px\mr{\rhot}+\phia\dv \mr{\tilde{u}})
	-\nabla \left(\rhom\psia\cdot\nabla \mr{\tilde{\theta}}+R\rhom\zetaa \dv \mr{\tilde{u}}\right)\Big]\\
	&+\Big(\frac{R\mr{\Tt}}{\rhom}\nabla\phia\cdot\nabla\tilde{Q}_0+\nabla\tilde{Q}_{
	1}\cdot\nabla \acute{\psi}+\nabla\tilde{Q}_4\cdot\frac{\nabla\zetaa}{\thetam}\Big)+\dv(\cdots):= T_{11}+T_{12}+T_{13}+T_{14}+\dv(\cdots).
		\end{aligned}
		\end{align}

	We 
	integrate the above equation 
	 on $\Omega$. 
	Note that we have achieved cancellations between those bad terms in \cref{equ0124}
	. Next, we estimate the 
	 terms on the right-hand side of \cref{equ0124}
	, which 
	were arranged by some characters for easy estimating. 
	
	First, by the Poincar
	{\'e}'s inequality \eqref{eq0112} and system \eqref{sys-0mode}, one has
	\begin{align}\label{eqt11}
		\int_{\Omega}T_{11}dy
		&\le C\nu^{-1}\eta_0^2
		e^{-\e^{n_2}\tau}+ 
		\frac{1}{160}\int_\Omega\Big(
		\bar\theta^\alpha|\nabla^2\psia|^2+\bar\theta^{\alpha-1}|\Delta
		\zetaa|^2\Big)dy.
	\end{align}
	From the equations \eqref{eq-neq-per}, we have
	\begin{align}
		\begin{aligned}
			\int_{\Omega}T_{12}dy
			& \le C\frac{\e^a}{\nu^{\gamma-1}}
			\left(\|\p_\tau( \acute{\psi} ,\acute{\zeta})\|^2+\|\nabla( \acute{\psi} ,\acute{\zeta})\|^2\right)\\
			&\leq C\eta_0^2
			\nu^{-1}e^{-\e^{n_2}\tau}+
			\frac{1}{160} \int_\Omega(\bar\theta^\alpha|\nabla^2
			\psia|^2+\bar\theta^{\alpha-1}|\Delta
			\zetaa|^2).
		\end{aligned}
	\end{align}
	Then it is easy to know that
	\begin{align}\label{eqt13}
		\int_{\Omega}T_{13}dy\leq C\eta_0^2
		\nu^{-1}e^{-\e^{n_2}\tau}+
		\frac{1}{160} \int_\Omega (\bar\theta^\alpha|\nabla^2
		\psia|^2+\bar\theta^{\alpha-1}|\Delta
		\zetaa|^2).
	\end{align}
	Estimates on the rest terms are similar to \cref{H1}, we have
	\begin{align}\label{eqt14}
	\int_\Omega T_{14}dy\le O(\eta_0)
	e^{-\e^{n_2}\tau}+
	\frac{1}{160} \int_\Omega(\bar\theta^\alpha|\nabla^2
	\psia|^2+\bar\theta^{\alpha-1}|\Delta
	\zetaa|^2).
	\end{align} 
Thus from {the} above inequalities in this step
, we get
\begin{align}\label{equ0125}
	\begin{aligned}
	\frac{d}{d\tau}\int_\Omega&\left[\frac{\bar{\T}}{\bar\rho}\left|\nabla\phia\right|^2+\bar\rho
	|\nabla  \acute{\psi} |^{2}
	+ \frac{\bar\rho}{\bar\T} 
	|\nabla \acute{\zeta}|^{2}
	\right]dy
	+\int_\Omega\bar\T^\alpha |\Delta  \acute{\psi} |^{2}
	+\Tb^\alpha |\nabla \dv   \acute{\psi} |^{2}+\bar\T^{\alpha-1}
	|\Delta \acute{\zeta}|^{2}dy\\
	&\le O(\eta_0)
	e^{-\e^{n_2}\tau}.
	\end{aligned}
	\end{align}
	
{\bf Step 4.} Now we estimate $\|\nabla^2\phia\|$. Taking $\p_{ij}\eqref{eq-neq-per}_1\times\frac{\Lambda^2(\mr\Tt)}{\rhom^3}\p_{ij}\phia+\p_i\eqref{eq-neq-per}_2\times\frac{\Lambda(\mr\Tt)}{\rhom^2}\p_{ij}\phia$ with $\p_{ij}:=\p_{y_iy_j}$, $i,j=1,2,3$, one has
\begin{align}
\begin{aligned}
\p_\tau&\left\{\frac{\Lambda^2(\mr\Tt)}{2\rhom^3}|\p_{ij}\phia|^2+\frac{\Lambda(\mr\Tt)\p_{ij}\phia\p_i\psia}{\rhom}\right\}+\frac{R\Lambda(\mr\Tt)\mr\Tt}{\rhom^2}|\p_{ij}\phia|^2\\
&=\Big[\left(\p_\tau\left\{\frac{\Lambda^2(\mr\Tt)}{2\rhom^3}\right\}+\dv\left\{\frac{\um\Lambda^2(\mr\Tt)}{2\rhom^3}\right\}\right)|\p_{ij}\phia|^2-\frac{\Lambda(\mr\Tt)\p_{ij}\phia}{\rhom^2}\Big(\p_i\left\{\rhom\psia\cdot\nabla\mr\ut-\frac{\phia}{\mr\rhot}\nabla\mr{\tilde p}\right\}+\p_i\phia\p_j\mr\Tt\\
&+\p_j\phia\p_i\mr\Tt+\phia\p_{ij}\mr\Tt\Big)-\frac{\Lambda^2(\mr\Tt)\p_{ij}\phia}{\rhom^3}\Big(\p_i\um\cdot\nabla\phia+\p_i\rhom\dv\psia+\p_i(\psia\cdot\nabla\mr\rhot+\phia\dv\mr\ut)\Big)\Big]\\
&-\Big[\frac{\Lambda(\mr\Tt)\p_{ij}\phia\um\cdot\nabla\p_i\psia_j}{\rhom}+\frac{\Lambda^2(\mr\Tt)\p_j\rhom}{\rhom^3}\p_{ij}\phia\dv\p_i\psia+\frac{\Lambda^2(\mr\Tt)\p_{ij}\phia\p_j\um\cdot\nabla\p_i\phia}{\rhom^3}\\
&+\p_i\left\{\frac{\Lambda(\mr\Tt)}{\rhom^2}\mu(\mr\Tt)\right\}\left(\p_{ij}\psia_j\p_{ij}\phia-\p_{ij}\phia\p_{ij}\psia_i\right)-\frac{\Lambda(\mr\Tt)\p_{ij}\phia}{\rhom^2}\Big(\dv\p_j\psia\cdot\mathbb{I}\cdot\nabla\lambda(\mr{\Tt})\\
&+2\nabla\mu(\mr{\Tt})\cdot\mathbb{D}\p_j\psia\Big)\Big]-\frac{\Lambda(\mr\Tt)\p_{ij}\phia\p_{ij}(\rhom\zetaa)}{\rhom^2}+\Big[\p_i\psia_j\p_{\tau}\left\{\frac{\Lambda(\mr\Tt)\p_{ij}\phia}{\rhom}\right\}+\frac{\Lambda^2(\mr\Tt)\p_{ij}\phia}{\rhom^3}\p_{ij}\tilde Q_0\\
&+\frac{\Lambda(\mr\Tt)\p_{ij}\phia}{\rhom^2}\p_i\tilde{Q}_1\Big]+\dv(\cdots).
\end{aligned}
\end{align}
Integrating above equation on $\Omega$, 
similar to \eqref{eqs29}, 
 one has
\begin{align}\label{s4phi}
\begin{aligned} 
\frac{d}{d\tau}&\int_\Omega\frac{\bar\T^{2\alpha}}{\bar\rho^3}|\nabla^2\phia|^2+\frac{\bar\T^\alpha\nabla\psia\cdot\nabla^2\phia}{\bar\rho}dy+\int_\Omega\frac{\bar\T^{\alpha+1}}{\bar\rho^2}|\nabla^2\phia|^2dy\\
&\le O(
\eta_0)e^{-\e^{n_2}\tau}+\frac{1}{160}\int_\Omega\bar\T^\alpha|\nabla^2\dv\psia|^2dy.
\end{aligned}
\end{align}
The proof of $
\|\nabla^3(\psia,\zetaa)\|$ is similar to that of steps 3-4 and omitted for brevity. 
Therefore, the proof of Lemma \ref{non-zero} is completed.
\end{proof}	

{\bf Proof of Theorem \ref{expde}.}
Let 
\begin{align*}
\begin{aligned}
&H(\tau):=\int_\Omega\frac{\bar\T}{\bar\rho}|\phia|^2+\bar\rho|\psia|^2+\frac{\bar\rho}{\bar\T}|\zetaa|^2+\bar\rho|\nabla\psia|^2+\frac{\rhob}{\Tb}|\nabla\zetaa|^2+\frac{\Tb^{2\alpha}}{\rhob^3}|\nabla\phia|^2
dy,\\
&G(\tau):=\int_\Omega\Tb|\nabla\psia|^2+\Tb^{\alpha-1}|\nabla\zetaa|^2+
\Tb^\alpha|\nabla^2
\psia|^2
+\Tb^{\alpha-1}|\nabla^2
\zetaa|^2+\frac{\Tb^{\alpha+1}}{\rhob^2}|\nabla\phia|^2+\frac{\Tb^{\alpha+1}}{\rhob^2}|\nabla^2\phia|^2dy.
\end{aligned}
\end{align*}
Then by Lemma \ref{non-zero}, it yields 
\begin{align}
{H'(\tau)+G(\tau)\le O(\eta_0
)e^{-\e^{n_2}\tau}}.
\end{align}
From 
\eqref{eq0112}, it easy to know 
\begin{align}
H(\tau)\le \nu^{-(\gamma+1)}G(\tau)+O(\eta_0
)e^{-\e^{n_2}\tau}.
\end{align}
It holds that
\begin{align}
H'(\tau)+\nu^{\gamma+1}H(\tau)\le O(\eta_0
)e^{-\e^{n_2}\tau}.
\end{align}
By making use of Gronwall's inequality, we can get 
\begin{align}\label{eqg1}
H(\tau)\le O(\e^{\frac{1}{3}}
)e^{-\e^{n_2}\tau}, \ \ \forall \ \tau\ge0.
\end{align}
Then one can use \eqref{eq46i} for $\Lambda=\e^{b-1}$ and \eqref{eqg1} to get 
\begin{align}
\|(\phia,\psia,\zetaa)\|_\infty\le O(\e^{\frac{1}{3}}
)e^{-\e^{n_2}\tau}.
\end{align}
Thus, \eqref{eqe7} is true and the proof of Theorem \ref{expde} is finished.

{\bf Proof of Theorem \ref{mt1}.} Once Theorem \ref{expde} is proved, by \eqref{eq094} and \eqref{eqes5}, it is straightforward to obtain 
\begin{align}
\begin{aligned}
\|(\rho,u,\T)(x,t)-(\rhob,\ub,\Tb)(\frac{x_1}{t})\|_{W^{1,\infty}(\R^3)}
&\le \|(\phi,\psi,\zeta)(y,\tau)\|_{W^{1,\infty}(\Omega)}+\|\rhot(x,t)-\rhob(\frac{x_1}{t})\|_{W^{1,\infty}(\Omega)}\\
&\ +\|\ut(x,t)-\ub(\frac{x_1}{t})\|_{W^{1,\infty}(\Omega)}+\|\Tt(x,t)-\Tb(\frac{x_1}{t})\|_{W^{1,\infty}(\Omega)}\\
&\le\|(\phi,\psi,\zeta)(y,\tau)\|_{W^{1,\infty}(\Omega)}+O(\eta_0
)e^{-\e^{n_2}\tau}.
\end{aligned}
\end{align}
Then it yields \eqref{time0} immediately.
In addition, it follows from \eqref{eq-ansatz} and Lemma \ref{lem-periodic-solution}, one has
\begin{align}
\begin{aligned}
&\|\acute{\rhot}\|_\infty\le\|\rhot-\rhob\|_\infty\le\|\rhot_+,\rhot_-\|_\infty\le \eta_0e^{-\e^{n_2}t},\\
&\|\acute{\mt}\|_\infty\le\|\mt-\mb\|_\infty\le\|\mt_+,\mt_-\|_\infty\le \eta_0e^{-\e^{n_2}t},\\
&\|\acute{\Et}\|_\infty\le\|\Et-\Eb\|_\infty\le\|\Et_+,\Et_-\|_\infty\le \eta_0e^{-\e^{n_2}t},
\end{aligned}
\end{align}
for $n_2>1$. Then we have
\begin{align}
\begin{aligned}
&\|\rhoa\|_\infty\le\|\acute\rhot\|_{\infty}+\|\phia\|_\infty\le O(\e^{\frac{1}{3}})e^{-\e^{n_2} t},\\
&\|\ua\|_\infty\le\left\|\frac{\acute\mt}{\acute\rhot}\right\|_\infty+\|\psia\|_\infty\le 
O(\e^{\frac{1}{3}})e^{-\e^{n_2}t},\\
&\|\Ta\|_\infty\le\left\|
\frac{\acute\Et}{\acute\rhot}\right\|_\infty+\left\|
\acute\ut
\right\|^2_\infty+\|\zetaa\|_\infty\le O(\e^{\frac{1}{3}})e^{-\e^{n_2}t}.
\end{aligned}
\end{align}
The proof of Theorem \ref{mt1} is completed.
\subsection{Appendix: Proof of Lemma \ref{lem-periodic-solution}}

\begin{proof}
Similar as \eqref{sys-perturb}, we study the perturbation of $(\rho,u,\T)(t,x)$ around $(\rho_-,u_-,\T_-)$ by 
\begin{align*}
\phi{(\tau,y)}:={\rho}(t, x)-{\rho}_{-}
,\quad
\psi{(\tau,y)}:=\uf(t, x)-{u}_{-}
,\quad
\xi{(\tau,y)}:=\theta(t, x)-{\theta}_{-}
,
\end{align*}
where $(\rho_-,u_-,\T_-)=(\nu, u_{1\nu},0,0, e^{\bar{S}}\nu^{\gamma-1})$ (refer to \eqref{nrws}, \eqref{eq091}).
Then the perturbation system is
\begin{align}\label{period-perturb}
\begin{cases}
\phi_{\tau}+u \cdot \nabla \phi+\rho \operatorname{div} \psi=0, \\
\rho\psi_{\tau}+\rho u \cdot \nabla \psi+\nabla(p-p_{-})=\dv\left(2\mu(\T)\mathbb{D}(\psi)+\lambda(\T)\dv\psi \mathbb{I}\right),\\
\rho\zeta_{\tau}+\rho u \cdot \nabla \zeta+p \operatorname{div} \psi =\dv(\kappa(\theta)\nabla \zeta)+\mu(\theta)\frac{\left|\nabla \psi+(\nabla \psi)^{t}\right|^{2}}{2}+\lambda(\theta)(\operatorname{div} \psi)^{2},
\end{cases}
\end{align}
with the initial data 
{
\begin{equation}
\begin{aligned}
\left\|(\phi_{0}, {\psi}_{0},\zeta_0)(y)\right\|_{H^{k+2}\left(\bar{\mathbb{T}}_{\e}^{3}\right)}^2 
\leq \eta_0, \ \ \ \bar\Torus_\e^3:=[0,\e^{b-1}]^3.
\end{aligned}
\end{equation}}
The energy estimates are based on the following a priori assumptions
{
\begin{align}\label{prior}
\begin{aligned}
	&\sup_{\tau\in[0,\tau_1(\varepsilon)]}\|(\phi,\psi)\|_{L^{\infty}(\bar{\mathbb{T}}_\e^{3})}\le\varepsilon^{a},\ \ \ \sup_{\tau\in[0,\tau_1(\varepsilon)]}\|\zeta\|_{L^{\infty}(\bar{\mathbb{T}}_\e^{3})}\le\varepsilon^{(\gamma-1)a},\\
	&\sup_{\tau\in[0,\tau_1(\varepsilon)]}\|(\nabla\phi,\nabla\psi,\nabla\zeta)\|_{H^{1}(\bar{\mathbb{T}}_\e^{3})}\leq 1.\ 
\end{aligned}
\end{align}}
From \eqref{prior}, we immediately get that
\begin{align*}
\frac{\rho_-}{2}<\rho<\frac{3\rho_-}{2},\frac{\T_-}{2}<\T<\frac{3\T_-}{2}.
\end{align*}

{\bf Step  1.} 
Multiplying $\eqref{period-perturb}_1$ by 
$R\theta_{-}(1-\frac{\rho_{-}}{\rho}),$ $\eqref{period-perturb}_2$ by $\psi$, $\eqref{period-perturb}_3$ by $\frac{\zeta}{\theta}$, we get
\begin{align}
\p_\tau(\rho{\mathbb{E}}_-)+\dv Q_-+\frac{\mu({\T})\T_-}{2\T}|\nabla\psi+(\nabla\psi)^t|^2+\frac{\lambda({\T})\T_-}{\T}|\dv\psi|^2+\frac{\kappa({\T}){\T_-}}{\T^2}|\nabla\zeta|^2=0,
\end{align}
where 
\begin{align}
\begin{aligned}
&{\mathbb{E}}_-:=R{{\theta}_-}\hat\Phi(\frac{{\rho_-}}{\rho})+\frac{|\psi|^2}{2}+ {\theta}_-\hat\Phi(\frac{\theta}{{\theta_-}}),\\
&Q_+=\rho u\mathbb{E}_-+(p- p_-)\psi-2\mu(\theta)\mathbb{D}(\psi)\cdot\psi-\lambda(\theta)\dv\psi\mathbb{I}\cdot\psi-\kappa(\theta)\nabla\zeta\frac{\zeta}{\theta}.
\end{aligned}
\end{align}
Integrating the above equation over $\bar{\mathbb{T}}_\e^{3},$ we have 
\begin{align}\label{be}
\begin{aligned}
\frac{d}{d\tau}\int_{\bar\Torus_\e^3}\big({\rho}_-^{\gamma-2}\phi^2+{\rho}_-\psi^2+{\rho}_-^{2-\gamma}\zeta^2\big)dy+C\int_{\bar\Torus_\e^3}\T_-^{\alpha}(|\nabla\psi|^2+\frac{|\nabla\zeta|^2}{\theta})dy\leq 0.
\end{aligned}
\end{align}


{\bf Step  2.} 
For $\|\nabla{\phi}\|^2$, multiplying $\eqref{period-perturb}_2\times\frac{\T_-^{\beta}}{\rho^2}\nabla{\phi}+\nabla\eqref{period-perturb}_1\times\frac{\Lambda(\T_-)\nabla\phi}{\rho^3}$, one has
\begin{align}\label{eq-phiy1}
\begin{aligned}
&\p_\tau\left[\frac{\T_-^{\beta}}{\rho}\psi\cdot\nabla\phi+\frac{\Lambda(\T_-)}{{2}\rho^3}\left|\nabla\phi\right|^2\right]+\frac{R\T_-^{\beta}{\T}}{\rho^2}|\nabla\phi|^2+\dv(\cdots)
=O(1)\bigg\{\frac{\T_-^{\beta}}{\rho}|(\nabla\psi,\nabla\zeta)||\nabla\phi|\\
&+\T_-^{\beta}|\nabla\psi|^2+\left|\frac{\T_-^{\alpha+\beta-1}}{\rho^2}\right||\nabla\psi||\nabla\phi||\nabla\zeta|+\left|\frac{\zeta\T_-^{\alpha+\beta-1}}{\rho^2}\right||\nabla^2\psi||\nabla\phi|+\frac{\Lambda({\T_-})}{\rho^3}|\nabla\phi|^2|\nabla\psi|\bigg\},
\end{aligned}
\end{align}
where $\Lambda(\T_-):=\theta_-^{\beta}(2\mu(\T_-)+\lambda(\T_-))$. Integrating \eqref{eq-phiy1} over $\bar{\mathbb{T}}_\e^{3}$ and using G-N {inequality} 
 \eqref{GN2} with $\Lambda=\varepsilon^{b-1}$, under assumptions \eqref{prior}, we have
\begin{align}\label{naphi}
\begin{aligned}
\frac{d}{d\tau}&\int_{\bar\Torus_\e^3}\frac{\Lambda(\T_-)}{2\rho^3}|\nabla\phi|^2+\frac{\T_-^{\beta}}{\rho}\psi\cdot\nabla\phi dy
+
\int_{\bar\Torus_\e^3}\frac{R\T_-^{\beta+1}}{\rho^2}|\nabla\phi|^2dy\\
&\leq 
\nu^{-12-4\gamma}\Lambda^{-6}\T_-^{-4\alpha}\T_-^{-3(\alpha+\beta)}|\ln\varepsilon|^3\bigg(\sup_{[0,\tau_1]}\|(\nabla\phi,\nabla\psi,\nabla\zeta)\|_{H^1}^2+1\bigg)\|(\theta_-^{\frac{\alpha-1}{2}}\nabla\zeta,\theta_-^{\frac{\alpha}{2}}\nabla\psi)\|^{2}\\
&\qquad+\T_-^{\alpha+\beta}|\ln\varepsilon|^{-1}\|(\T_-^{\frac{\alpha}{2}}\nabla^2\psi,\T_-^{\frac{\alpha}{2}}\nabla^2\zeta)\|^2.
\end{aligned}
\end{align}
Also, we can get the estimation for $\|(\nabla\psi,\nabla{\zeta})\|^2$.
Taking $\eqref{period-perturb}_{2} \times  -\frac{\triangle{\psi}}{\rho}$, we have
\begin{align}\label{eq0124}
\begin{aligned}
\p_\tau&\left[\frac{|\nabla{\psi}|^{2}}{2}\right]+\frac{\mu(\T)}{\rho} |\Delta  {\psi} |^{2}+(\mu(\T)+\lambda(\T)) \frac{|\nabla \dv {\psi} |^{2}}{\rho}+\dv(\cdots)\\
&=O(1)\bigg\{|\nabla\psi|^3+\left|(\frac{{\T_-}}{\rho}\nabla{\phi}, \nabla{\zeta})\right||\Delta{\psi}|
+\theta_-^{\alpha-1}|\nabla\zeta||\nabla\psi|\frac{|\nabla^2\psi|}{\rho}+\theta_-^{\alpha}|\nabla\psi||\nabla\phi|\frac{|\nabla^2\psi|}{\rho^2}\bigg\}.
\end{aligned}
\end{align}
Taking $\eqref{period-perturb}_{3} \times  -\frac{\triangle{\zeta}}{\rho}$, we have
\begin{align}\label{eq0125}
\begin{aligned}
\p_\tau&\left[ \frac{|\nabla {\zeta}|^{2}}{2}\right]+\frac{\kappa({\T})  }{\rho}|\Delta{\zeta}|^{2}+\dv(\cdots)\\
&=O(1)\bigg\{|\nabla\psi||\nabla\zeta|^2+\theta_-|\nabla\psi||\Delta\zeta|+\theta_-^{\alpha-1}|\nabla\zeta|^2\frac{|\Delta\zeta|}{\rho}+\theta_-^{\alpha}|\nabla\psi|^2\frac{|\Delta\zeta|}{\rho}\bigg\}.
\end{aligned}
\end{align}
Combining \eqref{eq0124}-\eqref{eq0125} and integrating the resulted 
 over $\bar{\mathbb{T}}_\e^{3}
 $, we have
\begin{align}\label{eq0131}
\begin{aligned}
\frac{d}{d\tau}&\int_{\bar{\mathbb{T}}_\e^{3}}\frac{|\nabla\psi|^2}{2}+\frac{|\nabla\zeta|^2}{2}dy+\int_{\bar{\mathbb{T}}_\e^{3}}\frac{\T_-^{\alpha}  }{\rho}|\Delta{\zeta}|^{2}+\frac{\T_-^{\alpha}  }{\rho}|\nabla^2{\psi}|^{2}dy\\
&\leq  C\nu^{1-2\gamma}\Lambda^{-3}\theta_-^{-\alpha}
\bigg(\sup_{[0,\tau_1]}\|(\nabla\psi,\nabla\zeta)\|_{H^1}^2+1\bigg)\|(\theta_-^{\frac{\alpha-1}{2}}\nabla\zeta,\theta_-^{\frac{\alpha}{2}}\nabla\psi)\|^{2}\\
&\qquad+C\nu^{-12}\Lambda^{-6}\T_-^{-4\alpha}(\sup_{[0,\tau_1]}\|\nabla\phi\|_{H^1}^2+1)\|\T_-^{\frac{\alpha}{2}}\nabla\psi\|^2+\T_-^{-(\alpha+\beta)}\left\|\sqrt{\frac{\T_-^{\beta+1}}{\rho^2}}\nabla\phi\right\|^2.
\end{aligned}
\end{align}
Inserting \eqref{eq0131} into \eqref{naphi}, set $a_1=\max\{(6\gamma-2+(\gamma-1)(3\alpha+\beta))a,(12+4\gamma+(7\alpha+3\beta)(\gamma-1))a+6(b-1)\},$ and making use of \eqref{be},
\begin{equation}\label{naph}
\begin{aligned}
&\frac{d}{d\tau}\int_{\bar{\mathbb{T}}_\e^{3}}\frac{\Lambda(\T_-)}{2\rho^3}|\nabla\phi|^2dy
+\int_{\bar{\mathbb{T}}_\e^{3}}\frac{R\T_-^{\beta+1}}{\rho^2}|\nabla\phi|^2dy\\
\leq & C\varepsilon^{-a_1}\bigg(\sup_{[0,\tau_1]}\|(\nabla\phi,\nabla\psi,\nabla\zeta)\|_{H^1}^2+1\bigg)\|(\theta_-^{\frac{\alpha-1}{2}}\nabla\zeta,\theta_-^{\frac{\alpha}{2}}\nabla\psi)\|^2.
\end{aligned}
\end{equation}
And \eqref{eq0131} turns to 
\begin{equation}\label{eq0132}
\begin{aligned}
&\frac{d}{d\tau}\int_{\bar{\mathbb{T}}_\e^{3}}\frac{|\nabla\psi|^2}{2}+\frac{|\nabla\zeta|^2}{2}dy+\int_{\bar{\mathbb{T}}_\e^{3}}\frac{\T_-^{\alpha}  }{\rho}|\Delta{\zeta}|^{2}+\frac{\T_-^{\alpha}  }{\rho}|\nabla^2{\psi}|^{2}dy\\
\leq &C\varepsilon^{-a_2}\bigg(\sup_{[0,\tau_1]}\|(\nabla\phi,\nabla\psi,\nabla\zeta)\|_{H^1}^2+1\bigg)\|(\theta_-^{\frac{\alpha-1}{2}}\nabla\zeta,\theta_-^{\frac{\alpha}{2}}\nabla\psi)\|^2,
\end{aligned}
\end{equation}
where $a_2=a_1+(\alpha+\beta)a$.

{\bf Step  3.} 
 We turn to estimate $\|\nabla^2\phi\|^2.$ Similar as \eqref{eq-phiy1},
\begin{equation}\label{eq0139}
\begin{aligned}
\p_\tau&\left[\frac{\T_-^{\beta}}{\rho}\nabla\psi\cdot\nabla^2\phi+\frac{\Lambda(\T_-)}{{2}\rho^3}\left|\nabla^2\phi\right|^2\right]+\frac{R\T_-^{\beta}{\T}}{\rho^2}|\nabla^2\phi|^2+\dv(\cdots)\\
=&\frac{\T_-^{\beta}}{\rho}|(|\nabla\psi|^2,|\nabla^2\psi|,|\nabla^2\zeta|,\frac{|\nabla\zeta||\nabla\phi|}{\rho},\frac{\T}{\rho^2}|\nabla\phi|^2)||\nabla^2\phi|+\frac{\Lambda(\T_-)}{\rho^3}|\nabla^2\phi|^2|\nabla\psi|\\
&+\frac{\Lambda(\T_-)}{\rho^3}|\nabla\phi||\nabla^2\phi||\nabla^2\psi|+\frac{\T_-^{\beta}}{\rho^2}|\nabla^2\phi||\nabla\psi||\nabla\phi|+\T_-^{\beta}|\nabla^2\psi|^2+\frac{\T_-^{{\beta}}}{\rho}|\nabla^2\psi||\nabla\psi||\nabla\phi|\\
&+\frac{\T_-^{\beta}}{\rho^2}|\nabla\phi|^2|\nabla\psi|^2+\left|\frac{\T_-^{\alpha+\beta-1}}{\rho^2}\right||\nabla^2\psi||\nabla^2\phi||\nabla\zeta|+\left|\frac{\zeta\T_-^{\alpha+\beta-1}}{\rho^2}\right||(\nabla^2\dv\psi,\nabla\triangle\psi)||\nabla^2\phi|\\
&+\left|\frac{\T_-^{\alpha+\beta-1}}{\rho^3}\right||\nabla\psi||\nabla\phi||\nabla^2\phi||\nabla\zeta|+\left|\frac{\T_-^{\alpha+\beta-1}}{\rho^2}\right||\nabla^2\phi||\nabla^2\zeta||\nabla\psi|+\left|\frac{\T_-^{\alpha+\beta-2}}{\rho^2}\right||\nabla^2\phi||\nabla\zeta|^2|\nabla\psi|.
\end{aligned}
\end{equation}
Integrating \eqref{eq0139} over $\bar{\mathbb{T}}_\e^{3}
$, we have
\begin{equation}\label{n2phi}
\begin{aligned}
&\frac{d}{d\tau}\int_{\bar{\mathbb{T}}_\e^{3}}\frac{\Lambda(\T_-)}{2\rho^3}|\nabla^2\phi|^2+\frac{\T_-^{\beta}}{\rho}\nabla\psi\cdot\nabla^2\phi dy+\int_{\bar{\mathbb{T}}_\e^{3}}\frac{R\T_-^{\beta+1}}{\rho^2}|\nabla^2\phi|^2dy\\
\leq &\T_-^{(\alpha+\beta)}|\ln\varepsilon|^{-1}\|\T_-^{\frac{\alpha}{2}}\nabla^3\psi\|^2+C\nu^{-a_3}\bigg(\sup_{[0,\tau_1]}\|(\nabla\phi,\nabla\psi,\nabla\zeta)\|_{H^1}^2+1\bigg)\|(\theta_-^{\frac{\alpha-1}{2}}\nabla\zeta,\theta_-^{\frac{\alpha}{2}}\nabla\psi)\|^2\\
&+C\nu^{-a_4}\bigg(\sup_{[0,\tau_1]}\|(\nabla\phi,\nabla\psi,\nabla\zeta)\|_{H^1}^2+1\bigg)\|\T_-^{\frac{\beta+1}{2}}\nabla\phi\|^2\\
&+\nu^{-a_5}\bigg(\sup_{[0,\tau_1]}\|(\nabla\phi,\nabla\psi,\nabla\zeta)\|_{H^1}^2+1\bigg)\|(\theta_-^{\frac{\alpha-1}{2}}\nabla^2\zeta,\theta_-^{\frac{\alpha}{2}}\nabla^2\psi)\|^2.
\end{aligned}
\end{equation}

{We turn to estimate $\|(\nabla^2\psi,\nabla^2{\zeta})\|^2.$}
Similar as \eqref{eq0124}, we could get
\begin{align}\label{eq0140}
\begin{aligned}
&\p_\tau\left[ \frac{|\nabla^2 {\psi}|^{2}}{2}\right]+\frac{\mu({\T})  }{\rho}|\nabla\Delta{\psi}|^{2}+\frac{\mu({\T})+\lambda({\T})}{\rho}|\nabla^2\dv{\psi}|^{2}+\dv(\cdots)\\
=&O(1)\bigg\{\frac{|\nabla\zeta||\nabla\phi||\nabla^3\psi|}{\rho}+\frac{\T|\nabla\phi|^2}{\rho^2}|\nabla^3\psi|+\frac{\T}{\rho}|\nabla^2\phi||\nabla^3\psi|+|\nabla^2\zeta||\nabla^3\psi|\\
&+|\nabla\psi||\nabla^2\psi|^2+
|\nabla\psi|^2|\nabla^3\psi|+\frac{\theta^{\alpha-1}|\nabla\zeta||\nabla^2\psi||\nabla^3\psi|}{\rho}+\frac{\theta^{\alpha-2}|\nabla\zeta|^2|\nabla\psi||\nabla^3\psi|}{\rho}\\
&+\frac{\T^{\alpha-1}|\nabla^2\zeta||\nabla\psi||\nabla^3\psi|}{\rho}+\frac{|\nabla\phi|}{\rho^2}\big\{\T^{\alpha}|\nabla^2\psi||\nabla^3\psi|+\T^{\alpha-1}|\nabla\zeta||\nabla\psi||\nabla^3\psi|\big\}\bigg\}.
\end{aligned}
\end{align}
Similar as \eqref{eq0125}, we have 
\begin{align}\label{eq0141}
\begin{aligned}
&\p_\tau\left[ \frac{|\nabla^2 {\zeta}|^{2}}{2}\right]+\frac{\kappa({\T})  }{\rho}|\nabla\Delta{\zeta}|^{2}+\dv(\cdots)\\
=&O(1)\bigg\{|\nabla\zeta||\nabla\psi||\nabla\triangle\zeta|+\T|\nabla^2\psi||\nabla\Delta\zeta|+|\nabla\psi||\nabla^2\zeta|^2++\theta^{\alpha-1}\frac{|\nabla\zeta||\nabla^2\zeta||\nabla\triangle\zeta|}{\rho}
\\
&+\T^{\alpha-2}\frac{|\nabla\zeta|^3|\nabla\triangle\zeta|}{\rho}+\T^{\alpha-1}\frac{|\nabla\phi|}{\rho^2}|\nabla\zeta|^2|\nabla\triangle\zeta|+\T^{\alpha}\frac{|\nabla\phi|}{\rho^2}|\triangle\zeta||\nabla\triangle\zeta|\\
&+\T^{\alpha-1}\frac{|\nabla\zeta|}{\rho}|\nabla\psi|^2|\nabla\triangle\zeta|+\frac{\T^{\alpha}|\nabla\phi|}{\rho^2}|\nabla\psi|^2|\nabla\triangle\zeta |+\frac{\T^{\alpha}}{\rho}|\nabla\psi||\nabla^2\psi||\nabla\triangle\zeta|\bigg\}.
\end{aligned}
\end{align}
Integrating \eqref{eq0140}-\eqref{eq0141} over $\mathbb{T}_\e^{3}
$, we have
\begin{align}\label{n2pz}
\begin{aligned}
&\frac{d}{d\tau}\int_{\bar{\mathbb{T}}_\e^{3}}\frac{|\nabla^2\psi|^2}{2}+\frac{|\nabla^2\zeta|^2}{2} dy+C\int_{\bar{\mathbb{T}}_\e^{3}}\frac{\T_-^{\alpha}  }{\rho}|\nabla^3{\psi}|^{2}+\frac{\T_-^{\alpha}  }{\rho}|\nabla\Delta{\zeta}|^{2}dy\\
\leq&C\e^{-a_6}\bigg(\sup_{[0,\tau_1]}\|(\nabla\phi,\nabla\psi,\nabla\zeta)\|_{H^1}^2+1\bigg)\|(\theta_-^{\frac{\alpha-1}{2}}\nabla\zeta,\theta_-^{\frac{\alpha}{2}}\nabla\psi)\|^2\\
&+\e^{-a_7}\bigg(\sup_{[0,\tau_1]}\|(\nabla\phi,\nabla\psi,\nabla\zeta)\|_{H^1}^2+1\bigg)\|(\theta_-^{\frac{\alpha-1}{2}}\nabla^2\zeta,\theta_-^{\frac{\alpha}{2}}\nabla^2\psi)\|^2\\
&+\T_-^{-(\alpha+\beta)}\left\|\sqrt{\frac{\T^{\beta+1}}{\rho^2}}\nabla^2\phi\right\|^2.
\end{aligned}
\end{align}
Choosing $\beta=\alpha,$ under assumptions \eqref{prior}, we make use of \eqref{be},\eqref{naph},\eqref{eq0132}, \eqref{n2phi},\eqref{n2pz} to obtain
\begin{equation}
\begin{aligned}
&\frac{d}{d\tau}\int_{\bar\Torus_\e^3}\big({\rho}_-^{\gamma-2}\phi^2+{\rho}_-\psi^2+{\rho}_-^{2-\gamma}\zeta^2\big)dy+\int_{\bar\Torus_\e^3}\T_-^{\alpha}(|\nabla\psi|^2+\frac{|\nabla\zeta|^2}{\theta})dy
+\frac{d}{d\tau}\e^{a_1}\int_{\bar\Torus_\e^3}\frac{\T_-^{2\alpha}}{2\rho^3}|\nabla\phi|^2\\
&+\T_-^{2\alpha}\bigg(|\nabla\psi|^2+|\nabla\zeta|^2\bigg)dy+\e^{a_1}\bigg(\int_{\bar{\mathbb{T}}_\e^{3}}\frac{R\T_-^{\alpha+1}}{\rho^2}|\nabla\phi|^2dy+\T_-^{2\alpha}\int_{\bar\Torus_\e^3}\frac{\T_-^{\alpha}  }{\rho}|\Delta{\zeta}|^{2}+\frac{\T_-^{\alpha}  }{\rho}|\nabla^2{\psi}|^{2}dy\bigg)\\
&+\frac{d}{d\tau}\e^{a_2}\int_{\bar{\mathbb{T}}_\e^{3}}\frac{\T_-^{2\alpha}}{2\rho^3}|\nabla^2\phi|^2+\T_-^{2\alpha}\bigg(|\nabla^2\psi|^2+|\nabla^2\zeta|^2\bigg)dy+\e^{a_2}\bigg(\int_{\bar{\mathbb{T}}_\e^{3}}\frac{R\T_-^{\alpha+1}}{\rho^2}|\nabla^2\phi|^2dy\\
&+\T_-^{2\alpha}\int_{\bar{\mathbb{T}}_\e^{3}}\frac{\T_-^{\alpha}  }{\rho}|\nabla^3{\psi}|^{2}+\frac{\T_-^{\alpha}  }{\rho}|\nabla\Delta{\zeta}|^{2}dy\bigg)\leq 0,
\end{aligned}
\end{equation}
{where $a_1,a_2$ 
 are positive constants which depending on $\alpha,\beta,a,b,\gamma$.}

The Poincar
{\'e}'s inequality gives that 
\begin{equation}
\begin{aligned}
\|(\nabla^{k}\phi,\nabla^{k}W,\nabla^{k}Z)\|_{L^2(\bar{\mathbb{T}}_\e^{3})}\leq C\Lambda \|(\nabla^{k+1}\phi,\nabla^{k+1}W,\nabla^{k+1} Z)\|_{L^2(\bar{\mathbb{T}}_\e^{3})},\ \ \Lambda=\varepsilon^{b-1},\ \  k=0,1.
\end{aligned}
\end{equation}
Note that $$\|W\|_{L^2}\leq \|(\phi,\rho\psi)\|_{L^2},\|Z\|_{L^2}\leq \|(\phi,\rho\psi,\rho\zeta)\|_{L^2}.$$ 
 Set $E_1(\tau)=\|(\rho_-^{\frac{\gamma-2}{2}}\phi,{\rho}_-^{\frac{1}{2}}\psi,{\rho}_-^{\frac{2-\gamma}{2}}\zeta)(\tau)\|^2$ and choosing $n_0$ suitably large such that 
\begin{equation}
\begin{aligned}
&\frac{d}{d\tau}E_1(\tau)+\Lambda^{-1}\bigg\{\frac{\e^{a_1}\T_-^{\alpha+1}}{\rho_-^2}\|(\phi,W,Z)(\tau)\|^2\bigg\}\leq 0,\\
 &\frac{d}{d\tau}\|(\phi,W,Z)(\tau)\|^2+\rho_-^{\gamma}\Lambda^{-1}\frac{\e^{a_1}\T_-^{\alpha+1}}{\rho_-^2}\|(\phi,W,Z)(\tau)\|^2\leq 0,\\
 &\Rightarrow \|(\phi,W,Z)(\tau)\|_{L^2(\Omega)}
 \leq \|(\phi,W,Z)(0)\|_{L^2(\Omega)}e^{-\varepsilon^{n_0}\tau},\ \ \e^{n_0}\leq \rho_-^{\gamma-2}\Lambda^{-1}\e^{a_1}\T_-^{\alpha+1}.
\end{aligned}
\end{equation}
Similarly, we can also do the higher energy estimations for $\|\nabla^{j}\phi,\nabla^{j}\psi,\nabla^{j}\zeta\|^2$ for $j\geq 3,$ and we could get 
\begin{align}
\|(\phi,W,Z)(\tau)\|_{H^2(\Omega)}\leq \|(\phi,W,Z)(0)\|_{H^2(\Omega)} e^{-\varepsilon^{n_0}\tau}.
\end{align}
By using G-N inequality again, we have proved \eqref{eq096} for $k=0.$ The case $k\geq 1$ can be obtained similarly. 
\end{proof}

\

\textbf{Acknowledgements} \ 
The authors are grateful for the many useful discussions and suggestions of Professor Feimin Huang. The last author would like to express his gratitude to Ms. Yuchen Lu for her kind assistance and inspiration.

\

\textbf{Funding} \ 
 The work is supported by the NSFC (11831003, 12171111, 12371215) of China and the SFC (KZ202110005011) of Beijing.

\ 

\textbf{Competing interests} \ 
The authors do not have any other competing interests to declare.

\

\textbf{Data availability} \ 
This article does not have any external supporting data.

\

\textbf{Publisher's Note} \ 
Springer Nature remains neutral with regard to jurisdictional claims in published maps and institutional affiliations.

		\vspace{1.5cm}
		
\end{document}